%% file: eigfBmV3da.tex
\documentclass[reqno,11pt]{amsart}
\usepackage{latexsym,amssymb,amsfonts,amsmath,amsthm}
\usepackage{eucal}
\usepackage{mathtools}
\usepackage[dvipsnames,usenames]{xcolor}
\usepackage[colorlinks,citecolor=blue,urlcolor=blue]{hyperref}

\usepackage{tikz}
\usetikzlibrary{decorations.markings}
\usepackage{pgfplots}

\newtheorem{thm}{Theorem}[section]

\newtheorem{lem}[thm]{Lemma}
\newtheorem{prop}[thm]{Proposition}
\theoremstyle{definition}

\theoremstyle{remark}
\newtheorem{rem}[thm]{Remark}
\numberwithin{equation}{section}
\newcommand{\Real}{\mathbb R}
\newcommand{\eps}{\varepsilon}
\newcommand{\one}[1]{\mathbf{1}_{\{#1\}}}
\renewcommand{\P}{\mathbb{P}}
\newcommand{\E}{\mathbb{E}}
\newcommand{\F}{\mathcal{F}}
\newcommand{\sign}{\mathrm{sign}}

\newcommand{\Res}{\mathrm{Res}}
\renewcommand{\Re}{\mathrm{Re}}
\renewcommand{\Im}{\mathrm{Im}}

\def\Xint#1{\mathchoice
   {\XXint\displaystyle\textstyle{#1}}%
   {\XXint\textstyle\scriptstyle{#1}}%
   {\XXint\scriptstyle\scriptscriptstyle{#1}}%
   {\XXint\scriptscriptstyle\scriptscriptstyle{#1}}%
   \!\int}
\def\XXint#1#2#3{{\setbox0=\hbox{$#1{#2#3}{\int}$}
     \vcenter{\hbox{$#2#3$}}\kern-.5\wd0}}

\def\dashint{\Xint-}

\begin{document}

\title[Eigenproblems for fractional covariance operators]
{Exact asymptotics in eigenproblems for fractional Brownian covariance operators}

\author{P. Chigansky}%
\address{Department of Statistics,
The Hebrew University,
Mount Scopus, Jerusalem 91905,
Israel}
\email{pchiga@mscc.huji.ac.il}

\author{M. Kleptsyna}%
\address{Laboratoire de Statistique et Processus,
Universite du Maine,
France}
\email{marina.kleptsyna@univ-lemans.fr}

\thanks{P. Chigansky is supported by ISF 558/13
grant}

\date{\today}%
\begin{abstract}

Many results in the theory of Gaussian processes rely on the eigenstructure of the covariance operator.
However, eigenproblems are notoriously hard to solve explicitly and closed form solutions are known only 
in a limited number of cases. 
In this paper we set up a framework for the spectral analysis of the fractional type covariance operators, 
corresponding to an important family of processes, which includes the fractional Brownian motion and its noise.
We obtain accurate asymptotic approximations for the eigenvalues and the eigenfunctions. Our results provide a key to 
several problems, whose solution is long known in the standard Brownian case, but was missing in the more general 
fractional setting. This includes computation of the exact limits of $L^2$-small ball probabilities and asymptotic analysis of
singularly perturbed integral equations, arising in mathematical physics and applied probability.

\end{abstract}

\maketitle


\section{Introduction}

The eigenproblem for a centered random process $X=(X_t; t\in [0,1])$ with covariance function
$R(s,t)=\E X_tX_s$ and the corresponding covariance operator
$$
(Kf)(t) := \int_0^1 R(s,t)f(s)ds, \quad t\in [0,1]
$$
consists of finding all nontrivial pairs $(\lambda, \varphi)$ satisfying the equation 
\begin{equation}
\label{eigpr}
K \varphi = \lambda \varphi.
\end{equation}

For a self adjoint positive definite operator $K$, compact in $L^2(0,1)$, this problem is well known to have countably
many solutions: the eigenvalues $\lambda_n$ are nonnegative and converge to zero, when put in the decreasing order, and the 
normalized eigenfunctions $\varphi_n$ form an orthonormal basis in $L^2(0,1)$.
This fact has numerous applications in stochastic processes: the Karhunen–-Lo\`{e}ve series expansion \cite{BTA04},
equivalence and orthogonality of Gaussian measures \cite{Sh66}, asymptotics of the small ball probabilities \cite{LiS01}, 
sampling from heavy tailed distributions \cite{VT13}, non-central limit theorems \cite{LRMT17} are only a few problems to mention. 
However the eigenvalues and eigenfunctions are notoriously hard to find explicitly.

One notable exception is the standard Brownian motion $B=(B_t;t\in [0,1])$  for which
\begin{equation}
\label{Bmlambda}
\lambda_n = \frac 1 {(n-\frac 1 2)^2\pi^2} \quad \text{and}\quad \varphi_n(t) = \sqrt 2 \sin \left(n-\tfrac 1 2 \right)\pi t,
\quad n=1,2,...
\end{equation}
These well known formulas are easily found by reducing \eqref{eigpr} to a simple boundary value problem for an ordinary differential equation.
Similar reduction often works for other processes, derived from the Brownian motion, such as the Brownian bridge,
the Ornstein–-Uhlenbeck process, etc.  In essence it puts the original eigenproblem into the framework of Sturm–-Liouville
type theory for differential operators (see \cite{NN04ptrf}).

This approach  does not apply to covariance operators with a more complicated structure, including processes
with long range dependence. An important example is the {\em fractional} Brownian motion $B^H=(B^H_t; t\in [0,1])$, $H\in (0,1)$, that is,
the centered Gaussian process with covariance function
\begin{equation}
\label{RfBm}
R(s,t) = \frac 1 2 \Big(t^{2H}+s^{2H}-|t-s|^{2H}\Big), \quad s,t\in [0,1].
\end{equation}
The parameter $H$ is the Hurst exponent of $B^H$ and the standard Brownian motion corresponds to $H=\frac 1 2$.
The fBm is a useful and interesting process, which has been extensively studied since its introduction
in \cite{MvN68}. It is the only self-similar Gaussian process with exponent $H$, whose increments are stationary.
For $H\ne \frac 1 2$ it is neither a semimartingale nor a Markov process and
for $H>\frac 12$ its increments are positively correlated and have  long range dependence
$$
\sum_{n=1}^\infty \E B^H_1 (B^H_n-B^H_{n-1}) = \infty.
$$
The diversity of properties makes the fBm useful in modeling (see e.g. \cite{BLL14}), various aspects of the related theory and 
applications can be found in \cite{EM02}, \cite{BHOZ},  \cite{M08}, \cite{Nu06}, \cite{Ta03}.

The two main operators of interest in the fractional setting are
\begin{align}
\label{KfBm}
(K f)(t) & = \int_0^1 \frac 1 2 \big(t^{2H}+s^{2H}-|t-s|^{2H}\big) f(s)ds \\
\label{KfBn}
(\widetilde K f)(t) & =\frac {d}{dt} \int_0^1  H |t-s|^{2H-1} \sign(t-s)f(s) ds.
\end{align}
The first one is the covariance operator of the fBm itself and
the second operator corresponds to the fractional Brownian {\em noise}, that is, the formal derivative of the fBm.
More precisely, $\widetilde K$ determines the correlation structure of stochastic integrals 
of deterministic functions through the formula
$$
\E \int_0^1 f dB^H \int_0^1 g dB^H = \langle f, \widetilde K g\rangle,
$$
which makes it a useful tool in stochastic analysis.

For both operators the eigenproblem does not admit a closed form solution and turns out to be more
complicated than its standard Brownian counterpart; therefore accurate approximations are of considerable
interest. Finding asymptotics of the ordered sequence of eigenvalues is a classical theme in functional
analysis and the exact expression for the leading asymptotic term can be often derived by a number of available techniques.
Further refinement is more specific to particular structure of the kernel and is usually much harder to obtain.
Even less common are results on asymptotic approximation of the corresponding eigenfunctions.
This is hardly surprising, since, as we will see, the latter requires sharp estimates on the residual beyond the first order term
in the eigenvalues asymptotics.

For the operator $K$ only the first order asymptotics of the eigenvalues is known for all values of $H\in (0,1)$, see \eqref{Bron}
below. 
The operator $\widetilde K$ has been extensively studied for $H>\frac 1 2$, in which case the derivative and
integration in \eqref{KfBn} are interchangeable. Consequently $\widetilde K$ reduces to the integral operator
with the weakly singular kernel
\begin{equation}
\label{KfBn2}
(\widetilde K f)(t) = \int_0^1 H(2H-1) |s-t|^{2H-2}f(s)ds.
\end{equation}
For such operators precise eigenvalues asymptotics is known up to the second order, see \eqref{lambdaU}.
For $H<\frac 1 2$ operator $\widetilde K$ does not admit an integral form and its spectral properties
have never been studied before.
In fact it is not entirely obvious at the outset that such eigenproblem has (countably many) solutions.
This is nevertheless the case, since the inverse $\widetilde K^{-1}$ turns out to be an integral operator with a weakly
singular kernel (see \eqref{kappa} below).

In this paper we show how the eigenproblem \eqref{eigpr} for the operators defined in \eqref{KfBm} and \eqref{KfBn}
can be reduced to an equivalent integro-algebraic system of equations, which turns out to be more amenable to 
asymptotic analysis.  The method is based on the technique, used for solving the Riemann boundary value problems, 
rather than the more common Sturm–-Liouville type theory, mentioned above.

\medskip 

Specifically, for both operators and all values of $H\in (0,1)$

\medskip

\begin{enumerate}
\addtolength{\itemsep}{0.7\baselineskip}
\renewcommand{\theenumi}{\alph{enumi}}

\item  we derive asymptotic approximation for the eigenfunctions with respect to the uniform norm

\item  we improve previously known asymptotics of the eigenvalues up to the second order term

\end{enumerate}

\medskip

Our results  can be useful in a variety of applications. To demonstrate the ideas, we consider in this paper 
the following two problems:

\medskip

\begin{enumerate}
\addtolength{\itemsep}{0.7\baselineskip}
\renewcommand{\theenumi}{\roman{enumi}}

\item\label{Pi} refinement of the {\em exact} asymptotics of $L^2$-small ball probabilities for the fractional Brownian motion,
which was previously known only on the logarithmic scale (\cite{Br03a}, \cite{NN04tpa}, \cite{LP04}).

\item\label{Pii} analysis of singularly perturbed integral equations, arising in mathematical physics \cite{Ukai} and 
applied probability (e.g., stochastic analysis \cite{Ch03b}, \cite{CCK}, statistical inference \cite{CKstat}).

\end{enumerate}

\medskip

The key to \eqref{Pi} is the exact formula for the second order asymptotic term of the fBm eigenvalues;
problem \eqref{Pii} requires the exact asymptotics of the scalar products $\langle h, \varphi_n \rangle$
for certain functions $h$, which becomes available through the approximation for the eigenfunctions.

\medskip

The analysis framework, set up in this paper, is applicable with some nontrivial adjustments to processes, related to 
the fBm, such as the corresponding bridge and the Ornstein--Uhlenbeck process, integrated fractional Brownian motion, etc.
The results in this direction will be reported in the forthcoming work \cite{ChKM1,ChKM2}.  

\subsection{Frequently used notations} For numerical sequences $a_n$ and $b_n$, we write $a_n\propto b_n$ if 
$a_n=c b_n$ with a constant $c\ne 0$ and $a_n \sim b_n$ and $a_n \simeq b_n$ if 
$a_n \propto b_n (1+o(1))$ and $a_n =  b_n (1+o(1))$ respectively as $n\to\infty$. 
Similarly, $f(x)\propto g(x)$ stands for the equality $f(x)=c g(x)$ with a constant $c\ne 0$, etc.

Our main reference for the boundary value problems is the text \cite{Gahov}, where the particular form of the 
Riemann problem used below is detailed in \S 43. Another reference on the subject is the classic book \cite{M46}. 
Unless stated otherwise, standard principle branches of the multivalued complex functions will be used.
We will frequently work with functions, sectionally holomorphic on the complex plane cut along the real line.
For such a function $\Psi$, we denote by $\Psi^+(t)$ and $\Psi^-(t)$ the limits of $\Psi(z)$ as $z$ approaches the point $t$
on the real line from above and below respectively:
$$
\Psi^\pm (t) := \lim_{z\to t^\pm }\Psi(z).
$$

One of our main tools in what follows is the Sokhotski--Plemelj formula, which states that the 
Cauchy-type integral 
$$
\Phi(z):= \frac 1{2\pi i}\int_0^\infty \frac{\phi(\tau)}{\tau-z}d\tau
$$
of a H\"older continuous function $\phi$ on $\Real_{>0}$, possibly with integrable singularities at the origin and infinity, 
defines a function, analytic on the cut plane  $\mathbb{C}\setminus \mathbb{R}_{\ge 0}$, whose limits across the real axis satisfy 
$$
\Phi^\pm (t) = \frac 1{2\pi i} \dashint_0^\infty \frac{\phi(\tau)}{\tau-t}d\tau \pm  \frac 1 2 \phi(t), \quad t\in \Real_{>0},
$$
where the dash integral stands for the Cauchy principle value. In particular, $\Phi(z)$ is a solution of the 
Riemann boundary value problem $\Phi^+(t) -\Phi^-(t) = \phi(t)$, $t\in \Real_{>0}$, which vanishes at infinity. 
This fact is the main building block in all boundary problems to be encountered in this paper. 

\section{The main results}

\subsection{The fractional Brownian motion} 

The exact first order asymptotics for the eigenvalues of the fBm covariance operator $K$ defined in \eqref{KfBm} was found in \cite{Br03a, Br03b},
where it is shown that for any $\delta>0$
\begin{equation}
\label{Bron}
\lambda_n = \frac{\sin (\pi H)\Gamma(2H+1)}{(n\pi)^{2H+1}} + o \left(n^{-\frac{(2H+2)(4H+3)}{4H+5}+\delta}\right) \quad   \text{as}\ n\to\infty,
\end{equation}
with $\Gamma(\cdot)$ being the standard gamma function. The leading asymptotic term in \eqref{Bron} was also obtained in \cite{NN04tpa} and 
\cite{LP04} using different methods. To the best of our knowledge, nothing was known previously about the eigenfunctions.

\medskip

The following theorem reveals a more detailed asymptotic structure of the solutions to the eigenproblem \eqref{eigpr} with 
the covariance operator \eqref{KfBm}: 

\begin{thm}\label{main-thm-fbm}\

\medskip

\noindent
{\bf 1.} For $H\in (0,1)$ the eigenvalues are given by the formula 
\begin{equation}
\label{lambda}
\lambda_n =  \sin (\pi H)\Gamma(2H+1) \nu_n^{-2H-1} \qquad n=1,2,...
\end{equation}
where the sequence $\nu_n$ satisfies 
\begin{equation}
\label{nun}
\nu_n =  \Big(n -\frac 1 2\Big)\pi + \frac {1-2H}{ 4}  \pi  + \arcsin \frac{\ell_H}{\sqrt{1+\ell_H^2}}+O(n^{-1}) \quad \text{as}\ n\to\infty,
\end{equation}
with the constant 
$
\displaystyle
\ell_H:=\frac
{
\sin \frac{\pi}{2} \frac{ H-1/2}{H+1/2}
}
{\sin \frac \pi 2 \frac{1}{H+1/2}}
$

\medskip
\noindent
{\bf 2.} The corresponding normalized eigenfunctions admit the approximation 
\begin{multline}
\label{phin}
\varphi_n(x)  
= 
 \sqrt 2 \sin\bigg( \nu_{n} x+\frac {2H-1}{ 8}  \pi-\arcsin \frac{\ell_H}{\sqrt{1+\ell_H^2}}\bigg) \\
 + \frac {\sqrt{2H+1}} { \pi }  \int_0^{\infty}    \rho_0(u)
\bigg(
e^{-  x \nu_n u} \frac{ u-\ell_H}{ \sqrt{1+\ell_H^2}}-
(-1)^{n}   e^{-  (1-x) \nu_nu}\
\bigg)du + \nu_n^{-1}  r_n(x),
\end{multline}
where the residual $r_n(x)$ is bounded by a constant, depending only on $H$, and
$\rho_0(u)$ is the explicit function, defined in \eqref{rhofla} below.

\medskip
\noindent
{\bf 3.}  The eigenfunctions satisfy
\begin{equation}
\label{bndp}
\varphi_n(1) = -(-1)^{n} \sqrt{2H+1}   \big(1+O(n^{-1})\big)
\end{equation}
and their averages are given by
\begin{equation}
\label{ave}
\int_0^1 \varphi_n(x)dx =  -\sqrt{\frac{2H+1}{1+\ell_H^2}}\; \nu_n^{-1}.
\end{equation}

\end{thm}

\begin{rem}

\
\medskip

\noindent
a) The expression \eqref{lambda} gives the asymptotics of $\lambda_n$, precise up to the second order term:
$$
\lambda_n =  \frac{\sin (\pi H)\Gamma(2H+1)}{(n\pi)^{2H+1}}\Big(1 -C_H n^{-1}+ O(n^{-2})\Big)
\quad   \text{as}\ n\to\infty,
$$
with the constant 
$$
C_H = (2H+2) \left(-\frac 1 2  + \frac {1-2H}{ 4}     + \frac 1 \pi \arcsin \frac{\ell_H}{\sqrt{1+\ell_H^2}}\right).
$$
In particular, the estimate of the residual in \eqref{Bron} is not sharp. 

It is possible to derive an absolute numerical bound for the $O(n^{-1})$ term in \eqref{nun} and a  rough 
bookkeeping of the constants in the proof shows that it does not exceed $100 n^{-1}$. The left plot of  
Figure \ref{fig3} gives an idea of its actual magnitude: the difference  
$$
\left(\frac{\sin (\pi H)\Gamma(2H+1)}{\widetilde{\lambda}_n}\right)^{\frac 1 {2H+1}} -   \widehat \nu_n 
$$
is depicted here versus $n=1,...,40$, where $\widehat \nu_n$ is the expression on the right hand side of \eqref{nun} 
without the $O(n^{-1})$ term and $\widetilde\lambda_{n}$ is a sufficiently accurate numerical approximation of $\lambda_n$, produced by  
the standard Nystrom method. It should be stressed that the numerical precision deteriorates rapidly with $n$ and becomes unreliable 
for our purposes already for $n\ge 50$. This provides a practical motivation for exact asymptotic expansion, such as the one  
obtained in this paper.  

The plot on the right of Figure \ref{fig3} depicts the relative errors of the first and second order approximations:
$$
\frac{\widetilde \lambda_n - \widehat \lambda_n^{(j)}}{\widetilde \lambda_n}\times 100 \% \quad j=1,2
$$
where $\widehat \lambda_n^{(1)}$ and $\widehat \lambda_n^{(2)}$ are the estimates, obtained by the formula \eqref{lambda},
with $\nu_n$ from \eqref{nun}, replaced by one and two of its leading terms respectively. Observe how drastic is the accuracy 
improvement due to the second order correction!

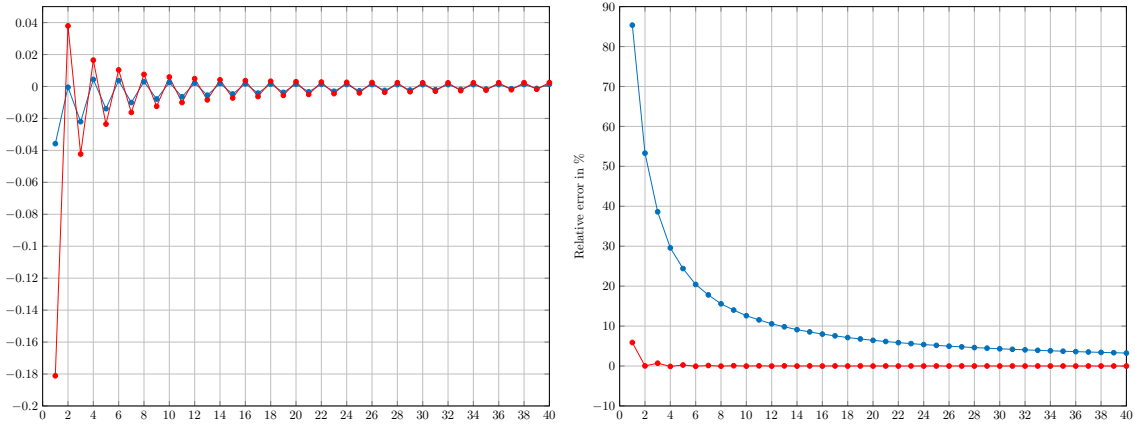
\begin{figure}[ht]
\input{nu_err.tex}
\input{relative_error.tex}
\caption{\label{fig3} 
Left: numerical evaluation of the residual in \eqref{nun}: $H=3/4$ (blue) and $H=1/4$ (red).
Right: the relative errors of the first (blue) and second (red) order approximations, $H=3/4$.
}
\end{figure}

\begin{figure}[ht]
\input{phi25gridon.tex}
\input{phi75gridon.tex}
  \caption{\label{fig1} Typical eigenfunctions for the fBm: $\varphi_{10}$ and $\varphi_{11}$ are depicted here in red and blue respectively for $H=\frac 1 4$ (left) and
  $H=\frac 3 4$ (right)}
\end{figure}
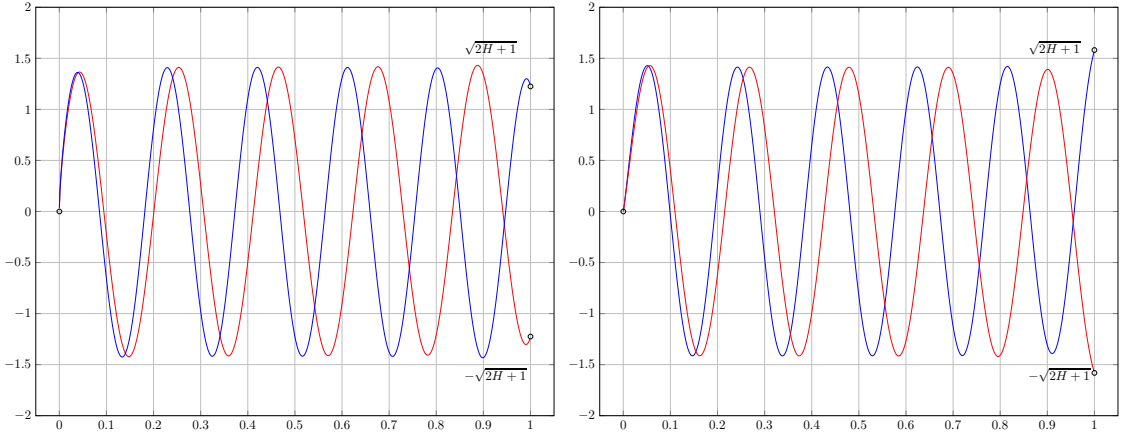

\medskip
\noindent
b) The formulas \eqref{bndp} and \eqref{ave} are obtained as an integral part of the proof, rather than being derived  
directly from the approximation \eqref{phin}. It is possible to obtain qualitative information on the structure of the residual 
$r_n(x)$ and, in fact, to deduce further asymptotic terms for the eigenfunctions. This would give an alternative way of finding 
the asymptotics of scalar products with various functions of interest.  
Results in this direction will be reported elsewhere.

\medskip

\noindent
c) The integral term in \eqref{phin} introduces a boundary layer: for large $\nu_n$ its contribution is negligible with respect to
the oscillatory term in the interior of the interval $[0,1]$. At the boundaries $x=0$ and $x=1$, it is persistent, forcing the values 
$\varphi_n(0)=0$ and $\varphi_n(1) \simeq (-1)^{n-1} \sqrt{2H+1}$.
Thus $\varphi_n$ appears as a shifted sinusoidal, slightly perturbed at the boundaries (see Figure \ref{fig1}).
Note, however, that the contribution of the boundary layer may not be asymptotically negligible on the level of scalar 
products $\langle h, \varphi_n\rangle$, being comparable to that of the oscillatory term.
For $H=\frac 1 2$ the second term in \eqref{nun}, the boundary layer and the shift in the oscillatory term in
\eqref{phin} all vanish, recovering the exact expressions \eqref{Bmlambda} up to $O(n^{-1})$ residual.

\medskip

\noindent
d) A simple calculation shows that adding a constant $\sigma^2>0$ to the covariance function of the standard Brownian motion, that is, 
starting it from a random initial condition with variance $\sigma^2$, introduces a shift of $-\pi/ 2$ 
in both the leading term of the eigenfunctions and in the second order term of the eigenvalues, c.f. \eqref{Bmlambda}: 
$$
\lambda_n = \frac 1 {\nu_n^2} \quad \text{and}\quad \varphi_n(x) = \sqrt 2 \cos (\nu_n x) +  \frac{1}{\sigma^2 \nu_n}\sqrt{2}\sin (\nu_n x)
\quad n=1,2,...
$$
where the sequence $\nu_n$ satisfies 
$$
\nu_n = \Big(n-\frac 1 2\Big) \pi -\frac\pi 2 + O(n^{-1}),\quad n\to\infty.
$$

It can be seen from our proof, that a similar phenomenon occurs in the fractional case: the second order constant term in \eqref{nun} 
changes so that 
$$
\nu_n =  \Big(n -\frac 1 2\Big)\pi + \frac {1-2H}{ 4}  \pi  -\frac \pi 2 +O(n^{-1}) \quad \text{as}\ n\to\infty.
$$
The same quantity is added to the phase of the oscillatory term in \eqref{phin}. Moreover, the values of the eigenfunctions at both 
endpoints of the interval approach $\pm (-1)^n \sqrt{2H+1}$ and their averages vanish at a much faster rate $\nu_n^{-1-2H}$ than in 
\eqref{ave}.

\medskip 

\noindent
e) For $H=1$ the fBm degenerates to the linear drift process $B^1_t = \xi t$ with  $N(0,1)$ random variable $\xi$. The eigenproblem \eqref{eigpr} in this case admits a
simple positive eigenvalue $\lambda_1=1/3$ with the corresponding eigenfunction $\varphi_1(t)=\sqrt{3} t$. The rest of the eigenvalues vanish and an arbitrary orthogonal
basic can be chosen to span the corresponding infinite dimensional eigenspace.
Note that formally this  also fits the asymptotic formula \eqref{lambda}. 

\medskip

\noindent
f) An area of research, somewhat related to the subject of this paper, is approximation of stochastic processes by series of 
deterministic functions. For a given norm $\|\cdot\|$ and a Gaussian process $X=(X_t, t\in [0,1])$, define the approximation error  
$$
\ell_n(X,\|\cdot\|):= \inf \left\{\left(\E \Big\|\sum_{k=n}^\infty \xi_k x_k\Big\|^2\right)^{1/2}: X  = \sum_{k=1}^\infty \xi_k x_k\ \text{a.s.}\right\}
$$
where $\xi_k$'s are i.i.d. $N(0,1)$ random variables and the infimum is taken over all sequences of non-random functions $(x_k)$. 
It is required to estimate the decay rate of the error as $n\to\infty$ and to identify a sequence, which attains the optimal rate. 
Such series expansions are useful for computer simulations of stochastic processes. 

Since the eigenfunctions form the basis which diagonalizes the covariance operator, the solution for this problem with 
the Euclidean norm $\|\cdot\|_2$ is provided by the Karhunen–-Lo\`{e}ve expansion. The corresponding optimal rate is 
$$
\ell_n(B^H,\|\cdot\|_2) \sim n^{-H} \quad \text{as\ } n\to\infty.
$$
The problem becomes more intricate for the uniform norm $\|\cdot\|_\infty$ and it was shown in \cite{KL02} that in this case 
the optimal rate for the fBm is
$$
\ell_n(B^H, \|\cdot\|_\infty) \sim n^{-H}\sqrt{\log n}\quad \text{as\ } n\to\infty.
$$
Several optimal approximating sequences were found in \cite{AT03}, \cite{DzhVZ04}, \cite{I05}.

\medskip 
\noindent 
g) The results of Theorem \ref{main-thm-fbm}, of course, apply to non-Gaussian processes with the same covariance function. 
One important family consists of the Hermite processes (see, e.g.  \cite{MT07}, for details). In this case, the coefficients 
in the Karhunen–-Lo\`{e}ve expansion are non-Gaussian and, still being orthogonal, are no longer independent in general.

\end{rem}

\subsection{The fractional Brownian noise}

As mentioned above, for $H>\frac 12$ the covariance operator \eqref{KfBn} reduces to the integral operator \eqref{KfBn2}.
The exact first order asymptotics of the eigenvalues in this case was found in \cite{Kac55}, \cite{Ros63}, \cite{W63}, \cite{BS70} and
improved in \cite{Ukai} (see also \cite{P74, P03}, \cite{D98}) up to the second order term:
\begin{equation}
\label{lambdaU}
\lambda_n =
\sin (\pi H)\Gamma(2H+1)\left( \left(n -\tfrac 1 2\right)\pi + \frac {1-2H}{ 4}  \pi  + O(n^{-1})\right)^{1-2H}
\quad \text{as\ \ }n\to \infty.
\end{equation}
No results have been reported so far regarding the properties of the corresponding eigenfunctions. 

\medskip

The following theorem asserts that formula \eqref{lambdaU} remains valid for $H<\frac 1 2$ as well and gives an accurate approximation
for the eigenfunctions for all $H\in (0,1)\setminus\{\frac 1 2\}$:

\medskip

\begin{thm}\label{main-thm-fbn}
\

\medskip
\noindent
{\bf 1.} For $H\in (0,1)\setminus\{\frac 1 2\}$
$$
\lambda_n  = 
\sin (\pi H)\Gamma(2H+1) \nu_n^{1-2H},\quad n=1,2,...
$$
where $\nu_n$ satisfies (c.f. \eqref{nun})
$$
\nu_n =  \Big(n -\frac 1 2\Big)\pi + \frac {1-2H}{ 4}  \pi  +O(n^{-1}) \quad \text{as}\ n\to\infty.
$$

\medskip
\noindent
{\bf 2.}
The corresponding normalized eigenfunctions  satisfy
\begin{multline}\label{phin1}
\varphi_n(x)
= \sqrt 2 \sin\left(\nu_n  x      +  \frac {2H+1} 8 \pi \right) \\
+\frac {\sqrt{|2H-1|}} \pi \int_0^\infty \rho_0(u)
\Big(
e^{-x\nu_n  u}-(-1)^n e^{-(1-x)\nu_n u}
\Big)  du
    + n^{-1} r_n(x),
\end{multline}
where the residual $r_n(x)$ is bounded by a constant, depending only on $H$, and $\rho_0(u)$ is the explicit function,
defined in \eqref{rho0fBn}.

\medskip
\noindent
{\bf 3.} The eigenfunctions with odd (even) indices are (anti)symmetric around the midpoint of the interval:
$$
\varphi_n(x)=(-1)^{n+1} \varphi_n(1-x), \quad n=1,2,...
$$
The averages of the symmetric eigenfunctions satisfy
\begin{equation}
\label{phinave}
\langle 1, \varphi_{2n-1}\rangle  \sim   \nu_{2n-1}^{ -\max\big(1,\frac 1 2 +H  \big)  } \qquad  \text{as}\ n\to\infty.
\end{equation}
\end{thm}

\medskip

\begin{rem}\

\medskip
\noindent
a) Note that for $H<\frac 1 2$ the sequence of eigenvalues increases to $+\infty$, in agreement with noncompactness of 
$\widetilde K$.

\medskip
\noindent
b) The expression \eqref{phin1} consists of the oscillatory and boundary layer terms (cf. \eqref{phin}).

\medskip
\noindent
c) It can be seen from the proof, that for $H>\frac 1 2$ the values of the eigenfunctions at the endpoints of the 
interval converge to nonzero constants. For $H<\frac 1 2$, the explicit formula \eqref{kappa} for the kernel of the 
inverse operator $\widetilde K^{-1}$ implies that in this case, the eigenfunctions vanish at the endpoints.

\end{rem}

\section{Some applications}\label{sec-3}

The asymptotic formulas from Theorems \ref{main-thm-fbm} and \ref{main-thm-fbn} can be useful in various problems of
applied probability,  statistics and mathematical physics.
Below we discuss a number of such applications.

\subsection{Small $\mathbf{L_2}$-ball probabilities}
Given a Gaussian process $X$ and a norm $\|\cdot\|$,
the small ball probability problem consists of computing the asymptotics of  $\P(\|X\|\le \eps)$ as $\eps\to 0$.
This problem inspired much research, some outcomes of which can be traced back in
the excellent survey \cite{LiS01}. For the $L^2$-norm
$$
\|X\|^2_2 = \int_0^1 X(t)^2 dt = \sum_{n=1}^\infty \lambda_n \xi_n^2,
$$
where $\xi_n$ are i.i.d. $N(0,1)$ random variables, and
the complete solution in this case was found in \cite{S74} in terms of the eigenvalues $\lambda_n$.
Computation of the precise asymptotics for particular processes is possible if a sufficiently detailed asymptotic
behavior of the eigenvalues is known.
Roughly speaking, the exact first order asymptotic term of $\lambda_n$ suffices for the {\em logarithmic} asymptotics of $\log \P(\|X\|\le \eps)$
and the exact second term gives the asymptotics of  $\P(\|X\|\le \eps)$ itself, usually precise up to the so called {\em distortion} constant
\begin{equation}
\label{Cd}
C_d = \prod_{n=1}^\infty \left(\frac{\widehat \lambda_n}{\lambda_n}\right)^{\frac 1 2}
\end{equation}
where $\widehat \lambda_n$ is the sequence obtained from $\lambda_n$, by leaving out all but the first two asymptotic terms.
This constant can be found only in a few cases (see, e.g., \cite{N09}) and is typically left to numerical approximation.

\medskip

For the fBm the logarithmic asymptotics was found in \cite{Br03b} (see also \cite{NN04tpa})
$$
\lim_{\eps\to 0}\eps^{\frac 1 { H}}\log \P(\|B^H\|_2 \le \eps)= - \frac{H}{(2H+1)^{\frac {2H+1}{2H}}}
\left(
\frac{\sin (\pi H) \Gamma(2H+1)}{\left(\sin \left(\frac{\pi}{2H+1}\right)\right)^{2H+1}}
\right)^{\frac 1 {2H}} := -\beta(H),
$$
but the exact asymptotics remained unknown for $H\ne \frac 1 2$. The refinement \eqref{lambda}-\eqref{nun}, plugged into
the general result of Theorem 6.2 in \cite{NN04ptrf}, allows to fill this gap:

%
%

\begin{prop}
For all $H\in (0,1)$,
$$
\P(\|B^H\|_2 \le \eps) \simeq C_d(H)
C(H) \eps^{\gamma(H)} \exp \left(-\beta(H)\eps^{-\frac 1 H}\right)
, \quad \eps \to 0
$$
where $C_d(H)$ is the distortion constant, defined in \eqref{Cd}, $C(H)$ is an explicit constant (given by (6.6) in \cite{NN04ptrf}) and

$$
\gamma(H) =  \frac 1{2H} \bigg(3/4+ H^2-(1+2H) \dfrac 1 \pi \arcsin \frac{\ell_H}{\sqrt{1+\ell_H^2}}\bigg),
$$

with $\ell_H$ defined in Theorem \ref{main-thm-fbm}.
\end{prop}

\subsection{Singularly perturbed integral equations}

Integral equations of the second kind
\begin{equation}
\label{ie2nd}
\eps u_\eps(x) + (Ku_\eps)(x) = f(x),\quad x\in [0,1]
\end{equation}
are well known to have unique solutions for any positive value of the constant $\eps$, under quite general assumptions on
the self-adjoint operator $K$ and the {\em forcing} function $f$ (see, e.g. \cite{RN55}). On the other hand, equations
of the first kind, formally obtained by setting $\eps=0$ in \eqref{ie2nd}
\begin{equation}
\label{ie1st}
(Ku_0)(x) = f(x),\quad x\in [0,1]
\end{equation}
may have no solutions under the same conditions or, if a solution $u_0$ exists, its properties may be quite different from
those of $u_\eps$. The free term in \eqref{ie2nd} therefore has a regularizing effect and can be viewed as a singular perturbation.
It then makes sense to study the asymptotic properties of $u_\eps$ as $\eps\to 0$.
If the limit equation \eqref{ie1st} does have a unique solution,
the natural questions are whether $u_\eps$ converges to $u_0$ and if it does,
in which sense and how fast.

Singularly perturbed integral equations are frequently encountered in physics and engineering applications and various ad-hoc 
approximation techniques have been proposed for their solutions (see, e.g. \cite{AO87a,AO87b} and references therein).
In general, asymptotic behavior of the solutions depends heavily on the singularities of the kernel.
Rigorous results for kernels with jump singularities can be found in \cite{LS88, LS93}, \cite{Sh06} and,
to the best of our knowledge, the weakly singular case has never been addressed before.

Equation \eqref{ie2nd} with kernels \eqref{KfBm} and \eqref{KfBn} arises in the stochastic analysis and 
optimal linear filtering problems involving fBm. As explained below, the particular forcing functions of interest in such problems 
are $f(u):=1$ and $f(u):=R(u,1)$ and the asymptotics of the corresponding solutions determine the steady state
behavior of the relevant probabilistic quantities.

\subsubsection{Mixed fractional Brownian motion}

Equation \eqref{ie2nd} arises in stochastic analysis of the mixture of independent 
standard and fractional Brownian motions $B$ and $B^H$: 
$$
\widetilde{B}_t = B_t + B_t^H, \quad t\in [0,T], \quad T<\infty.
$$
The process $\widetilde B$, called {\em the mixed} fBm, satisfies a number of curious properties with applications in 
mathematical finance, see \cite{BSV07}. In particular, as shown in \cite{Ch01, Ch03b}, it is a semimartingale if and only if 
$H\in \{\tfrac 1 2\}\cup (\tfrac 3 4,1]$ and, for $H>\tfrac 3 4$, the measure $\mu^{\widetilde B}$, induced by $\widetilde B$ 
on the  space of continuous functions, is equivalent to the standard Wiener measure $\mu^B$. 
On the other hand, it follows from the results in \cite{BP88}, \cite{vZ07} that $\mu^{\widetilde B}$ and $\mu^{B^H}$ are 
equivalent if and only if $H< \tfrac 1 4$.

These and other properties of $\widetilde B$ can be deduced from the canonical innovation representation obtained in \cite{CCK},
which is based on the martingale $M_t = \E(B_t|\F^{\widetilde B}_t)$, $t\in [0,T]$.
For any $H\in (0,1)$, this martingale satisfies
\begin{equation}
\label{repM}
M_t = \int_0^t g(s,t) d\widetilde B_s \quad \text{and}\quad \langle M\rangle_t = \int_0^t g(s,t) ds,\quad t\in (0,T],
\end{equation}
where $g(s,t)$ is the solution of the Wiener-Hopf type integro-differential equation (c.f. \eqref{KfBn})
\begin{equation}
\label{geq}
g(s,t) + \frac {d}{ds} \int_0^t g(r,t) H |s-r|^{2H-1} \sign(s-r)dr = 1, \quad 0<s<t \le T.
\end{equation}

The representation \eqref{repM} is useful in statistical analysis of linear models driven by the mixed fractional noise.
As shown in \cite{CKstat}, the large sample accuracy of parameter estimators in such models is governed by the
growth rate of the quadratic variation bracket $\langle M\rangle_T$ and its derivative $d\langle M \rangle_T/dT$ as $T\to\infty$.
Let us proceed with  $H>\frac 1 2$, for which by \eqref{KfBn2} the equation \eqref{geq} takes the simpler integral form
$$
g(s,t) + \int_0^t g(r,t) c_H  |r-s|^{2H-2} dr = 1, \quad 0<s<t \le T,
$$
with $c_H:=H(2H-1)$. It can be shown that its solution satisfies $g(t,t)>0$ and
$$
\langle M \rangle_t = \int_0^t g^2(s,s)ds, \quad t\ge 0.
$$

Let us define the small parameter $\eps := T^{1-2H}$, then a simple calculation shows that the function
$u_\eps(x):=T^{2H-1}g(xT,T)$ solves the singularly perturbed equation (c.f. \eqref{ie2nd}): 
\begin{equation}\label{WHeqmfbm}
\eps u_\eps(x)  + (\widetilde K u_\eps)(x)  =1, \quad x\in [0,1]
\end{equation}
with the operator $\widetilde K$ in the form \eqref{KfBn2}. 
Equations with such weakly singular kernels are well known to have a unique solution for any $\eps>0$,
which is continuous on the closed interval $[0,1]$, smooth in its interior, but not differentiable at the
endpoints (see, e.g., \cite{VP80}).
The unique solution of the limit equation in this case exists and is known in a closed form, \cite{LB98}:
$$
u_0(x) = a_H x^{\frac 1 2-H}(1-x)^{\frac 1 2 -H}, \quad x\in (0,1)
$$
where $a_H$ is an explicit constant. Note that $u_0$ explodes at the endpoints of the interval.

The martingale bracket and its derivative are related to the solutions of \eqref{WHeqmfbm} by the formulas: 
\begin{align}
\label{m1}
 \langle M \rangle_T  & =     \eps^{\frac{2-2H}{2H-1} }\langle  u_\eps, 1\rangle \\
\label{m2}
\frac {d \langle M \rangle_T}{dT}   &  = \eps^2 u^2_\eps (1).
\end{align}
In view of \eqref{m1}, we are particularly interested in  the {\em weak} convergence
\begin{equation}\label{weak}
\langle u_\eps,  \psi\rangle  \xrightarrow{\eps\to 0} \langle u_0, \psi\rangle,
\end{equation}
for test functions $\psi$ from a large class, which contains constants.
For \eqref{m2} we need to quantify the divergence rate of $u_\eps(1) \to +\infty$.

\medskip

\goodbreak

The following result gives a detailed description of both limits and reveal a curious phase transition 
of the $L_2$ convergence at $H=\frac  23$: 

\begin{prop}\label{thm2}\
\medskip

\begin{enumerate}
\addtolength{\itemsep}{0.7\baselineskip}
\renewcommand{\theenumi}{\roman{enumi}}

\item\label{ii}  $u_\eps \to u_0$ in $L^2(0,1)$ and as $\eps\to 0$,
$$
\|u_\eps -u_0\|_2  \sim  \begin{cases}
\eps^{\frac {1-H}{2H-1} } & H\in (\frac 2 3,1) \\
\eps \sqrt{\log \eps^{-1}} & H=\tfrac 2 3\\
\eps  & H \in (\frac 1 2, \frac 2 3)
\end{cases}
$$

\item\label{thm2-iii}
The solution $u_\eps$ diverges to $+\infty$ at the endpoints $x\in \{0,1\}$ and
$$
u_\eps(0)=u_\eps(1)  \sim \eps^{-\frac 1 2}, \quad \eps\to 0.
$$
\end{enumerate}

\end{prop}

The proof, deferred to Section \ref{sec-7}, relies on the asymptotics from Theorem \ref{main-thm-fbn}.

\subsubsection{Optimal linear filtering}

The optimal filtering problem deals with estimation of signals from noisy observations. The standard linear setup (see, e.g., \cite{LS1,LS2}) in continuous time consists of the
signal process $X$ and the observation process $Y$, generated by the equation
$$
Y_t = a \int_0^t  X_s ds + B_t, \quad t\in [0,T]
$$
where $B$ is a Brownian motion, independent of $X$ and $a$ is a fixed {\em gain} constant. In engineering literature this setup is known as the {\em additive Gaussian white noise} model.

The objective is to compute the optimal in the mean square sense estimator of $X_t$, given the past 
observations $Y_{[0,t]} = \{Y_s, s\in [0,t]\}$ for each time $t\in [0,T]$.
This amounts to calculation of the conditional expectation $\widehat X_t := \E(X_t|\F^Y_t)$, where $\F^Y_t$ is the filtration generated by $Y$.
The other important objective is to  compute the corresponding optimal filtering  error $P_t :=\E (X_t -\widehat X_t)^2$ and its steady state limit
$P_\infty :=\lim_{t\to\infty}P_t$, if it exists.

For a Gaussian process $X$, a standard calculation yields the formula
$$
\widehat X_t = \frac 1 a \int_0^t g(s,t) dY_s
$$
where $g(s,t)$ is the solution of the Wiener-Hopf integral equation
\begin{equation}\label{WH}
g(s,t) + \int_0^t g(r,t) a^2 R(r,s)dr = a^2 R(s,t), \quad 0\le s\le t\le T,
\end{equation}
and the estimation error is given by
$$
 P_t =\frac 1 {a^2}\left( R(t,t) - \int_0^t g(s,t)R(s,t)ds\right) = \frac 1 {a^2}g(t,t).
$$

When $X$ is a Gauss--Markov process, the equation \eqref{WH} can be reduced to the Riccati differential equation and in this case the optimal estimator $\widehat X_t$
and the corresponding error $P_t$ satisfy the celebrated Kalman--Bucy filtering equations.  In particular, both the  error $P_t$, $t\in [0,T]$ and its steady
state limit $P_\infty$ can be computed in closed forms.
When the signal process $X$ is not Markov, no such convenient way of computing $P_\infty$ is available.
If the kernel $R(s,t)$ belongs to $L^2[0,T]^2$, the unique solution of the equation \eqref{WH} can be expanded into series of its eigenfunctions and
much information about the filtering problem can be extracted from their asymptotic behavior.

Let us now revisit this filtering problem for the fBm, that is, with $X:=B^H$.
The covariance function \eqref{RfBm} satisfies the scaling property
$$
R(xT,yT) = T^{2H} R(x,y), \quad x,y\in [0,1]
$$
and, if we define small parameter $\eps:= a^{-2} T^{-2H-1}$ and set $u_\eps(x):=T g(xT,T)$, the equation \eqref{WH} with $t:=T$
gives the singularly perturbed equation \eqref{ie2nd} with operator $K$ from \eqref{KfBm} and the forcing $f(x)= R(x,1)$:
$$
\eps u_\eps(x) + \int_0^1 u_\eps(y)  R(y,x)dy  =     R(x,1), \quad x\in [0,1].
$$
Note that unlike in the problem discussed in the previous subsection,
the limit equation in this case does not have classical solution.

\begin{prop}\label{prop-filt}
The steady state filtering error is given by
$$
\lim_{T\to\infty }P_T = \frac { \big(\sin (\pi H)\Gamma(2H+1)\big)^{\frac 1 {2H+1}} }{\sin \frac {\pi}{2H+1}} a^{-\frac{4H}{2H+1}}.
$$
\end{prop}

\begin{rem}
Curiously, the worst steady state error at $a=1$ is obtained for $H=\frac 2 3$.
\end{rem}

\goodbreak

\section{The proof outline}\label{sec-strat}

As mentioned in Introduction, the eigenproblems for processes related to the standard Brownian motion are often solved by reduction to
boundary value problems for differential operators. For the fBm such reduction does not seem to be possible and we will take a
different route,  based on analytic properties of the Laplace transform
\begin{equation}
\label{LT}
\widehat \varphi(z) = \int_0^1 \varphi(x)e^{-zx}dx, \quad z\in \mathbb{C}.
\end{equation}

In a nutshell, the idea is to consider the Laplace transform $\widehat \varphi(z)$
as an analytic function on the complex plane $\mathbb{C}$. Using the particular structure of the eigenproblem,
an expression for $\widehat \varphi(z)$ with singularities is derived and their removal produces an
alternative characterization for the eigenvalues and eigenfunctions.
The proof is inspired by the approach in \cite{Ukai} to asymptotic approximation of the eigenvalues for
weakly singular integral operators.

As will become clear from the proof, the method, in principle, applies to operators with the difference kernels satisfying   
$$
K(|x-y|) = \int_0^\infty \kappa(t) e^{-t|x-y|}dt,
$$
with some function $\kappa(t)$, and their compositions with the integration operator.  However, the implementation of the method in 
each particular situation turns out to be very specific to the fine structure of the kernel under consideration and, in our experience, 
often requires different tricks and leads to entirely unexpected outcomes.  

\medskip 

\goodbreak

Let us sketch the main steps of the proof.

\medskip
\noindent
{\em 1) The Laplace transform.} For both eigenproblems considered in this paper, it is possible to find an expression for
the Laplace transform \eqref{LT}, whose main ingredient has the form
\begin{equation}\label{hatphi}
\frac 1{\Lambda(z)}\Big(e^{-z}\Phi_1(-z)+\Phi_0(z)\Big),\quad z\in \mathbb{C}
\end{equation}
where functions $\Phi_1(z)$ and $\Phi_0(z)$ are sectionally holomorphic on the cut plane $\mathbb{C}\setminus \Real_{\ge 0}$
and $\Lambda(z)$ is an explicit function (see Lemmas \ref{lem-main}, \ref{lem510} and \ref{lem61} below).
Such a representation is tailored to the particular structure of the operator and
is constructed on the case to case basis. The function $\Lambda(z)$ has a jump discontinuity on the real line
and a pair of purely imaginary zeros at $\pm i \nu$, where $\nu\in\Real_{>0}$ is related to $\lambda$ through an explicit
formula (see \eqref{nufla} and \eqref{nuflan} below).

\medskip
\noindent
{\em 2) Removal of singularities.} Since, a priori, $\widehat \varphi(z)$ is an entire function, 
both singularities must be removable. Removal of the poles in \eqref{hatphi} gives the algebraic conditions
\begin{equation}
\label{algPhi}
e^{\mp i\nu }\Phi_1(\mp i\nu)+\Phi_0(\pm i\nu)=0.
\end{equation}
Removing the discontinuity on the real line produces the boundary conditions
\begin{equation}\label{bindPhi}
\begin{aligned}
&
\Phi_0^+(t) - e^{2i\theta(t)}\Phi_0^-(t) = 2i  e^{-t} e^{i\theta(t)}\sin\theta(t) \Phi_1(-t)
\\
&
\Phi_1^+(t) - e^{2i\theta(t)}\Phi_1^-(t) = 2i e^{-t}  e^{i\theta(t)}\sin\theta(t) \Phi_0(-t)
\end{aligned}
\qquad
t\in \Real_{>0}
\end{equation}
which bind together the limits of $\Phi_0(z)$ and $\Phi_1(z)$ across the positive real semiaxis $\Real_{>0}$.
Here $\theta(t)$ is the argument of the limit $\Lambda^+(t):=\lim_{z\to t^+}\Lambda(z)$.

Therefore the eigenproblem \eqref{eigpr}  reduces to finding a pair of
functions $\Phi_0(z)$ and $\Phi_1(z)$, sectionally holomorphic  on the cut plane
$\mathbb{C}\setminus \Real_{\ge 0}$ satisfying the  boundary conditions \eqref{bindPhi} and the
algebraic constraint \eqref{algPhi}. Indeed, any such pair and a number $\nu>0$ can be used to
construct a solution $(\lambda,\varphi)$ to the eigenproblem by inverting the Laplace transform.
The next step is to find an equivalent formulation of this problem, using the techniques of solving 
boundary value problems on the complex plane.

\medskip
\noindent
{\em 3)\label{punkt3} An equivalent formulation of the eigenproblem.} The conditions \eqref{algPhi} and \eqref{bindPhi} can be rewritten as
an integro-algebraic system of equations as follows.

\medskip

{\em a)} The homogeneous Riemann boundary value problem is solved to find a nonvanishing function $X(z)$,
which is sectionally holomorphic on $\mathbb{C}\setminus \Real_{\ge 0}$  and satisfies
$$
\frac{X^+(t)}{X^-(t)}= e^{2i\theta(t)}, \quad t\in \Real_{>0}.
$$
The solution is not unique and is fixed to match the particular properties of $\theta(t)$ and a priori
estimates on $\Phi_0(z)$ and $\Phi_1(z)$ in each eigenproblem.

\medskip

{\em b)} The boundary conditions \eqref{bindPhi} can now be put into the decoupled form
\begin{equation}
\label{SDpm}
\begin{aligned}
&
S^+(t)\, -\, S^-(t)  =\phantom{+} 2ih_0(t/\nu)e^{-t} S(-t)
\\
&
D^+(t) - D^-(t) = - 2i h_0(t/\nu )e^{-t} D(-t)
\end{aligned}
\end{equation}
where we defined
\begin{equation}\label{SDPhi01}
\begin{aligned}
&
S(z)\, := \frac{\Phi_0(z)+\Phi_1(z)}{2X(z)} \\
&
D(z):= \frac{\Phi_0(z)-\Phi_1(z)}{2X(z)}
\end{aligned}
\end{equation}
and the function
$$
h_0(t):=e^{i\theta(\nu t)}\sin\theta(\nu t) \dfrac{X(-\nu t)}{X^+(\nu t)}.
$$
This function turns out to be real valued and independent of $\nu$.

Applying the Sokhotski--Plemelj formula (see \S 43 in \cite{Gahov}),  conditions \eqref{SDpm} can be further rewritten as
\begin{equation}
\label{SDsol}
\begin{aligned}
& S(z) = \phantom{+} \frac 1 {\pi } \int_0^\infty \frac{h_0(t/\nu)e^{-t} }{t-z}S(-t) dt+P_1(z)\\
& D(z) = -\frac 1 \pi \int_0^\infty  \frac{h_0(t/\nu)e^{-t}}{t-z} D(-t)dt + P_2(z)
\end{aligned}
\end{equation}
where $P_1(z)$ and $P_2(z)$ are arbitrary polynomials. Their degrees are chosen to match
the a priori growth of $S(z)$ and $D(z)$ at infinity: $S(z)-P_1(z)\to 0$ and $D(z)-P_2(z)\to 0$ as $z\to \infty$.
The coefficients of these polynomials  are chosen later using the conditions \eqref{algPhi}, which read
\begin{equation}
\label{ImRe}
\Im \Big\{e^{i\nu/2}X(i\nu)  D(i\nu)\Big\}   =\Re\Big\{e^{i\nu/2}X(i\nu)S(i\nu)\Big\}=0.
\end{equation}

At this point an equivalent characterization of the eigenvalues and eigenfunctions is obtained:
any nontrivial solution $(\lambda, \varphi)$ of the eigenproblem \eqref{eigpr} defines functions $S(z)$ and $D(z)$ and a positive constant $\nu$, which satisfy equations \eqref{SDsol} and \eqref{ImRe} and vice versa.

\medskip
\noindent
{\em 4) The integro-algebraic system and its properties.}
By setting $z:=-t$ with $t\in \Real_{>0}$ in \eqref{SDsol},
a pair of integral equations for $S(-t)$ and $D(-t)$ is obtained
\begin{equation}
\label{SDeqreal}
\begin{aligned}
& S(-t) = \phantom{+} \frac 1 {\pi } \int_0^\infty \frac{h_0(s/\nu)e^{-s} }{s+t}S(-s) ds + P_1(-t)\\
& D(-t) = -\frac 1 \pi \int_0^\infty  \frac{h_0(s/\nu)e^{-s}}{s+t} D(-s)ds +P_2(-t)
\end{aligned}\qquad t\in \Real_{>0}
\end{equation}
which together with \eqref{ImRe} and \eqref{SDsol} form an integro-algebraic system of equations. 

By construction, both functions $S-P_1$ and $D-P_2$ are square integrable. 
On the other hand, the integral equations \eqref{SDeqreal} have unique solutions with such property,  
since the operator on the right hand side is contracting on $L_2(0,\infty)$ for all sufficiently large $\nu >0$. 
Therefore, nontrivial solutions $(D,S,\nu)$ to the integro-algebraic system are at one-to-one correspondence 
with the solutions $(\lambda, \varphi)$ to the eigenproblem.

\medskip
\noindent
{\em 5) Inversion of the Laplace transform.}  The eigenfunctions are recovered by inverting the Laplace transform:
\begin{multline}
\label{varphi}
\varphi(x)  =  \frac 1 { \pi } \int_0^{\infty}    \frac{ \sin \theta(t)}{|\Lambda^+(t)|} \left(e^{-t(1-x)}\Phi_1(-t)+e^{-tx} \Phi_0(-t)\right)dt \\
-\Res\left\{e^{zx} \frac{\Phi_0(z)}{\Lambda(z)},i\nu\right\}
-\Res\left\{e^{zx} \frac{\Phi_0(z)}{\Lambda(z)},-i\nu\right\}.
\end{multline}
Note that $\varphi(x)$ is expressed here in terms of the solutions of integral equations \eqref{SDeqreal},
which determine $\Phi_1(z)$ and $\Phi_0(z)$ through \eqref{SDPhi01}

\medskip
\noindent
{\em 6) Asymptotic analysis.}
Enumerating solutions of the integro-algebraic system in a certain convenient way,
it is possible to find asymptotic approximation for their algebraic part $\nu>0$, which is precise up to the second order.
The announced asymptotics of the eigenvalues is then derived from the explicit relation between $\nu$ and $\lambda$.
Plugging it back into the system and using suitable estimates for the norm of the operator in the integral equations
\eqref{SDeqreal}, the eigenfunctions asymptotics is extracted from the formula \eqref{varphi}.

\medskip
\noindent
{\em 7) Enumeration alignment.} The enumeration chosen in the previous step may differ from the ``natural'' enumeration,
which puts the eigenvalues into decreasing order. This does not affect the leading term asymptotics, but may change the
second order term. Alignment of the two enumerations is done through a calibration procedure, based on continuity of
the spectrum. This delicate part of the proof is also carried out differently in the two eigenproblems
under consideration.

\goodbreak

\section{Proof of Theorem \ref{main-thm-fbm}}

In this section we implement the above program for the eigenproblem with the fBm covariance operator \eqref{KfBm}:
\begin{equation}
\label{eig}
\int_0^1 \frac 1 2 \big(t^{2H}+s^{2H}-|t-s|^{2H}\big) \varphi(s)ds = \lambda \varphi(t), \quad t\in [0,1].
\end{equation}
It will be convenient to work with the parameter $\alpha := 2-2H\in (0,2)$ and to address the case $\alpha<1$ first. The complementary case $\alpha>1$
can be treated along the same lines with  appropriate adjustments (see Section \ref{sec-case2}).

\subsection{The case $\alpha<1$}

\subsubsection{The Laplace transform} The following lemma gives a useful representation formula for the
Laplace transform \eqref{LT} (c.f. \eqref{hatphi}):

\begin{lem}\label{lem-main}
For  $\alpha \in (0,1)$,
\begin{equation}
\label{main}
\widehat \varphi(z) - \widehat \varphi(0) =-\frac 1{\Lambda(z)}\Big(e^{-z}\Phi_1(-z)+\Phi_0(z)\Big),\quad z\in \mathbb{C}
\end{equation}
where the functions $\Phi_0(z)$ and $\Phi_1(z)$, defined in \eqref{Phi0Phi1} below,  are sectionally holomorphic  on $\mathbb{C}\setminus \Real_{\ge 0}$ and
\begin{equation}
\label{Lmbd1}
\Lambda(z) :=
\frac{\lambda \Gamma (\alpha)} {c_\alpha} z + \frac 1 z  \int_0^\infty  \frac { 2 t^{\alpha }}{t^2-z^2} dt,
\end{equation}
with $c_\alpha : = (1-\frac\alpha 2)(1-\alpha)$.
\end{lem}

\begin{proof}
Let us define $\displaystyle\psi(x):=\int_x^1 \varphi(y)dy$ and note that $\psi(1)=0$, $\psi'(0)=-\varphi(0)=0$ and  $\psi(0)=\widehat \varphi(0)$.
Integration by parts in \eqref{eig} gives
\begin{equation}
\label{1diff}
\begin{aligned}
\lambda \psi'(x) 
&
=\int_0^1 R(y,x) \psi'(y)dy = -\int_0^1   \partial_y R(y,x)\psi(y) dy  \\
&
=-\int_0^1   (1-\tfrac \alpha 2)\left (y^{1-\alpha}+\sign(x-y) |x-y|^{1-\alpha}\right) \psi(y)dy,
\end{aligned}
\end{equation}
and differentiating with respect to $x$, we get
\begin{equation}
\label{interch}
\lambda \psi''(x) =
-\frac d{dx}\int_0^1   (1-\tfrac \alpha 2)   \sign(x-y) |x-y|^{1-\alpha}  \psi(y)dy=
-\int_0^1   c_\alpha   |x-y|^{-\alpha}  \psi(y)dy
\end{equation}
where interchanging integration and derivative is possible since $\alpha<1$ is assumed. Thus the problem \eqref{eig} is equivalent to
the generalized eigenproblem
\begin{equation}
\label{geig}
\begin{aligned}
\int_0^1  & c_\alpha  |x-y|^{-\alpha}  \psi(y)dy = -\lambda \psi''(x) \\
&
\psi(1)=0, \; \psi'(0)=0
\end{aligned}
\end{equation}
Using the identity
\begin{equation}
\label{idnt}
|x-y|^{-\alpha} = \frac 1 {\Gamma (\alpha)} \int_0^\infty t^{\alpha-1}e^{-t|x-y|}dt,\quad \alpha\in (0,1)
\end{equation}
the equation  \eqref{geig} can be rewritten as
$$
\frac {c_\alpha}{\Gamma (\alpha)}\int_0^\infty t^{\alpha-1} \int_0^1  e^{-t|x-y|} \psi(y) dy  dt = -\lambda \psi''(x).
$$
If we now define
$$
u(x,t) :=  \int_0^1  e^{-t|y-x|} \psi(y) dy  \quad \text{and}\quad
u_0(x):=  \int_0^\infty t^{\alpha-1}u(x,t) dt,
$$
we get
\begin{equation}
\label{u0}
\frac {c_\alpha}{\Gamma (\alpha)}u_0(x) = -\lambda \psi''(x).
\end{equation}

Differentiating $u(x,t)$ twice with respect to $x$ gives the equation
\begin{equation}
\label{uttag}
u''(x,t) = t^2 u(x,t) -2t \psi(x)
\end{equation}
subject to the boundary conditions
\begin{equation}
\label{bnccnd}
\begin{aligned}
& u'(0,t)= \phantom{+} tu(0,t) \\
& u'(1,t) = -t u(1,t).
\end{aligned}
\end{equation}
Let us compute the Laplace transform of both sides of \eqref{uttag}. To this end, we have
\begin{align*}
\widehat u''(z,t) & :=
\int_0^1 e^{-zx}u''(x,t)dx = 
e^{-z}u'(1,t)-u'(0,t) +z \Big(e^{-z} u(1,t)-u(0,t)+z \widehat u(z,t)\Big)\\
&
=e^{-z}\Big(u'(1,t) +z    u(1,t)\Big) -\Big(u'(0,t) +z u(0,t)\Big)+z^2 \widehat u(z,t).
\end{align*}
After plugging in the boundary conditions \eqref{bnccnd}, this becomes
$$
\widehat u''(z,t)  = e^{-z} (z -t )u(1,t) - (z+t)u(0,t)+z^2 \widehat u(z,t),
$$
which, along with \eqref{uttag}, gives
$$
(z^2-t^2)\widehat u(z,t)  =   u(0,t) (z+t)-e^{-z} u(1,t)(z-t)-2t \widehat \psi(z).
$$
Consequently, for all $z\in \mathbb{C}\setminus \Real$,
\begin{align*}
\widehat u_0(z) & =  \int_0^\infty t^{\alpha-1}\widehat u(z,t) dt   \\
&
=
\int_0^\infty  \frac {t^{\alpha-1}} {z -t } u(0,t)  dt
-e^{-z}\int_0^\infty\frac {t^{\alpha-1}} {z+t}  u(1,t) dt- \widehat \psi(z)\int_0^\infty  \frac { 2 t^{\alpha }}{z^2-t^2} dt.
\end{align*}
On the other hand,  by \eqref{u0} and the boundary conditions in \eqref{geig}
\begin{align*}
\widehat u_0(z) 
& = -\lambda\frac{ \Gamma (\alpha)}{c_\alpha} \widehat\psi''(z) =
-\lambda\frac{ \Gamma (\alpha)}{c_\alpha} \Big(e^{-z} \psi'(1) +z\widehat \psi'(z)\Big) \\
&
 =
-\lambda\frac{ \Gamma (\alpha)}{c_\alpha}
\left(
z^2 \widehat \psi(z)+\psi'(1)e^{-z} -z \psi(0)
\right).
\end{align*}
Combining the two expressions gives
\begin{align*}
&
\left( \lambda\frac{ \Gamma (\alpha)}{c_\alpha} z -\frac 1 z \int_0^\infty  \frac { 2 t^{\alpha }}{z^2-t^2} dt
\right)z\widehat \psi(z)
 = \\
& \lambda\frac{ \Gamma (\alpha)}{c_\alpha}
  \psi(0) z + \int_0^\infty  \frac {t^{\alpha-1}} {t-z } u(0,t)  dt
 -e^{-z}\lambda\frac{ \Gamma (\alpha)}{c_\alpha} \psi'(1) +e^{-z} \int_0^\infty\frac {t^{\alpha-1}} {t+z}  u(1,t) dt.
\end{align*}
Note that $z\widehat \psi(z)= \widehat \psi'(z)+\psi(0)=-\widehat \varphi(z)+\widehat \varphi(0)$ and thus \eqref{main} is obtained if we define $\Lambda(z)$ as in \eqref{Lmbd1} and
set
\begin{equation}\label{Phi0Phi1}
\begin{aligned}
& \Phi_0(z)  := \phantom{+}\lambda\frac{ \Gamma (\alpha)}{c_\alpha}  \psi(0) z + \int_0^\infty  \frac {t^{\alpha-1}} {t-z } u(0,t)  dt\\
& \Phi_1(z)  :=
-\lambda\frac{ \Gamma (\alpha)}{c_\alpha} \psi'(1)  +  \int_0^\infty\frac {t^{\alpha-1}} {t-z}  u(1,t) dt.
\end{aligned}
\end{equation}

\end{proof}

Several useful properties of $\Lambda(z)$ are summarized  in the following lemma

\begin{lem}\

\medskip
\noindent
a) The function $\Lambda(z)$, defined in \eqref{Lmbd1}, admits the explicit expression
\begin{equation}
\label{Lz}
\Lambda(z)  =
\frac{\lambda \Gamma (\alpha)} {c_\alpha} z + z^{\alpha-2}  \frac {\pi }{  \cos \frac {\pi } 2 \alpha } \begin{cases}
e^{\frac {1-\alpha}  2 \pi i     } &   \arg(z)\in (0,\pi) \\
e^{-\frac {1-\alpha}  2 \pi i     } &  \arg(z)\in (-\pi, 0)
\end{cases}
\end{equation}
which is discontinuous along the real axis and has two zeros at $\pm z_0 = \pm i \nu$ with
\begin{equation}
\label{nufla}
\nu^{\alpha-3} =  \frac{\lambda \Gamma (\alpha)} {c_\alpha}\frac{\cos \frac \pi 2 \alpha}{\pi}.
\end{equation}

\medskip
\noindent
b) The limits $\Lambda^{\pm} (t) := \lim_{z\to t^\pm}\Lambda(z)$ are given by the expressions 
\begin{equation}
\label{Lfla}
\Lambda^{\pm} (t)  =
\frac{\lambda \Gamma (\alpha)} {c_\alpha} t \pm  |t|^{\alpha-2}  \frac {\pi  }{  \cos \frac {\pi } 2 \alpha }  \begin{dcases}
e^{\frac {1\mp\alpha}2 \pi i   }  & t> 0 \\
e^{\frac {1\pm\alpha}2 \pi i   }  & t<0
\end{dcases}
\end{equation}
which admit the following symmetries
\begin{align}
 \Lambda^+(t) & =\overline{\Lambda^-(t)} \label{conjp}\\
 \frac{\Lambda^+(t)}{\Lambda^-(t)} & =\frac{\Lambda^-(-t)}{\Lambda^+(-t)}   \label{prop}   \\
 \big|\Lambda^+(t)\big| &  =\big|\Lambda^+(-t)\big|   \label{absL}
\end{align}

\medskip
\noindent 
c) The argument $\theta(t):=\arg\{\Lambda^+(t)\}\in (-\pi,\pi]$ satisfies $\theta(-t)=\pi-\theta(t)$ and
\begin{equation}\label{thetat}
\theta(t) =  \arctan \frac{(t/\nu)^{\alpha-3}\sin \frac{1-\alpha}{2}\pi }{1+(t/\nu)^{\alpha-3}\cos \frac {1-\alpha} 2\pi}, \quad t>0
\end{equation}
decreasing continuously from $\theta(0+):=\lim_{t\to 0+}\theta(t)=\frac{1-\alpha}{2}\pi>0$ to $\theta(\infty):=0$ as $t\to\infty$.
Moreover, 
\begin{equation}\label{thetaint}
\frac 1 \pi \int_0^\infty \theta(t) dt = \nu \frac{\sin \Big(\frac \pi {3-\alpha}\frac{1-\alpha}{2}\Big)}{\sin\frac \pi {3-\alpha}}=: \nu b_\alpha.
\end{equation}

\end{lem}

\begin{proof}
The formula \eqref{Lz} follows from \eqref{Lmbd1} and the identity
\begin{equation}
\label{Lint}
\int_{0}^{\infty} \frac{t^\alpha}{t^2-z^2}dt  = z^{\alpha-1}\frac 1 2
\frac{\pi   }{ \cos \frac{\pi } 2\alpha}
\begin{cases}
e^{\frac{1-\alpha}{2}\pi i} & \arg (z) \in (0,\pi) \\
e^{-\frac{1-\alpha}{2}\pi i } &   \arg(z)\in (-\pi, 0)
\end{cases}
\end{equation}
which is obtained by integrating
the function $\xi \mapsto \frac{\xi^\alpha }{\xi^2-z^2}$ over the  circular
contour, cut along the negative real semiaxis. 
Properties \eqref{nufla}-\eqref{thetat} are deduced by direct calculations, using \eqref{Lz}.
The integral in \eqref{thetaint} reduces to 
$$
\int_0^\infty \theta(t) dt = \nu \Big(\sin \tfrac{1-\alpha}{2}\pi\Big)^{\frac 1{\alpha-3}} \int_0^\infty \frac{\tau^{\frac 1 {3-\alpha}}}{1+(\tau+\cot \tfrac{1-\alpha}{2}\pi)^2}d\tau
$$
through a change of variable and the claimed formula is obtained by an appropriate contour integration. 

\end{proof}

\subsubsection{Removal of singularities}

Since $\widehat \varphi(z)$ is a priori an entire function, all the singularities must be removable.
Note that $\Phi_i(-z_0)=\overline{\Phi_i(z_0)}$, $i=0,1$ and therefore
the numerator in \eqref{main}  vanishes at $\pm z_0=\pm i\nu$ if
\begin{equation}
\label{cond1}
e^{- z_0}\Phi_1(- z_0)+\Phi_0(z_0)=0.
\end{equation}
Removal of discontinuity on the real line gives the condition
\begin{equation}
\label{drem}
\lim_{z\to t^+}\frac 1{\Lambda(z)}\big(e^{-z}\Phi_1(-z)+\Phi_0(z)\big)  = \lim_{z\to t^-}\frac 1{\Lambda(z)}\big(e^{-z}\Phi_1(-z)+\Phi_0(z)\big).
\end{equation}
When $-z$ approaches $t\in \Real$ from below the real line, $z$ approaches it from above,
%
%
and therefore, since both $\Phi_0(z)$ and $\Phi_1(z)$ do not have singularities on the negative real semiaxis, the condition \eqref{drem} is equivalent to
$$
\begin{aligned}
& \frac 1{\Lambda^+(t)}\big(e^{-t}\Phi_1(-t)+\Phi_0^+(t)\big)  = \frac 1{\Lambda^-(t)}\big(e^{-t}\Phi_1(-t)+\Phi_0^-(t)\big),
\quad t>0   \\
& \frac 1{\Lambda^+(t)}\big(e^{-t}\Phi_1^-(-t)+\Phi_0(t)\big)  = \frac 1{\Lambda^-(t)}\big(e^{-t}\Phi_1^+(-t)+\Phi_0(t)\big), \quad   t<0
\end{aligned}
$$
In view of \eqref{prop} these two equations can be rewritten as
\begin{equation}
\label{Hp1}
\begin{aligned}
&
\Phi_0^+(t) - \frac{\Lambda^+(t)}{\Lambda^-(t)}\Phi_0^-(t) = e^{-t} \Phi_1(-t)\left(\frac{\Lambda^+(t)}{\Lambda^-(t)}-1\right)
\\
&
\Phi_1^+(t)-\frac{\Lambda^+(t)}{\Lambda^-(t)}\Phi_1^-(t) = e^{-t} \Phi_0(-t)\left(\frac{\Lambda^+(t)}{\Lambda^-(t)}-1\right)
\end{aligned}\quad t>0.
\end{equation}
Since  $\Lambda^+(t)$ and $\Lambda^-(t)$ are complex conjugates,   it follows that
$
\Lambda^+(t)/\Lambda^-(t)=e^{2i\theta(t)}
$
and
$$
\frac{\Lambda^+(t)}{\Lambda^-(t)}-1=e^{2i\theta(t)}-1 = 2i e^{i\theta(t)}\sin\theta(t).
$$
Therefore \eqref{Hp1} becomes
\begin{equation}
\label{Hp}
\begin{aligned}
&
\Phi_0^+(t) - e^{2i\theta(t)}\Phi_0^-(t) = 2i  e^{-t} e^{i\theta(t)}\sin\theta(t) \Phi_1(-t)
\\
&
\Phi_1^+(t) - e^{2i\theta(t)}\Phi_1^-(t) = 2i e^{-t}  e^{i\theta(t)}\sin\theta(t) \Phi_0(-t)
\end{aligned}\quad t>0.
\end{equation}

Further, since $t u(0,t)$ and $tu(1,t)$ are uniformly bounded, the formulas \eqref{Phi0Phi1} imply that for any $\alpha \in (0,1)$
\begin{equation}\label{grinf}
\Phi_0(z)   \sim z  \quad \text{and}\quad
\Phi_1(z)   \sim \text{const.}
\quad \text{as\ } z\to\infty
\end{equation}
and
\begin{equation}
\label{grzero}
|\Phi_0(z)| \sim  |z|^{\alpha-1}  \quad \text{and}\quad |\Phi_1(z)|\sim   |z|^{\alpha-1} \quad \text{as\ } z\to 0.
\end{equation}

\subsubsection{An equivalent formulation of the eigenproblem}

At this point the eigenproblem \eqref{eig} reduces to finding functions $\Phi_0(z)$ and $\Phi_1(z)$, which are
sectionally holomorphic on $\mathbb{C}\setminus \Real_{\ge 0}$ and satisfy the boundary conditions \eqref{Hp},
the constraint \eqref{cond1} and the growth conditions \eqref{grinf} and \eqref{grzero}.
Following the program outlined in Section \ref{sec-strat}, we will construct such functions as solutions of an
integro-algebraic system of equations. 

To this end,  define
\begin{equation}\label{X0z}
X_0(z) =\exp \left(\frac 1 {\pi} \int_0^\infty \frac{\theta_0(t)}{  t-  z}d  t\right), \quad z\in \mathbb{C}\setminus \Real_{\ge 0}
\end{equation}
where the function $\theta_0(t):=\theta(\nu t)$ does not depend on $\nu$, see \eqref{thetat}. For any fixed $\nu>0$
let $p_\pm(t)$ and $q_\pm(t)$  be solutions of the integral equations
\begin{equation}
\label{pq}
\begin{aligned}
& p_\pm(t  ) = \pm \frac 1 {\pi } \int_0^\infty \frac{h_0(s)e^{-\nu s} }{s +t }p_\pm (s) ds  + 1 \\
& q_{\pm }(t  ) = \pm \frac 1 {\pi } \int_0^\infty \frac{h_0(s)e^{-\nu s} }{s+t  }q_\pm (s) ds + t
\end{aligned}\quad\quad t>0
\end{equation}
where the function
$$
h_0(t) := e^{i\theta_0(t)}\sin \theta_0(t) \frac {X_0(-t)}{X_0^+(t)}
$$
is real valued, see \eqref{hsq} below. We will extend the definition of $q_\pm$ and $p_\pm$ to the cut plane, setting
$$
\begin{aligned}
& p_{\pm }(z) := \pm \frac 1 {\pi } \int_0^\infty \frac{h_0(s)e^{-\nu s} }{s+z}p_\pm (s) ds + 1 \\
& q_{\pm }(z) := \pm \frac 1 {\pi } \int_0^\infty \frac{h_0(s)e^{-\nu s} }{s+z}q_\pm (s) ds + z
\end{aligned}
\quad\quad  z\in \mathbb{C}\setminus \Real_{\le 0}
$$
Finally let us define the functions
\begin{equation}\label{abpm}
\begin{aligned}
& a_\pm(z) = p_+(z)\pm p_-(z) \\
& b_\pm(z) = q_+(z)\pm q_-(z)
\end{aligned}
\end{equation}
and the constants
\begin{equation}\label{xieta}
\begin{aligned}
&
\xi := e^{ i\nu/2} X_0(i)\big(b_+(-i)-b_\alpha a_+(-i) \big)+e^{-i\nu/2} X_0(-i) \big(b_-(i)-b_\alpha a_-(i)\big) \\
&
\eta:=  e^{i\nu/2} X_0(i) a_-(-i)+e^{-i\nu/2} X_0(-i) a_+(i)
\end{aligned}
\end{equation}
where $b_\alpha$ is defined in \eqref{thetaint}.

\medskip

The following lemma gives an equivalent formulation of the eigenproblem \eqref{eig}:

\begin{lem}\label{lemia}
Let $(p_\pm, q_\pm, \nu)$ with $\nu>0$ be a solution of the system,
which consists of the integral equations \eqref{pq} and the algebraic equation
\begin{equation}\label{alg}
\Im\{\xi \overline {\eta}\} = 0.
\end{equation}
Let $\varphi$ be defined by the Laplace transform \eqref{main}, where
\begin{equation}
\label{Phiz}
\begin{aligned}
& \Phi_0(z) :=  
  X_0(z/\nu)\Big( b_+(-z/\nu)-b_\alpha a_+(-z/\nu)  -  \frac{\xi}{\eta}   a_-(-z/\nu)\Big)
\\
& \Phi_1(z) := 
  X_0(z/\nu) \Big( b_-(-z/\nu)-b_\alpha a_-(-z/\nu) -  \frac{\xi}{\eta}   a_+(-z/\nu)\Big)\end{aligned}
\end{equation}
and $\lambda$ be defined by \eqref{nufla}. Then the pair $(\lambda, \varphi)$ solves the eigenproblem \eqref{eig}.
Conversely, any pair $(\lambda, \varphi)$ satisfying \eqref{eig} defines a solution of the above integro-algebraic system.
\end{lem}

\begin{proof}
Suppose $(\lambda, \varphi)$ is a solution of the eigenproblem and let us show that the functions $\Phi_0(z)$ and $\Phi_1(z)$,
defined in \eqref{Phi0Phi1}, and the real number $\nu>0$, defined by \eqref{nufla}, solve the aforementioned
integro-algebraic system.

\medskip
\noindent
a) The first step is to find a nowhere vanishing sectionally holomorphic function $X(z)$ on 
$\mathbb{C}\setminus \Real_{\ge 0}$ satisfying the boundary condition
\begin{equation}\label{Xpm}
\frac {X^+(t)}{X^-(t)} = e^{2i \theta(t)}, \quad t>0.
\end{equation}
Since $\theta(t)$ is H\"older on $\Real_{\ge 0}$ and $\theta(t)\sim t^{\alpha-3}$ as $t\to\infty$, such function can be
obtained  by the Sokhotski--Plemelj formula
$$
X(z) =  \exp \left(\frac 1 {\pi} \int_0^\infty \frac{\theta(t)}{t-z}dt\right), \quad z\in \mathbb{C}\setminus \Real_{\ge 0}
$$
and it is unique up to a multiplicative factor $z^\beta$ with an integer $\beta$. Any $\beta\le 0$
will suit our purposes, rendering the functions $S$ and $D$, defined in \eqref{SDdef} below, to be square integrable near the origin. 
This is needed, since the emerging integral equations \eqref{SDHp} can be guaranteed to have the unique solution in $L_2(0,\infty)$. 
The particular choice $\beta=0$ is a matter of convenience:
other values would introduce the multiplicative factor $z^\beta$ in the calculations below, which eventually cancels
out, ultimately leading to the same integro-algebraic system.

Since $\lim_{t\to 0+}\theta(t)=\frac{1-\alpha}{2}\pi>0$,  $X(z)$ satisfies
\begin{equation}
\label{Xgr}
\begin{aligned}
& X(z) \simeq 1 - z^{-1} \nu \, b_\alpha \quad & & \text{as}\ z\to \infty \\
& X(z)\sim z^{\frac{\alpha-1} 2}\quad & &\text{as}\  z\to 0
\end{aligned}
\end{equation}
where $b_\alpha$ is the constant defined in \eqref{thetaint}.
Plugging \eqref{Xpm} into the conditions \eqref{Hp} gives
$$
\begin{aligned}
&
\frac{\Phi_0^+(t)}{X^+(t)} - \frac{\Phi_0^-(t)}{X^-(t)} = 2i  e^{-t} e^{i\theta(t)}\sin\theta(t) \frac{X(-t)}{X^+(t)}\frac{\Phi_1(-t)}{X(-t)}
\\
&
\frac{\Phi_1^+(t)}{X^+(t)} -  \frac{\Phi_1^-(t)}{X^-(t)} = 2i e^{-t}  e^{i\theta(t)}\sin\theta(t) \frac{X(-t)}{X^+(t)}\frac{\Phi_0(-t)}{X(-t)}
\end{aligned}\qquad t>0
$$
and therefore the functions
\begin{equation}\label{SDdef}
\begin{aligned}
& S(z)  := \frac{\Phi_0(z)+\Phi_1(z)}{2X(z)} \\
& D(z)  := \frac{\Phi_0(z)-\Phi_1(z)}{2X(z)}
\end{aligned}
\end{equation}
satisfy the decoupled equations
\begin{equation}
\begin{aligned}\label{SDHp}
&
S^+(t) - S^-(t) \; =\phantom{+} 2ih(t)e^{-t} S(-t)
\\
&
D^+(t)-D^-(t) = - 2i h(t)e^{-t} D(-t)
\end{aligned}\qquad t>0
\end{equation}
where we defined
$$
h(t) := e^{i\theta(t)}\sin \theta(t) \frac {X(-t)}{X^+(t)}.
$$

By the Sokhotski--Plemelj formula
\begin{align*}
\frac {X(-t)}{X^+(t)} =\; &
   \exp \left(\frac 1 {\pi} \int_0^\infty \frac{\theta(s)}{s+t}ds
-\frac 1 {\pi} \dashint_0^\infty \frac{\theta(s)}{s-t}ds-i\theta(t)
\right) =\\
&
 \exp \left(-\frac {2t} {\pi} \dashint_0^\infty \frac{\theta(s)}{s^2-t^2}ds-i\theta(t)
\right)
\end{align*}
where the dash integral stands for the Cauchy principle value.
Integration by parts gives
\begin{equation}
\label{hsq}
\begin{aligned}
h(t) =\; & e^{i\theta(t)}\frac {X(-t)}{X^+(t)}\sin \theta(t)  =
  \exp \left(-\frac {2t} {\pi} \dashint_0^\infty \frac{\theta(s)}{s^2-t^2}ds \right)\sin \theta(t) = \\
 &
  \exp \left(-\frac {1} {\pi} \int_0^\infty \theta'(s)\log \left |\frac{t+s}{t-s} \right|ds \right)\sin \theta(t),
\end{aligned}
\end{equation}
which shows that $h(t)$ is a real valued function, H\"older continuous on $\Real_{\ge 0}$ and
$h(0) :=\sin \theta(0+)=\sin \frac {1-\alpha}2 \pi$.

\medskip
\noindent
b) By the estimates \eqref{grinf}-\eqref{grzero} and \eqref{Xgr}, the functions in the right hand sides of \eqref{SDHp} are
H\"older continuous and vanish as $t\to\infty$. Therefore by the Sokhotski--Plemelj formula
\begin{equation}
\label{SD}
\begin{aligned}
& S(z) = \phantom{+} \frac 1 {\pi } \int_0^\infty \frac{h(t)e^{-t} }{t-z}S(-t) dt+P_1(z)\\
& D(z) = -\frac 1 \pi \int_0^\infty  \frac{h(t)e^{-t}}{t-z} D(-t)dt + P_2(z),
\end{aligned}
\end{equation}
where $P_1(z)$ and $P_2(z)$ are polynomials, to be chosen to match the growth of $S(z)$ and $D(z)$ as $z\to\infty$
(see \eqref{punkt3}(b) in Section \ref{sec-strat}). 

The asymptotic formula $(az+b)/(1-c/z)\simeq a(z+c)+b$ as $z\to\infty$ 
and the estimates \eqref{Xgr} and \eqref{Phi0Phi1} give 
\begin{align*}
& P_1(z)  := - c_2 (z+\nu b_\alpha) +c_1 \\
& P_2(z)  :=  - c_2 (z+\nu b_\alpha) - c_1
\end{align*}
where we defined the constants 
\begin{align*}
& c_1 = -\frac 1 2  \frac{\lambda\Gamma(\alpha)}{c_\alpha}\psi'(1)\\
& c_2 = -\frac 1 2  \frac{\lambda\Gamma(\alpha)}{c_\alpha}\psi(0).
\end{align*}

If we set $z:=-t$ in \eqref{SD} with $t\in \Real_{>0}$,  the following integral equations for $S(-t)$ and $D(-t)$ on the
positive real semiaxis are obtained
\begin{align*}
& S(-t) = \, \phantom{+} \frac 1 {\pi } \int_0^\infty \frac{h(s)e^{-s} }{s+t}S(-s) ds + c_2 (t-\nu b_\alpha) + c_1\\
& D(-t) = -\frac 1 \pi \int_0^\infty  \frac{h(s)e^{-s}}{s+t} D(-s)ds + c_2 (t-\nu b_\alpha) - c_1
\end{align*}
and by linearity
\begin{equation}
\label{SDq}
\begin{aligned}
& S(z\nu) =  c_2 \nu \big(q_+(-z)-b_\alpha p_+(-z)\big)+c_1 p_+(-z)  \\
& D(z\nu) =  c_2 \nu \big( q_-(-z)-b_\alpha p_-(-z)\big)-c_1 p_-(-z).
\end{aligned}
\end{equation}
Combining the definitions \eqref{abpm} and \eqref{SDdef} with \eqref{SDq} gives
\begin{equation}
\label{PhPhA}
\begin{aligned}
& \Phi_0(z\nu) =  
c_2 \nu X_0(z)\big(b_+(-z)-b_\alpha a_+(-z) \big)+c_1 X_0(z) a_-(-z)
\\
& \Phi_1(z\nu) = 
c_2 \nu  X_0(z) \big(b_-(-z)-b_\alpha a_-(-z)\big)+c_1 X_0(z) a_+(-z).
\end{aligned}
\end{equation}

If we now plug  \eqref{PhPhA} into condition \eqref{cond1}, the linear equation 
$$
c_2 \nu \xi +c_1 \eta =0
$$
with respect to the real valued coefficients $c_1$ and $c_2$ is obtained. 
This equation has nontrivial solutions if and only if \eqref{alg} holds, in which case $c_1 =-c_2 \nu \xi/\eta$.
Plugging this back into \eqref{PhPhA} and omitting multiplicative constants gives \eqref{Phiz}.
To recap, starting with an eigenvalue $\lambda$ and the corresponding eigenfunction $\varphi$ we have constructed a
solution to the integro-algebraic system defined in the lemma.

The other direction is shown by reverting all the operations: starting with $(p_\pm, q_\pm, \nu)$ solving the system
\eqref{pq} and \eqref{alg}, a pair of sectionally holomorphic functions $\Phi_0(z)$ and $\Phi_1(z)$ is constructed, so that
the conditions \eqref{cond1} and \eqref{drem} are met. Consequently, inverting the Laplace  transform, a function
$\varphi$ is obtained, so that the pair $(\lambda, \varphi)$ with $\lambda>0$ defined by the formula \eqref{nufla}, solves
the eigenproblem \eqref{eig}.

\end{proof}

\begin{rem}
Note that $b_+(z)-2z\to 0$ and  $a_+(z) \to 2$, while $b_-(z)$ and $a_-(z)$ vanish as $z\to \infty$.
Therefore in view of \eqref{Phi0Phi1} and \eqref{Phiz}, we also have
\begin{equation}
\label{intphilim}
\int_0^1 \varphi(x)dx =  \psi(0)=\frac{c_\alpha}{\lambda \Gamma(\alpha)}\lim_{z\to\infty} z^{-1} \Phi_0(z) =
 -\frac{2 c_\alpha}{\lambda \Gamma(\alpha)} \nu^{-1}
\end{equation}
and
\begin{equation}\label{philim}
\varphi(1)=-\psi'(1)=\frac{c_\alpha}{\lambda \Gamma(\alpha)}\lim_{z\to\infty}  \Phi_1(z) = -\frac{2c_\alpha}{\lambda \Gamma(\alpha)} \frac{\xi}{\eta}
\end{equation}
These formulas will be used to find the precise asymptotics for both quantities.
\end{rem}

The value of $X_0(i)$ for the function $X_0(z)$ defined in \eqref{X0z}, can be found in a closed form and will be useful in calculations to follow:

\begin{lem} \label{lemX0}
$$
\arg \left\{X_0(i)\right\}= \frac {1-\alpha}{ 8}  \pi \quad\text{and} \quad |X_0(i)| = \sqrt{\frac{3-\alpha }{ 2  }}.
$$
\end{lem}

\begin{proof}
Define
$$
\Gamma_0(z)=\frac 1 {\pi} \int_0^\infty \frac{\theta_0(t)}{  t-  z}d  u
$$
and note that
$$
\Im \{\Gamma_0(i)\} =\frac 1 {\pi} \int_0^\infty \frac{\theta_0(t)}{  t^2+ 1}d  t=
\frac 1 {\pi} \int_0^1 \frac{\theta_0(u)+\theta_0(u^{-1})}{  u^2+ 1}d  u.
$$
By \eqref{nufla} and  \eqref{Lfla}
\begin{align*}
&
 \theta_0(u) + \theta_0(u^{-1}) =
\arg\Big\{\Lambda^{+} (u\nu)\Lambda^{+} (u^{-1}\nu)\Big\} =  \\
&
\arg\Big\{\left( 1+  u^{\alpha-3}  i e^{-  \frac \pi  2 i \alpha}\right)
\left( 1+  u^{3-\alpha}  i e^{-  \frac \pi  2 i \alpha}\right)\Big\} = \\
&
\arg\Big\{ie^{-  \frac \pi  2 i \alpha}\left( 1-u^{3-\alpha}i e^{   \frac \pi  2 i \alpha} \right)
\left( 1+  u^{3-\alpha}  i e^{-  \frac \pi  2 i \alpha}\right)\Big\} = \\
&
\arg\Big\{i e^{-  \frac \pi  2 i \alpha}  \Big\} = \frac {1- \alpha} 2  \pi = \theta_0(0+)
\end{align*}
and  $\theta_0(u) +\theta_0(u^{-1})=\theta_0(0+)$ by continuity. Hence
$$
\Im \{\Gamma_0(i)\} = \frac {\theta_0(0+)} {\pi} \int_0^1 \frac{1}{  u^2+ 1}d  u=\frac {\theta_0(0+)} {4}
$$
and
$
\arg\left\{X_0(i)\right\}=\arg\left\{e^{i\Im\{\Gamma_0(i)\}}\right\}= \theta_0(0+)/4
$
as claimed.

\medskip

Let us now compute the absolute value of $X_0(i)$. To this end note that
$$
\log X(z)=\frac 1 {\pi} \int_0^\infty \frac{\theta(t)}{t-z}dt=\frac 1 {2\pi i} \int_0^\infty \frac{\log \Lambda^+(t)/\Lambda^-(t) }{t-z}dt
$$
and by  \eqref{prop}
\begin{align*}
\log X (-z) = & \frac 1 {2\pi i} \int_0^\infty \frac{\log \Lambda^+(t)/\Lambda^-(t) }{t+z}dt= \\
&
\frac 1 {2\pi i} \int_{-\infty}^0 \frac{\log \Lambda^+(-t)/\Lambda^-(-t) }{-t+z}dt =
\frac 1 {2\pi i} \int_{-\infty}^0 \frac{\log \Lambda^+(t)/\Lambda^-(t) }{t-z}dt.
\end{align*}
Therefore
$$
\log X(z)+\log X(-z)=\frac 1 {2\pi i} \int_{-\infty}^\infty \frac{\log \Lambda^+(t)/\Lambda^-(t) }{t-z}dt.
$$
The function
$$
\widetilde \Lambda(z):= \frac {c_\alpha} {\lambda \Gamma (\alpha)} \frac{z\Lambda(z)}{z^2-z_0^2}
$$
does not vanish on the cut plane and satisfies $\widetilde \Lambda^+(t)/\widetilde \Lambda^-(t)=\Lambda^+(t)/\Lambda^-(t)$ and $\widetilde \Lambda^\pm(t)\to 1$ as $t\to\pm \infty$.
Thus
$$
\log X(z)+\log X(-z)=\frac 1 {2\pi i} \int_{-\infty}^\infty \frac{\log \widetilde \Lambda^+(t)  }{t-z}dt
-\frac 1 {2\pi i} \int_{-\infty}^\infty \frac{\log  \widetilde \Lambda^-(t) }{t-z}dt,
$$
where both integrals are well defined and can be evaluated by integrating the function $g(\xi):=\log \widetilde \Lambda(\xi)/(\xi-z)$ over half circular
contours in the upper and lower half-planes. For $z$ with $\mathrm{Im}(z)> 0$,
$$
\frac 1 {2\pi i} \int_{-\infty}^\infty \frac{\log \widetilde\Lambda^-(t)  }{t-z}dt=0
$$
and
$$
\frac 1 {2\pi i} \int_{-\infty}^\infty \frac{\log \widetilde \Lambda^+(t)  }{t-z}dt = \Res(g;z) =   \log \widetilde \Lambda(z).
$$
Similarly, for $z$ with $\mathrm{Im}(z)<0$,
$$
\frac 1 {2\pi i} \int_{-\infty}^\infty \frac{\log \widetilde \Lambda^+(t)  }{t-z}dt=0
$$
and
$$
\frac 1 {2\pi i} \int_{-\infty}^\infty \frac{\log  \widetilde \Lambda^-(t) }{t-z}dt = -\Res(g;z) = -\log \widetilde \Lambda(z).
$$
Thus we obtain
$
\log X(z)+\log X(-z) = \log  \widetilde \Lambda(z)
$
or
$$
X(z)X(-z)=  \frac {c_\alpha} {\lambda \Gamma (\alpha)} \frac{z\Lambda(z)}{z^2-z_0^2}.
$$
Taking $z\to z_0$  gives
$$
X(z_0)X(-z_0)= \frac {c_\alpha} {\lambda \Gamma (\alpha)}\frac{1 }{ 2  }\Lambda'(z_0).
$$
Differentiating the expression in \eqref{Lz}, we get
$$
\Lambda'(z_0) = \frac{\lambda \Gamma(\alpha)}{c_\alpha} + (\alpha-2) (i\nu)^{\alpha-3}\frac{\pi i}{\cos \frac \pi 2 \alpha}e^{-\frac\pi 2 \alpha i}
=
\frac{\lambda \Gamma(\alpha)}{c_\alpha} (3-\alpha)
$$
and consequently
$$
|X_0(i)| = \sqrt{X_0(i)X_0(-i)}=\sqrt{\frac {3-\alpha} 2}.
$$

\end{proof}

\subsubsection{Properties of the integro-algebraic system}
So far we established only correspondence between the solutions of the eigenproblem and the integro-algebraic
system of equations. Solvability of this system, proved in Lemma \ref{cor} below, relies essentially on the
properties of the operator
\begin{equation}
\label{Aop}
(A f)(t) := \frac 1 \pi \int_0^\infty  \frac{h_0(s)e^{-\nu s}}{s+t}f(s)ds,
\end{equation}
verified in this subsection.

\begin{lem}\label{lem-A}
The operator $A$ is a contraction on $L^2(0,\infty)$ for all $\nu$ large enough. More precisely, for any $\alpha_0\in (0,1)$
there exist an $\eps>0$ and a constant $\nu'>0$, such that $\|A\| \le 1-\eps$ for all  $\nu\ge \nu'$ and all $\alpha \in [\alpha_0, 1]$.
\end{lem}

\begin{proof}
The function $h_0(t)  := h(t\nu)$ does not depend on $\nu$, is continuous on $(0,\infty)$
and satisfies $\lim_{t\to\infty}h_0(t)=0$ and $h_0(0+) = \sin \frac {1-\alpha} 2\pi\in (0,1)$.
Moreover, a direct calculation shows that $\theta_0(t)$ decreases and
$|\theta_0'(t)|/3\le   t^{-3} \wedge 1$ and hence by \eqref{hsq}
$$
|h_0(t)| \le g_0(t) \big|\sin \theta_0(t)\big| =:\widetilde h_0(t)
$$
where the function
$$
g_0(t):=\exp \left( 3 \int_0^\infty \big(s^{-3} \wedge 1\big) \log \left |\frac{t+s}{t-s} \right|ds \right)
$$
is continuous and bounded with $g_0(0)=1$ and $\lim_{t\to\infty}g_0(t)=1$. In particular, both norms $\|h_0\|_\infty$ and
$\|\widetilde h_0\|_\infty$ are uniformly bounded over $\alpha\in (0,1]$.

Since $|\widetilde h_0(0)|\le \sin \frac {1-\alpha_0} 2\pi=:1-2\eps<1$, there exists a $\delta>0$ so that $\widetilde h_0(t)\le 1-\eps$ for all $t\le \delta$ and therefore
$$
|h_0(t)| e^{-t\nu }\le |\widetilde h_0(t)| e^{-t\nu } \le (1-\eps) \one{t\le \delta} + \|\widetilde h\|_\infty e^{-\nu \delta}\one{t>\delta} \le 1-\eps,\quad t\ge 0,
$$
for all $\nu\ge \nu':= \frac 1 \delta \log \big(\|\widetilde h\|_\infty/(1-\eps)\big)$ and all $\alpha\in [\alpha_0,1]$.
Consequently, for any $g,f\in L^2(0,\infty)$ and all such $\nu$
\begin{align*}
&
|\langle g, A f \rangle| \le
\frac {1-\eps} \pi
\int_0^\infty \int_0^\infty |g(t)| |f(s)|  \frac{1}{s+t}  dsdt \le \\
&
\frac {1-\eps} \pi
\left(\int_0^\infty |f(s)|^2 \int_0^\infty     \frac{(s/t)^{\frac 1 2} }{ s+t }   dtds  \right)^{\frac 1 2}
\left(\int_0^\infty |g(t)|^2 \int_0^\infty   \frac{(t/s)^{\frac 1 2}}{ s+t}     ds dt\right)^{\frac 1 2} \le
(1-\eps)
\|f\|\|g\|,
\end{align*}
where the last inequality holds, since
$$
\int_0^\infty     \frac{(s/t)^{\frac 1 2} }{ s+t }   dt = \int_0^\infty     \frac{u^{-\frac 1 2} }{ 1+u }   du = \pi.
$$
The claim now follows, since
$
\|A f\|^2= \langle Af, Af\rangle  \le (1-\eps)\|f\|\|Af\|.
$
\end{proof}

The following lemma gathers several useful estimates:
\begin{lem}\label{lem-bnds}
For any $\alpha_0\in (0,1)$ there exist constants $\nu'$ and $C$, such that for all $\nu\ge \nu'$ and all $\alpha \in [\alpha_0, 1]$,
\begin{equation}
\label{AAi}
\begin{aligned}
& \big|a_-(\pm i)\big| \le C  \nu^{-1},   \quad  \big|a_+(\pm i)-2| \le C \nu^{-1} \\
& \big|b_-(\pm i)\big| \le C  \nu^{-2},  \quad  \big|b_+(\pm i)\mp 2i  | \le C \nu^{-2} \\
\end{aligned}
\end{equation}
and for all $\tau>0$
\begin{equation}
\label{AA}
\begin{aligned}
& \big|a_-(\tau)\big| \le C  \nu^{-1} \tau^{-1},  \quad \big|a_+(\tau)-2| \le C  \nu^{-1} \tau^{-1}\\
& \big|b_-(\tau)\big| \le C    \nu^{-2} \tau^{-1}, \quad \big|b_+(\tau)-2\tau  | \le C    \nu^{-2} \tau^{-1}
\end{aligned}
\end{equation}
\end{lem}

\begin{proof}
The operator $A$ maps functions $1$ and $t$ into $L^2(0,\infty)$ and since, by Lemma \ref{lem-A}, for any
$\alpha_0\in (0,1)$,  $A$ is a contraction for all $\nu$ large enough,
equations \eqref{pq} have unique solutions, satisfying $\big(p_\pm (t)-1; t\ge 0\big)\in L^2(0,\infty)$ and
$\big(q_\pm (t)-t; t\ge 0\big)\in L^2(0,\infty)$.
Let us estimate the norms of the functions $p_\pm(t)-1$ and $q_\pm(t)-t$. To this end, write
\begin{equation}
\label{canwrite}
(q_{\pm }(\tau) -\tau) = \pm \frac 1 {\pi } \int_0^\infty \frac{h_0(u)e^{-u\nu } }{u +\tau }\big(q_\pm (u)-u\big) du
\pm \frac 1 {\pi } \int_0^\infty \frac{h_0(u)e^{-u\nu } }{u +\tau }u du.
\end{equation}
By the generalized Minkowski inequality, the $L^2(0,\infty)$ norm of the last term on the right satisfies
\begin{align*}
&
\left(
\int_0^\infty\left(
\int_0^\infty \frac 1 {\pi } \frac{h_0(u)e^{-u\nu } }{u +\tau }u du
\right)^2 d\tau
\right)^{\frac 1 2} \le
\int_0^\infty\left(
\int_0^\infty \left(\frac 1 {\pi } \frac{h_0(u)e^{-u\nu } }{u +\tau }u \right)^2 d\tau
\right)^{\frac 1 2}du \le \\
&
\frac 1 {\pi } \|h_0\|_\infty
\int_0^\infty u e^{-u\nu }\left(
\int_0^\infty \frac{1 }{(u +\tau)^2}  d\tau
\right)^{\frac 1 2}du=
\frac 1 {\pi } \|h_0\|_\infty
\int_0^\infty \sqrt{u} e^{-u\nu }du  
\le \|h_0\|_\infty \nu^{-\frac 32}.
\end{align*}
The integral operator in   \eqref{canwrite} is contracting in $L^2(0,\infty)$ with $\|A\|<1-\eps$, where $\eps$ is independent of $\nu$.
Therefore, the norm of $q_{\pm }(\tau) -\tau$ is bounded by $\eps^{-1}\|h_0\|_\infty \nu^{-\frac 3 2 }$. Similarly, the norm of $p_{\pm }(\tau) -1$ is bounded by $ \eps^{-1}\|h_0\|_\infty \nu^{-\frac 1 2 }$.
Using these estimates, we obtain
\begin{align*}
&
\big|q_{\pm }(\tau) - \tau\big| = \left|
\frac 1 {\pi } \int_0^\infty \frac{h_0(u)e^{-u\nu} }{u +\tau }q_\pm (u) du
\right|  \le  \\
&
\frac 1 {\pi } \int_0^\infty \frac{|h_0(u)|e^{-u\nu} }{u +\tau }\big|q_\pm (u)-u\big| du
+\frac 1 {\pi } \int_0^\infty \frac{|h_0(u)|e^{-u\nu} }{u +\tau }u du \le  \\
&
 \|h_0\|_\infty\frac 1 {\tau}\left(\int_0^\infty  e^{-2 u\nu}    du \right)^{1/2}\left(\int_0^\infty    \big(q_\pm (u)-u\big)^2 du\right)^{1/2}
+ \|h_0\|_\infty \frac 1 {\tau}\int_0^\infty   e^{-u\nu}  u du \le  \frac C {\nu^2 \tau},
\end{align*}
and, similarly,
$
\big|p_{\pm }(\tau) - 1\big| \le C \nu^{-1}\tau^{-1}.
$
Here constant $C$ depends only on  $\|h_0\|_\infty$ and $\eps$ and therefore only on $\alpha_0$ (see Lemma \ref{lem-A}).
The  bounds claimed in \eqref{AA} now follow from the definition of  $a_\pm(\tau)$ and $b_\pm(\tau)$. The estimates \eqref{AAi} are obtained analogously.
\end{proof}

\subsubsection{Inversion of the Laplace transform}

\begin{figure}
\input{contour.tex}
  \caption{\label{fig2} Integration contour, used for the Laplace transform inversion: the outer circular arcs are of radius $R$ and the half circles around the poles
  have radius $\delta$ }
\end{figure}
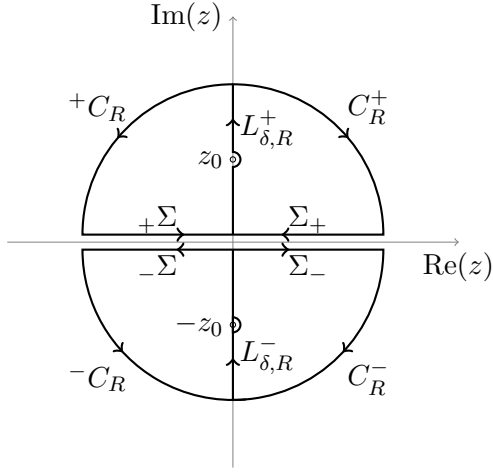

The expression for the eigenfunctions is obtained by inverting the Laplace transform, given by the formula \eqref{main}:

\begin{lem}
Let $\big(\Phi_0, \Phi_1, \nu\big)$ satisfy the integro-algebraic system introduced in Lemma \ref{lemia}, then the function
\begin{multline}\label{phinu}
\varphi(x) =
- \frac 2{3-\alpha} \mathrm{Re}\Big\{e^{i\nu x} \Phi_0(i\nu) \Big\} \\
+ \frac 1 { \pi }  \int_0^{\infty}    \frac{ \sin \theta_0(u)}{\gamma_0(u)} \left(e^{-u\nu (1-x)}\Phi_1(-u\nu )-e^{-u\nu x} \Phi_0(-u\nu )\right)du
\end{multline}
where
$
\gamma_0(u) := \frac{\cos \frac \pi 2 \alpha}{\pi} \nu^{2-\alpha} |\Lambda^+(\nu u )|
$
and $\lambda$ defined by the formula \eqref{nufla} solve the eigenproblem \eqref{eig}.
Moreover,
\begin{equation}\label{intphi}
\int_0^1 \varphi(x)dx =   -2\nu^{-1} \quad \text{and}\quad
\varphi(1)= -2\frac{\xi}{\eta}.
\end{equation}

\end{lem}

\begin{proof}
For any $\nu$ solving the integro-algebraic system \eqref{pq} and \eqref{alg},  $\widehat \varphi(z)$ is an entire function and we can carry out the inversion, integrating on the imaginary axis:
\begin{equation}\label{phix}
\begin{aligned}
\varphi(x)  = & -\frac 1 {2\pi i}\lim_{R\to\infty} \int_{-iR}^{iR}
\left(\frac{e^{-z}\Phi_1(-z)}{\Lambda(z)}+\frac{\Phi_0(z)}{\Lambda(z)}- \widehat \varphi(0)\right) e^{zx}dz\\
& =: -\frac 1 {2\pi i}  \int_{-i\infty}^{i\infty} \big(f_1(z)+f_2(z)\big)dz,
\end{aligned}
\end{equation}
where we defined
$$
f_1(z)= e^{z(x-1)} \frac{\Phi_1(-z)}{\Lambda(z)}
\quad \text{and}\quad
f_2(z)=
 e^{zx}  \left(\frac{\Phi_0(z)}{\Lambda(z)}- \widehat \varphi(0)\right).
$$
Note that for all $x\in (0,1)$, the functions $f_1(z)$ and $f_2(z)$ decrease exponentially as $z\to \infty$, when $\mathrm{Re}(z)>0$ and $\mathrm{Re}(z)<0$
respectively. Hence integrating $f_1(z)$ over the contours, depicted in Figure \ref{fig2}, which lie in the right half plane, we get
\begin{align*}
&
\int_{L^+_{\delta,R}}f_1(z)dz + \int_{C_R^+}f_1(z)dz + \int_{\Sigma_+}f_1(z)dz=0 \\
&
\int_{L^-_{\delta,R}}f_1(z)dz + \int_{C_R^-}f_1(z)dz + \int_{\Sigma_-}f_1(z)dz=0
\end{align*}
Similarly, integrating $f_2(z)$ over the contours in the left half plane  gives
\begin{align*}
& \int_{L^+_{\delta,R}}f_2(z)dz + \int_{{^+}C_R}f_2(z)dz+\int_{{_+}\Sigma}f_2(z)dz = 2\pi i\; \Res(f_2,z_0)\\
& \int_{L^-_{\delta,R}}f_2(z)dz + \int_{{^-}C_R}f_2(z)dz+\int_{{_-}\Sigma}f_2(z)dz = 2\pi i\; \Res(f_2,-z_0)
\end{align*}
Summing up these equations, taking $R\to\infty$ and $\delta\to 0$ and using Jordan's lemma we get
\begin{align*}
\int_{-i\infty}^{i\infty} \big(f_1(z)+f_2(dz)\big)dz = &
%
%
\int_0^{\infty} \big(f_1^+(t)-f_1^-(t)\big)dt  + \int_{0}^\infty\big(f_2^-(-t)-f_2^+(-t)\big)dt \\
&
+2\pi i\; \Res\big(f_2,z_0\big)+2\pi i\; \Res\big(f_2,-z_0\big)
\end{align*}
By the properties \eqref{conjp}-\eqref{absL}, for $t>0$
$$
f_1^+(t)-f_1^-(t) = e^{t(x-1)}\Phi_1(-t)\left( \frac{1}{\Lambda^+(t)}-  \frac{1}{\Lambda^-(t)}\right)=
-e^{t(x-1)}\Phi_1(-t) \frac{2i\sin \theta(t)}{\gamma(t)},
$$
and
$$
f_2^-(-t)-f_2^+(-t)=e^{-tx} \Phi_0(-t)\left( \frac{1}{\Lambda^-(-t)} -   \frac{1}{\Lambda^+(-t)}\right)=
e^{-tx} \Phi_0(-t)\frac{2i\sin \theta(t)}{\gamma(t)},
$$
where $\gamma(t)=|\Lambda^+(t)|$.
Plugging all these expressions into \eqref{phix} gives
$$
\varphi(x)  =
\frac 1 { \pi } \int_0^{\infty}    \frac{ \sin \theta(t)}{\gamma(t)} \left(e^{-t(1-x)}\Phi_1(-t)-e^{-tx} \Phi_0(-t)\right)dt
-\Res\big(f_2,z_0\big)-\Res\big(f_2,-z_0\big).
$$
The integral term on the right hand side can be rewritten as
\begin{multline*}
\frac 1 { \pi } \int_0^{\infty}    \frac{ \sin \theta(t)}{\gamma(t)} \left(e^{-t(1-x)}\Phi_1(-t)-e^{-tx} \Phi_0(-t)\right)dt = \\
 \frac{\cos \frac \pi 2 \alpha}{\pi} \nu^{3-\alpha } \frac 1 { \pi }  \int_0^{\infty}    \frac{ \sin \theta_0(u)}{\gamma_0(u)} \left(e^{-u\nu (1-x)}\Phi_1(-u\nu )-e^{-u\nu x} \Phi_0(-u\nu )\right)du,
\end{multline*}
where we defined $\gamma_0(u)$
$$
\gamma(u\nu)=|\Lambda^+(\nu u )| =
\frac{\pi} {\cos \frac \pi 2 \alpha}\nu^{\alpha-2}
\left|
  u  +  u^{\alpha-2}      e^{\frac {\pi} 2 i(1-\alpha)}
\right| := \frac{\pi} {\cos \frac \pi 2 \alpha}\nu^{\alpha-2} \gamma_0(u).
$$
Let us now compute the residues, using the  formula \eqref{Lz}:
\begin{align*}
&
\Res\big(f_2,z_0\big) =
\frac{e^{z_0x} \Phi_0(z_0)}{\Lambda'(z_0)} =
e^{i\nu x} \Phi_0(i\nu)\nu^{3-\alpha} \frac {\cos \frac \pi 2 \alpha}{\pi} \frac 1{3-\alpha}
\\
&
\Res\big(f_2,-z_0\big) =
e^{-z_0x} \Phi_0(-z_0) \frac{1}{\Lambda'(-z_0)} = e^{-i\nu x} \Phi_0(-i\nu )\nu^{3-\alpha}
 \frac{\cos \frac \pi 2 \alpha}{\pi} \frac 1{3  -\alpha}.
\end{align*}
Hence 
\begin{align*}
\Res\big(f_2,z_0\big)+\Res\big(f_2,-z_0\big) =\;
&
\nu^{3-\alpha} \frac {\cos \frac \pi 2 \alpha}{\pi} \frac 1{3-\alpha}\Big(e^{i\nu x} \Phi_0(i\nu)+e^{-i\nu x} \Phi_0(-i\nu )\Big)= \\
&
\nu^{3-\alpha} \frac {\cos \frac \pi 2 \alpha}{\pi} \frac 2{3-\alpha} \mathrm{Re}\Big\{e^{i\nu x} \Phi_0(i\nu) \Big\}.
\end{align*}
Assembling all parts together, we obtain \eqref{phinu}. The formulas \eqref{intphi} follow from \eqref{intphilim}-\eqref{philim} and the definition \eqref{nufla}.

\end{proof}

\subsubsection{Asymptotic analysis}

The following lemma derives asymptotics for the algebraic part of solutions of the system from Lemma \ref{lemia}:

\begin{lem}\label{cor}
The integro-algebraic system \eqref{pq} and \eqref{alg} has countably many solutions, which can be enumerated so that
\begin{equation}
\label{ournu}
\nu_n = \pi \Big(n+\frac 1 2\Big) -\frac{1-\alpha}{4}\pi + \arcsin\frac{b_\alpha}{\sqrt{1+b_\alpha^2}} + n^{-1} r_n(\alpha), \quad n\to\infty,
\end{equation}
where $b_\alpha$ is defined in \eqref{thetaint} and the residual $r_n(\alpha)$ is bounded, uniformly in $n$ and $\alpha\in [\alpha_0,1]$ for any $\alpha_0\in (0,1)$.
\end{lem}

\begin{proof}
Let us rewrite \eqref{xieta} as
\begin{align*}
&
\xi :=   e^{ i\nu/2} X_0(i)\Big(b_+(-i)-b_\alpha a_+(-i) \Big)
\left(
1+e^{-i\nu }\frac{ X_0(-i) \big(b_-(i)-b_\alpha a_-(i)\big)}
{  X_0(i)\big(b_+(-i)-b_\alpha a_+(-i) \big)} 
\right)
\\
&
\overline{\eta}:=  e^{i\nu/2}X_0(i) a_+(-i)\left(1+e^{-i\nu}\frac{X_0(-i) a_-(i) } {X_0(i) a_+(-i)} \right)
\end{align*}
Then  by Lemma  \ref{lem-bnds} and Lemma \ref{lemX0}
$$
\xi \overline{\eta} =  4 \frac{3-\alpha }{ 2  } \sqrt{1+b_\alpha^2} \exp \left\{i\left(\nu  +  \frac {1-\alpha}{ 4}  \pi -\pi + \arg\{i+b_\alpha\} \right)\right\}\Big(1+R(\nu)\Big)
$$
where $|R(\nu)|\le C_1\nu^{-1}$ with a constant $C_1$, depending only on $\alpha_0$. Consequently  the equation \eqref{alg} becomes
\begin{equation}
\label{nueqn}
\nu  +  \frac {1-\alpha}{ 4}  \pi  - \pi + \arg\{i+b_\alpha\}   - \pi n  + \arctan \frac{\mathrm{Im}\{R(\nu)\}}{1+\mathrm{Re}\{R(\nu)\}}=0, \quad n\in \mathbb{Z}.
\end{equation}
This defines a particular enumeration of {\em all possible} solutions $\big(p_{\pm n}, q_{\pm n},\nu_n\big)$ of the integro-algebraic system.

Calculations similar to those in the proof of Lemma \ref{lem-bnds} also show that  $|R'(\nu)|\le C_2 \nu^{-1}$ with a
constant $C_2$ and in view of Lemma \ref{lem-A}, the integro-algebraic system is contracting for all positive $n$ large enough, so that the unique solution for such $n$ is given by the fixed-point iterations.
The asymptotics \eqref{ournu} follows from  \eqref{nueqn}.
\end{proof}

Let us now derive asymptotic approximation for the eigenfunctions.
To this end, under the enumeration fixed in Lemma \ref{cor},
$$
\frac{\xi}{\eta}=\frac{\xi\bar \eta}{|\eta|^2} = (-1)^n \frac{|\xi|}{|\eta|} = (-1)^n \sqrt{1+b_\alpha^2}\big(1+O(\nu_n^{-1})\big),
\quad n\to \infty
$$
where the second equality holds by condition \eqref{alg} and the choice of $n$ in \eqref{nueqn}
and the asymptotics is obtained, applying the estimates \eqref{AAi} to the definitions  \eqref{xieta}.
Plugging \eqref{Phiz} into \eqref{phinu} and using the bounds from Lemma \ref{lem-bnds}, we obtain
\begin{align*}
&
\varphi_n(x)
=\; 
  \frac 4{3-\alpha} \Re\Big\{e^{i\nu_n x} X_0(i) (i+b_\alpha) \Big\} \\
& +  \frac 2 { \pi }  \int_0^{\infty}    \frac{ \sin \theta_0(u)}{\gamma_0(u)} X_0(-u)
\Big(
-e^{-u\nu_n x} (u-b_\alpha)-
(-1)^{n} \sqrt{1+b_\alpha^2}  e^{-u\nu_n (1-x)}\
\Big)du + \nu_n^{-1} \widetilde r_n(x),
\end{align*}
with a uniformly bounded residual $\widetilde r_n(x)$.
By Lemma \ref{lemX0}
$$
\Re\Big\{e^{i\nu_n x} X_0(i) (i+b_\alpha) \Big\}= 
\sqrt{1+b_\alpha^2}\sqrt{\frac{3-\alpha }{ 2  }}
\cos\Big(\nu_n x+\frac {1-\alpha}{ 8}  \pi+\arg\{i+b_\alpha\}\Big)
$$
and therefore
\begin{multline}
\label{varphin}
\varphi_n(x)
\propto 
\sqrt 2 \cos\Big(\nu_n x+\frac {1-\alpha}{ 8}  \pi+\arg\{i+b_\alpha\}\Big) \\
 + \frac {\sqrt{3-\alpha}} { \pi }  \int_0^{\infty}    \rho_0(u)
\Big(
-e^{-u\nu_n x} \frac{u-b_\alpha}{ \sqrt{1+b_\alpha^2}}-
(-1)^{n}   e^{-u\nu_n (1-x)}\
\Big)du + \nu_n^{-1} \widetilde r_n(x),
\end{multline}
where we defined
$$
\rho_0(u)= \dfrac{ \sin \theta_0(u)}{\gamma_0(u)} X_0(-u)  >0.
$$
Note that the integral in \eqref{varphin} is well defined for all $x\in [0,1]$ and $\alpha\in (0,1)$ since
$\rho_0(u) \sim u^{\frac 1 2-\frac \alpha 2 }$ as $u\to 0$ and $\rho_0(u)\sim u^{\alpha-2}$ as $u\to\infty$.

The $L^2(0,1)$ norm of the boundary layer term in \eqref{varphin} vanishes asymptotically
\begin{align*}
&
\int_0^1 \left(\int_0^{\infty}     \rho_0(u)\left(e^{-u\nu_n (1-x)} (-1)^{n+1} +e^{-u\nu_n x}   \frac{u-b_\alpha}{ \sqrt{1+b_\alpha^2}}\right)du\right)^2 dx \le \\
&
\left(\int_0^{\infty}      \rho_0(u)
\left(\int_0^1  \left(e^{-u\nu_n (1-x)}   +e^{-u\nu_n x}    \frac{u-b_\alpha}{ \sqrt{1+b_\alpha^2}}\right)^2dx\right)^{1/2}  du \right)^2 \le \\
&
4\nu_n^{-1}\left(\int_0^{\infty}      \rho_0(u)
\sqrt{u^{-1}+u\,} du \right)^2,
\end{align*}
and hence the norms of eigenfunctions in \eqref{phin} are asymptotic to 1 as $n\to\infty$.
The corresponding normalization of the expressions in \eqref{intphi} yields the formulas
\begin{equation}\label{psipsi}
\int_0^1 \varphi_n(x)dx  = - \sqrt{\frac{ 3-\alpha}{1+b_\alpha^2} }\; \nu_n^{-1} \quad \text{and}\quad
\varphi_n(1) =   \sqrt{ 3-\alpha }\; (-1)^{n-1}  \big( 1 +O(\nu_n^{-1})\big).
\end{equation}

\subsubsection{Enumeration alignment}\label{sec-cal}

Let us denote by $\tilde \lambda_n$, $n=1,2,...$ the sequence of the eigenvalues, put in the non-increasing order, that is,  $\tilde \lambda_1\ge \tilde \lambda_{2}\ge ...$,
or equivalently, $\tilde \nu_{1}\le \tilde \nu_2\le ...$. This enumeration, to which we refer as {\em natural},  does not necessarily coincide
with the enumeration, introduced in Lemma \ref{cor}.

Lemma \ref{cor} guarantees existence of $\nu_n$ only for all $n$ large enough.
In our enumeration $\nu_n$ ceases to be positive if $n$ is too small and form a strictly increasing sequence for all
$n$ large enough. Moreover, since the only accumulation point of the eigenvalues is the origin, there can be only finitely many
$\nu_n$'s on any bounded interval.
Therefore the two enumerations can differ only by a constant shift, possibly dependent
on $\alpha$, that is, $\tilde \lambda_n := \lambda_{n+k_\alpha}$ with an integer valued function $\alpha\mapsto k_\alpha$.
Hence
$$
\tilde \nu_n = \nu_{n+k_\alpha} = \pi \Big(n+\frac 1 2\Big) -\frac{1-\alpha}{4}\pi + \frac \pi 2 - \arg\{i+b_\alpha\} + k_\alpha \pi+ n^{-1} r_n(\alpha),\quad n\to\infty,
$$
where, by Lemma \ref{cor}, the residual term $r_n(\alpha)$ is bounded, uniformly in $\alpha\in [\alpha_0,1]$ and $n$ for any $\alpha_0\in (0,1)$.
Thus the leading order asymptotic term is not affected under passing to the natural enumeration, while the second order term
may change by an integer multiple of $\pi$.

This uncertainty can be eliminated using continuity of the spectrum with respect to the parameter $\alpha$
and the exact formula \eqref{Bmlambda}, corresponding to $\alpha=1$.
Taking into account the definition of   $b_\alpha$ in \eqref{thetaint},
for a fixed $\alpha'\in (\alpha_0, 1]$ 
$$
|k_{\alpha'}-k_\alpha| \le  |\tilde \nu_n (\alpha') -\tilde \nu_n (\alpha)| + 2|\alpha'-\alpha| +   n^{-1} \big(|r_n(\alpha)|+|r_n(\alpha')|\big).
$$
The last term on the right hand side can be made smaller than any $\epsilon>0$ uniformly over $\alpha,\alpha'\in [\alpha_0, 1]$ by choosing $n$ large enough.
By Theorem 1.14 in \cite{Kato} (see also Theorem 3 in \cite{R51}), $\tilde \nu_n(\alpha)$ is continuous for any fixed $n$ and therefore each of the other two terms
is smaller than $\epsilon$ on an open neighborhood of $\alpha'$. Therefore on this neighborhood $|k_{\alpha'}-k_\alpha|\le 3\epsilon$. Consequently $k_\alpha$ cannot have jumps
on the interval  $(\alpha_0,1]$, and in fact, on $(0,1]$ since $\alpha_0\in (0,1)$. Thus $k_\alpha$ is constant with respect
to $\alpha$ and therefore its value must equal $k_1$. In view of \eqref{Bmlambda}, $k_1=-1$
and hence the two enumerations are related by the formula $\tilde \nu_n=\nu_{n-1}$, $n=1,2,...$, that is,

$$
\tilde{\nu}_n = \pi \Big(n-\frac 1 2\Big) -\frac{1-\alpha}{4}\pi + \frac \pi 2 - \arg\{i+b_\alpha\} 
  + O(n^{-1}), \quad n\to\infty.
$$

This gives the eigenvalues asymptotics, claimed in \eqref{lambda}, after replacing $\alpha$ with $2-2H$
and omitting tilde from the notation:
\begin{align*}
\lambda_n  =
&
\nu_n^{\alpha-3} \frac{c_\alpha} {  \Gamma (\alpha)} \frac{\pi}{\cos \frac \pi 2 \alpha}=
%
\Gamma(2H+1) \sin (\pi H)\nu_n^{-1-2H}
\end{align*}
where we used the well known properties of the $\Gamma$ function.

\medskip

Similarly, under the natural enumeration, the formula \eqref{varphin} becomes
\begin{multline*}
\tilde \varphi_n(x) = \varphi_{n-1}(x) 
= 
\sqrt 2 \cos\Big(\tilde \nu_{n} x+\frac {1-\alpha}{ 8}  \pi+\arg\{i+b_\alpha\}\Big) \\
 + \frac {\sqrt{3-\alpha}} { \pi }  \int_0^{\infty}    \rho_0(u)
\Big(
-e^{-u\tilde \nu_n x} \frac{u-b_\alpha}{ \sqrt{1+b_\alpha^2}}-
(-1)^{n-1}   e^{-u\tilde \nu_n (1-x)}\
\Big)du + \nu_n^{-1} \widetilde r_n(x).
\end{multline*}
Replacing $\alpha$ with $2-2H$, setting $\ell_H:=b_{2-2H}$ and removing tilde from the notation,  \eqref{phin} is obtained with
\begin{equation}
\label{rhofla}
\begin{aligned}
&
\rho_0(u)= 
\frac{\sin \theta_0(u)}{\gamma_0(u)}\exp \left\{\frac 1 {\pi} \int_0^\infty \frac{\theta_0(v)}{v+u}dv\right\} \\
&
\theta_0(u) = -\arctan \frac{\cos (\pi H) }{u^{2H+1}+\sin (\pi H)}\\
&
\gamma_0^2(u) = \Big(u+u^{-2H} \sin (\pi H)\Big)^2+ \Big(u^{-2H}\cos(\pi H)\Big)^2
\end{aligned}
\end{equation}
and the formulas \eqref{bndp} and \eqref{ave} follow from \eqref{psipsi} by the corresponding normalization.
%
%

\subsection{The case $\alpha>1$}\label{sec-case2}

Our main representation formula for the Laplace transform of $\varphi$ in Lemma \ref{lem-main} was derived using the fact that
derivative and integration in \eqref{interch} are interchangeable. For $\alpha>1$ this is no longer
possible and, though the same approach still works, some of the calculations are made differently.

\subsubsection{The Laplace transform}

\medskip
The following representation is the analog of Lemma \ref{lem-main}:

\begin{lem}\label{lem510}
For  $\alpha  \in (1,2)$,
$$
\widehat \varphi(z) - \widehat \varphi(0) =-\frac 1{\Lambda(z)}\Big(e^{-z}\Phi_1(-z)+\Phi_0(z)\Big),\quad z\in \mathbb{C}
$$
where the functions $\Phi_0(z)$ and $\Phi_1(z)$, defined in  \eqref{Phi0Phi1case2} below,
are sectionally holomorphic on $\mathbb{C}\setminus \Real_{\ge 0}$, and
\begin{equation}
\label{Lmbd2}
\Lambda(z):=    z \frac{\lambda\Gamma(\alpha)}{|c_\alpha|}   -  z\int_0^\infty\frac{2 t^{\alpha-2}}{t^2-z^2}dt.
\end{equation}

\end{lem}

\begin{proof}

The eigenproblem \eqref{eig} can still be rewritten as \eqref{1diff}, however this time the derivative and integration can no longer be
interchanged and we get (c.f. \eqref{geig})
\begin{equation}
\label{geig2}
\begin{aligned}
\frac{d}{dx} &  \int_0^1   (1-\tfrac \alpha 2) \sign(x-y) |x-y|^{1-\alpha}  \psi(y)dy=-\lambda \psi''(x)  \\
&
\psi(1)=0, \; \psi'(0)=0
\end{aligned}
\end{equation}
Using the identity (c.f. \eqref{idnt})
$$
|x-x'|^{-(\alpha-1)} = \frac 1 {\Gamma (\alpha-1)} \int_0^\infty t^{\alpha-2}e^{-t|x-x'|}dt,\quad \alpha\in (1,2)
$$
we can rewrite the left hand side of \eqref{geig2} as
\begin{multline*}
\frac {d}{dx}\int_0^1  (1-\tfrac \alpha 2)  \sign(x-y) |x-y|^{-(\alpha-1)}  \psi(y)dy = \\
  \frac  {1-\frac \alpha 2} {\Gamma (\alpha-1)} \frac {d}{dx}\int_0^\infty t^{\alpha-2} \left(\int_0^1    \sign(x-y)
 e^{-t|x-y|}
 \psi(y)dy\right) dt.
\end{multline*}
If we define
$$
u(x,t) :=  \int_0^1 \sign(x-y)   e^{-t|x-y|} \psi(y) dy  \quad \text{and}\quad
u_0(x):=  \int_0^\infty t^{\alpha-2}u(x,t) dt,
$$
the equation \eqref{geig2} reads
\begin{equation}
\label{u0tag}
u'_0(x) = -\lambda \frac{\Gamma(\alpha-1)}{1-\frac \alpha 2} \psi''(x),
\end{equation}
and  $u(x,t)$ solves the equation
$$
u''(x,t) = 2\psi'(x) +t^2 u(x,t),
$$
subject to boundary conditions
\begin{align*}
& u'(0,t) = \phantom{+} t u(0,t)+2\psi(0)\\
& u'(1,t) = -t u(1,t).
\end{align*}

On the other hand, the Laplace transform $\widehat u''(z,t)$ satisfies
\begin{align*}
\widehat u''(z,t) =\;
&
e^{-z} u'(1,t)-u'(0,t) +   ze^{-z} u(1,t)-z u(0,t)+z^2 \widehat u(z,t) =\\
&
e^{-z} (z -t)u(1,t) -(z+t) u(0,t)   -2\psi(0)+z^2 \widehat u(z,t),
\end{align*}
and therefore
$$
(z^2-t^2) \widehat u(z,t)=
-e^{-z}(z -t) u(1,t) +(z+t) u(0,t) +  2\big(\widehat \psi'(z)  +\psi(0)\big).
$$
Multiplying both sides by $t^{\alpha-2}/(z^2-t^2)$ and integrating we obtain
$$
 \widehat u_0(z)=
-e^{-z}\int_0^\infty\frac{t^{\alpha-2}}{z+t} u(1,t)dt +  \int_0^\infty \frac{t^{\alpha-2}}{z-t}u(0,t)dt +  2\big(\widehat \psi'(z)  +\psi(0)\big) \int_0^\infty\frac{t^{\alpha-2}}{z^2-t^2}dt,
$$
and since $\widehat u'_0(z) = e^{-z} u_0(1)-u_0(0)+z \widehat u_0(z)$,
\begin{align*}
\widehat u'_0(z) =&\; e^{-z} u_0(1)-u_0(0)  \\
&
  -ze^{-z}\int_0^\infty\frac{t^{\alpha-2}}{z+t} u(1,t)dt
+z\int_0^\infty \frac{t^{\alpha-2}}{z-t}u(0,t)dt +  z^2 \widehat \psi(z) \int_0^\infty\frac{2t^{\alpha-2}}{z^2-t^2}dt \\
=\; &
e^{-z}   \int_0^\infty\frac{t^{\alpha-1}}{z+t} u(1,t)dt  +
  \int_0^\infty \frac{t^{\alpha-1}}{z-t}u(0,t)dt +  z^2 \widehat \psi(z) \int_0^\infty\frac{2t^{\alpha-2}}{z^2-t^2}dt.
\end{align*}
On the other hand, by \eqref{u0tag}
$$
\frac 1 \lambda  \frac{ 1-\frac \alpha 2} {\Gamma(\alpha-1)} \widehat u'_0(z) =
- \widehat \psi''(z) =
-  e^{-z} \psi'(1) -  z \widehat \psi'(z)=
-  e^{-z} \psi'(1) +  z  \psi(0) -  z^2 \widehat \psi(z).
$$
Combining the two expressions and noting that
$\frac {\Gamma (\alpha-1)} {1-\frac \alpha 2}  =-\frac {\Gamma (\alpha)} {(1-\frac \alpha 2)(1-\alpha)}  =  \frac {\Gamma (\alpha)} {|c_\alpha|} $, we get
\begin{multline*}
  - \Big(  \lambda    \frac {\Gamma (\alpha)} {|c_\alpha|} z
 +z \int_0^\infty\frac{2t^{\alpha-2}}{z^2-t^2}dt\Big)z  \widehat \psi(z) =\\
  \lambda  \frac {\Gamma (\alpha)} {|c_\alpha|} e^{-z} \psi'(1)
  +
e^{-z}   \int_0^\infty\frac{t^{\alpha-1}}{t+z} u(1,t)dt
 - \lambda   \frac {\Gamma (\alpha)} {|c_\alpha|}z  \psi(0)
     -
  \int_0^\infty \frac{t^{\alpha-1}}{t-z}u(0,t)dt.
\end{multline*}
Since $z  \widehat \psi(z)=-\widehat \varphi(z)+\widehat \varphi(0)$, the equation \eqref{main} is obtained, if we define $\Lambda(z)$ as in \eqref{Lmbd2} and
set
\begin{equation}
\label{Phi0Phi1case2}
\begin{aligned}
& \Phi_0(z):=  \phantom{+}\lambda   \frac {\Gamma (\alpha)} {|c_\alpha|}   \psi(0)z
     +
  \int_0^\infty \frac{t^{\alpha-1}}{t-z}u(0,t)dt\\
& \Phi_1(z):=  -\lambda   \frac {\Gamma (\alpha)} {|c_\alpha|} \psi'(1)
  -
   \int_0^\infty\frac{t^{\alpha-1}}{t-z} u(1,t)dt
\end{aligned}
\end{equation}

\end{proof}

While the expressions for $\Phi_0(z)$ and $\Phi_1(z)$ in \eqref{Phi0Phi1case2} coincide with those in \eqref{Phi0Phi1}, 
the growth of these functions near the origin and at infinity is different for $\alpha>1$ and $\alpha<1$. Nevertheless, 
all further calculations can be carried out as before.   
Using the identity \eqref{Lint} with $\alpha$ replaced by $\alpha-2$, we find that
$\Lambda(z)$ in this case is given by the formula
$$
\Lambda(z)=     \lambda \frac {\Gamma (\alpha)} {|c_\alpha|} z  +     z^{\alpha-2}  \frac {\pi }{  |\cos \frac {\pi } 2  \alpha | } \begin{cases}
e^{\frac {1-\alpha} 2\pi i  } & \arg(z)\in (0,\pi) \\
 e^{  -\frac {1-\alpha} 2\pi i } & \arg(z)\in (-\pi, 0)
\end{cases}
$$
which basically coincides with \eqref{Lz}. Consequently, the rest of the proof and all the results remain intact, after replacing $c_\alpha$ and $\cos\frac \pi 2 \alpha$
with their absolute values.

\section{Proof of Theorem \ref{main-thm-fbn}}

In this section we give the details for the steps outlined in Section \ref{sec-strat} for the eigenproblem with the
covariance operator \eqref{KfBn}:
\begin{equation}
\label{eigen}
\frac {d}{dx} \int_{0}^1    C_\alpha |x-y|^{1-\alpha} \sign(x-y) \varphi(y)dy = \lambda \varphi(x), \quad x\in [0,1],
\end{equation}
where $C_\alpha = 1-\tfrac \alpha 2$ and $\alpha = 2-2H\in (0,2)\setminus\{1\}$ as before.
The exact second order asymptotics of the eigenvalues was obtained in \cite{Ukai} for $\alpha\in (0,1)$.
We will derive the formulas \eqref{lambdaU} and \eqref{phin1} in the complementary case $\alpha\in (1,2)$, leaving out
similar derivation of the eigenfunctions asymptotics for $\alpha\in (0,1)$.



\subsection{The case $\alpha>1$}

\subsubsection{The Laplace transform}

The first step is to find a useful representation for the Laplace transform \eqref{LT}:

\begin{lem} \label{lem61}
For $\alpha \in (1,2)$
\begin{equation}\label{hatvarphi}
\widehat {\varphi}(z)= \frac{\lambda^{-1}C_\alpha}{\Gamma(\alpha-1)} \frac{1}{\Lambda(z)}\Big(e^{-z} \Phi_{1}(-z)- \Phi_{0}(z)\Big),
\end{equation}
where functions $\Phi_0(z)$ and $\Phi_1(z)$, defined in \eqref{Psipm} below, are sectionally holomorphic
on $\mathbb{C} \setminus \Real_{\ge 0}$ and
\begin{equation}
\label{Lambda}
\Lambda(z) := 1-\frac{2\lambda^{-1}C_\alpha}{\Gamma(\alpha-1)} z^2 \int_0^\infty \frac{t^{\alpha-2}}{z^2-t^2 }dt.
\end{equation}
\end{lem}

\begin{proof}

Let us define the functions
$$
u(x,t) = \int_{0}^1  \varphi(x')  e^{-t|x-x'|} \sign(x-x')dx'
$$
and
$$
u_0(x) = \int_0^\infty  t^{\alpha-2}  u(x,t) dt.
$$
Using  identity \eqref{idnt} with $\alpha$ replaced by $\alpha-1\in (0,1)$
we get
\begin{align*}
(\widetilde K \varphi)(x) = &
\frac {C_\alpha} {\Gamma (\alpha-1)}
\frac{d}{dx} \int_0^\infty t^{\alpha-2} \int_{0}^1  \varphi(x')  e^{-t|x-x'|} \sign(x-x')dx' dt =
\\
&
\frac {C_\alpha} {\Gamma (\alpha-1)}
\frac{d}{dx} \int_0^\infty  t^{\alpha-2}  u(x,t) dt =
 \frac {C_\alpha} {\Gamma (\alpha-1)} u'_0(x),
\end{align*}
and therefore
\begin{equation}
\label{phi}
\varphi(x) = \frac{\lambda^{-1}C_\alpha}{\Gamma(\alpha-1)} u'_0(x), \quad x\in [0,1].
\end{equation}
Differentiating $u(x,t)$ with respect to $x$ twice gives
\begin{multline}\label{utag}
u'(x,t) =
\frac d{dx}\left(
\int_{0}^x  \varphi(x')    e^{-t(x-x')}dx'
-
\int_{x}^1  \varphi(x')    e^{-t(x'-x)}dx'
\right) =\\
  2\varphi(x)    -t\int_{0}^x  \varphi(x')    e^{-t(x-x')}dx'
       -t\int_{x}^1  \varphi(x')    e^{-t(x'-x)}dx',
\end{multline}
and
$$
u''(x,t) =  2\varphi'(x)     + t^2u(x,t),\quad x\in (0,1).
$$
Let us stress that it is not clear at this stage whether $\varphi'(x)$ exists at the endpoints $x\in \{0,1\}$.
Thus we define
$$
v(x,t) := u(x,t)-\frac{2\lambda^{-1}C_\alpha}{\Gamma(\alpha-1)}u_0(x), \quad x\in [0,1],
$$
which, for $x\in (0,1)$, satisfies
$
v' (x,t) = u' (x,t)-2\varphi(x).
$
Hence the definition of $v'(x,t)$  extends continuously to $[0,1]$  and using \eqref{utag}, we obtain

\begin{equation}
\label{bnd}
\begin{aligned}
& v'(1,t) = -t u(1,t) = -t  v(1,t)-t \frac{2\lambda^{-1}C_\alpha}{\Gamma(\alpha-1)}u_0(1) \\
& v'(0,t) = \phantom{+}t u(0,t) = \phantom{+}t v(0,t) + t \frac{2\lambda^{-1}C_\alpha}{\Gamma(\alpha-1)}u_0(0).
\end{aligned}
\end{equation}
The Laplace transform of $v''(x,t)$ satisfies
$$
\widehat v''(z,t) =   \int_{0}^1 e^{-zx} v''(x,t) dx = 
e^{-z}v'(1,t)- v'(0,t) +z e^{-z}v(1,t)-z  v(0,t) +z^2 \widehat v (z,t)
$$
and since $v''(x,t) = t^2 u(x,t)$, using  the boundary conditions \eqref{bnd},
\begin{align*}
&
t^2 \widehat u(z,t) =   \widehat v''(t,z) =\\
&
e^{-z}v'(1,t)-  v'(0,t) +z e^{-z}v(1,t)-z  v(0,t) +z^2 \widehat u(z,t)  -z^2\frac{2\lambda^{-1}C_\alpha}{\Gamma(\alpha-1)}\widehat u_0(z)=\\
&
e^{-z} (z-t)  u(1,t)- (z+t)u(0,t)
-\frac{2\lambda^{-1}C_\alpha}{\Gamma(\alpha-1)}z\Big(
 e^{-z}   u_0(1)
-    u_0(0)+z \widehat u_0(z)
\Big)
 +z^2  \widehat u(z,t).
\end{align*}
Rearranging, we obtain
\begin{multline*}
(t^2 -z^2)\widehat u(z,t)= -\frac{2\lambda^{-1}C_\alpha}{\Gamma(\alpha-1)}z^2 \widehat u_0(z)
-\frac{2\lambda^{-1}C_\alpha}{\Gamma(\alpha-1)}z\Big(
 e^{-z}   u_0(1)
-    u_0(0)\Big) \\
+e^{-z} (z-t)  u(1,t)-(z+t)u(0,t).
\end{multline*}
Multiplying by $t^{\alpha-2}/(t^2 -z^2)$ and integrating gives
\begin{align*}
 \widehat u_0(z) = & \; \widehat u_0(z) \frac{2\lambda^{-1}C_\alpha}{\Gamma(\alpha-1)} z^2 \int_0^\infty \frac{t^{\alpha-2}}{z^2-t^2 }dt
+ \frac{2\lambda^{-1}C_\alpha}{\Gamma(\alpha-1)}\Big(
 e^{-z}   u_0(1)
-   u_0(0)\Big)  z \int_0^\infty \frac{t^{\alpha-2}}{z^2-t^2 }dt \\
& -e^{-z} \int_0^\infty \frac{t^{\alpha-2}}{t+z}u(1,t)dt -\int_0^\infty \frac{t^{\alpha-2}}{t-z}u(0,t) dt.
\end{align*}
Using the definition of $\Lambda(z)$  in \eqref{Lambda}, we get
$$
\Lambda(z) \widehat u_0(z)=
 \frac{1-\Lambda(z)}{z} \Big(
 e^{-z}   u_0(1)
-     u_0(0)\Big)
-e^{-z} \int_0^\infty \frac{t^{\alpha-2}}{t+z}u(1,t)dt -\int_0^\infty \frac{t^{\alpha-2}}{t-z}u(0,t) dt.
$$
Since
$
\widehat u'_0(z) = e^{-z} u_0(1)-  u_0(0) +z \widehat u_0(z)
$
it follows that
\begin{align*}
\Lambda(z)   \widehat u'_0(z)=
&
\; e^{-z}   u_0(1)-     u_0(0)  -e^{-z} \int_0^\infty \frac{zt^{\alpha-2}}{t+z}u(1,t)dt - \int_0^\infty \frac{zt^{\alpha-2}}{t-z}u(0,t) dt=\\
&
\; e^{-z}   \int_0^\infty      \frac{t^{\alpha-1} }{t+z} u(1,t)dt  -     \int_0^\infty   \frac{t^{\alpha-1}}{t-z} u(0,t) dt.
\end{align*}
If we now define
\begin{equation}
\label{Psipm}
\begin{aligned}
\Phi_{0}(z) &=\int_0^\infty \frac{t^{\alpha-1}}{t-z}u(0,t)dt \\
\Phi_{1}(z) &=\int_0^\infty \frac{t^{\alpha-1}}{t-z}u(1,t)dt
\end{aligned}
\end{equation}
the following expression for the Laplace transform of $u'_0(x)$ is obtained
$$
 \widehat u'_0(z)= \frac{1}{\Lambda(z)}\Big(e^{-z} \Phi_{1}(-z)- \Phi_{0}(z)\Big),
$$
and the claimed formula follows from \eqref{phi}.

\end{proof}

The following lemma gives a simpler expression  for $\Lambda(z)$, which exhibits its singularities:

\begin{lem}
\

\medskip
\noindent
a) The function $\Lambda(z)$, defined in \eqref{Lambda}, admits the following explicit expression
\begin{equation}
\label{Lambda2}
\Lambda(z)=1-(z/\nu)^{\alpha-1}
\begin{cases}
e^{-i\frac \pi 2   (\alpha   -1) } &  \arg(z)\in (0,\pi) \\
e^{i\frac {\pi} 2  (\alpha   -1)  } &  \arg(z)\in (-\pi, 0)
\end{cases}
\end{equation}
which is discontinuous along the real axis and has two zeros at $\pm z_0 = \pm i \nu$ with
\begin{equation}\label{nuflan}
\nu^{\alpha-1} =  \frac{\Gamma(\alpha-1)\sin \frac {\pi (\alpha-1)} 2 }{ \pi\lambda^{-1}C_\alpha}.
\end{equation}

\medskip
\noindent
b) The limits $\Lambda^{\pm}(t) =\lim_{z\to t^{\pm}} \Lambda(z)$ are given by the formulas
$$
\Lambda^{\pm}(t) =
1 -  |t/\nu|^{\alpha-1}
\begin{dcases}
e^{\mp i\frac \pi 2( \alpha- 1)} & t>0 \\
e^{\pm i \frac \pi 2 (\alpha- 1)} & t<0
\end{dcases}
$$
and satisfy the symmetries \eqref{conjp}-\eqref{absL}.

\medskip
\noindent
c) The argument $\theta(t):=\arg\{\Lambda^+(t)\}\in (-\pi, \pi]$ is an odd function, $\theta(-t)=-\theta(t)$, given by
\begin{equation}\label{tht}
\theta(t) =
\arctan \frac{   (t/\nu)^{\alpha-1} \sin \frac {\pi (\alpha-1)} 2}{1 - (t/\nu)^{\alpha-1} \cos \frac {\pi (\alpha-1)} 2}, \quad t\ge 0
\end{equation}
where the branches of $\arctan$ are chosen so that $\theta(t)$ increases continuously from $\theta(0)=0$ to
$\theta(\infty):=\lim_{t\to\infty}\theta(t)=\frac{3-\alpha}{2}\pi$ as $t\to\infty$.
\end{lem}

%
%
%

\begin{proof}
The claimed formulas follow from the identity \eqref{Lint}  with $\alpha$ being replaced by $\alpha-2$.

\end{proof}

\subsubsection{Removal of singularities}

The inverse of  $\widetilde K$ is an integral operator with a weakly singular kernel (see \eqref{kappa} below)
and therefore its eigenfunctions are continuous on $[0,1]$ and, consequently, the Laplace transform
$\widehat\varphi(z)$ is an entire function. Therefore both
singularities must be removable.
Since $\Phi_i(-z_0)=\overline{\Phi_i(z_0)}$, $i=0,1$ the poles at $\pm z_0=\pm i\nu$ are removed if
\begin{equation}
\label{eq1}
e^{- z_0} \Phi_{1}(- z_0)- \Phi_{0}( z_0)=0.
\end{equation}
The discontinuity $\widehat \varphi^+(t)\ne \widehat \varphi^-(t)$, $t\in \Real$ vanishes if
\begin{align*}
&
\frac{1}{\Lambda^+(t)}\Big(e^{-t} \Phi_{1}(-t)- \Phi^+_{0}(t)\Big)
=
\frac{1}{\Lambda^-(t)}\Big(e^{-t} \Phi_{1}(-t)-  \Phi^-_{0}(t)\Big),
\quad t>0 \\
&
\frac{1}{\Lambda^+(t)}\Big(e^{-t} \Phi^-_{1}(-t)-  \Phi_{0}(t)\Big) =
\frac{1}{\Lambda^-(t)}\Big(e^{-t} \Phi^+_{1}(-t)- \Phi_{0}(t)\Big), \quad t<0.
\end{align*}

%

\noindent
Since $\theta(t)$ is antisymmetric around the origin, a rearrangement reduces these conditions
to
\begin{equation}
\label{Psipmcond}
\begin{aligned}
&
 \Phi^+_{0}(t) -   e^{2i\theta(t)}\Phi^-_{0}(t)
=
-2i  e^{i\theta(t)}\sin\theta(t) e^{-t} \Phi_{1}(-t)
\\
&
  \Phi^+_{1}(t)- e^{2i\theta(t)}\Phi^-_{1}(t)      =
-2i e^{ i\theta(t)}\sin\theta(t) e^{- t} \Phi_{0}(-t)
\end{aligned}\quad t>0
\end{equation}
Moreover, definition \eqref{Psipm} gives the estimates
\begin{equation}\label{Phigr}
\Phi_i(z) = \begin{cases}
O( z^{\alpha-2}) &  z \to\infty \\
O(1) &  z \to 0
\end{cases}\quad i=0,1
\end{equation}

\subsubsection{An equivalent formulation of the eigenproblem}
Let us find functions $\Phi_0(z)$ and $\Phi_1(z)$, which are sectionally holomorphic on
$\mathbb{C}\setminus \Real_{\ge 0}$ and satisfy the boundary condition \eqref{Psipmcond},
the algebraic constraint \eqref{eq1} and the growth conditions \eqref{Phigr}.
Following the program outlined in Section \ref{sec-strat}, we will construct such
functions as solutions to a certain integro-algebraic system of equations. To this end, let us define (cf. \eqref{X0z})
\begin{equation}\label{Xz0}
X_0(z):=
(-z)^{ - \theta_0(\infty)/\pi}\exp \left(\frac 1 { \pi } \int_0^\infty \frac{\theta_0(t)-\theta_0(\infty)}{t-z}dt\right)
\end{equation}
where $\theta_0(t):=\theta(\nu t)$ does not depend on $\nu$, see \eqref{tht}. For any $\nu>0$ let $p_\pm(t)$ be
the solutions of the integral equations
\begin{equation}
\label{qpmeq}
p_\pm( t)  = \pm \frac 1 \pi \int_0^\infty \frac{h_0(\tau)e^{-\nu\tau}}{\tau+ t}p_\pm(\tau)d\tau + 1, \quad t>0,
\end{equation}
where the function
$$
h_0(t) =  e^{i\theta_0(t)} \sin \theta_0(t)\frac{X_0(-t)}{X_0^+(t)}, \quad t >0
$$
takes real values. Let us extend the domain of $p_\pm$ to the cut plane by setting
\begin{equation}
\label{qpmfbn}
p_\pm(z)  := \pm \frac 1 \pi \int_0^\infty \frac{h_0(\tau)e^{-\nu\tau}}{\tau+z}p_\pm(\tau)d\tau + 1, \quad z\in \mathbb{C}\setminus \Real_{\le 0}
\end{equation}
and define
\begin{equation}\label{xibareta}
\begin{aligned}
& \xi  := e^{-i\nu/2}X_0(-i) p_-(i) \\
& \bar\eta  :=  e^{-i\nu/2}  X_0(-i) p_+(i).
\end{aligned}
\end{equation}

The following lemma gives an equivalent formulation of the eigenproblem \eqref{eigen}:

\begin{lem}\label{lemia2}
Let $(p_\pm, \nu)$ with $\nu>0$ be a solution of the system, which consists of integral equations \eqref{qpmfbn} and
algebraic equation
\begin{equation}
\label{algfbn}
\Im\{ \xi\}\Re\{\bar \eta\} =0.
\end{equation}
Let $\varphi$ be defined by the Laplace transform \eqref{hatvarphi}, where
\begin{equation}
\label{defPhi}
\begin{aligned}
&
\Phi_{1}( z) =    X_0(z/\nu )\Big(\Re\{\eta\}  p_-(-z/\nu )+\Im\{\xi\} p_+(-z/\nu ) \Big) \\
&
\Phi_{0}( z) =    X_0(z/\nu )\Big(\Re\{\eta\}  p_-(-z/\nu )-\Im\{\xi\} p_+(-z/\nu ) \Big),
\end{aligned}
\end{equation}
and $\lambda$ be defined by \eqref{nuflan}. Then the pair $(\lambda, \varphi)$ solves the eigenproblem \eqref{eigen}.
Conversely, any pair $(\lambda, \varphi)$ satisfying \eqref{eigen}, defines a solution of the above integro-algebraic system.
\end{lem}

\begin{proof}\

\medskip
\noindent
a) Let us find a nonvanishing sectionally holomorphic function $X(z)$ on the cut plane
$\mathbb{C}\setminus \Real_{\ge 0}$, satisfying the boundary condition
\begin{equation}
\label{Xmp}
\frac{X^+(t)}{X^-(t)} = e^{2i\theta(t)}\quad t>0.
\end{equation}
Since
$\theta_0(\infty)=\lim_{t\to\infty}\theta_0(t)\ne 0$, the appropriate choice in this case is
$$
X(z) = (-z)^{ - \theta(\infty)/\pi}\exp \left(\frac 1 { \pi } \int_0^\infty \frac{\theta(t)-\theta(\infty)}{t-z}dt\right),
$$
as can be readily checked by the Sokhotski--Plemelj formula.
This function is unique up to a multiplicative factor $z^\beta$ with any integer $\beta$. Any $\beta \le 1$
will suit our purposes, that is, render the functions in the right hand side of \eqref{SDfbn} below to be H\"older
on the positive real line and vanish at infinity. Fixing $\beta=0$ is a matter of convenience and any other
choice will not change the ultimate form of the integro-algebraic system.

Note that $X(z)$ satisfies
\begin{equation}\label{Xzest}
|X(z)|\sim \begin{cases}
|z|^{ \frac{\alpha-3}{2}}& z\to\infty\\
1 & z\to 0
\end{cases}
\end{equation}

Plugging \eqref{Xmp} into \eqref{Psipmcond}, we obtain
$$
\begin{aligned}
&
\frac{\Phi^+_{0}(t)}{X^+(t)} -   \frac{\Phi^-_{0}(t)}{X^-(t)}
=
-2i  e^{i\theta(t)}\sin\theta(t) e^{-t} \frac{X(-t)}{X^+(t)}\frac{\Phi_{1}(-t)}{X(-t)}
 \\
&
 \frac{\Phi^+_{1}(t)}{X^+(t)}- \frac{\Phi^-_{1}(t)}{X^-(t)}      =
-2i e^{ i\theta(t)}\sin\theta(t) e^{- t}\frac{X(-t)}{X^+(t)} \frac{\Phi_{0}(-t)}{X(-t)}
\end{aligned}
$$
and hence the functions
\begin{equation}\label{SDPhi}
\begin{aligned}
& S(z) = \frac{\Phi_{1}(z)+\Phi_{0}(z)}{2X(z)} \\
& D(z) = \frac{\Phi_{1}(z)-\Phi_{0}(z)}{2X(z)}
\end{aligned}
\end{equation}
solve the decoupled system
\begin{equation}
\label{SDfbn}
\begin{aligned}
&
S^+(t) -  S^-(t) = -2i   e^{-t} h(t)S(-t) \\
&
D^+(t) - D^-(t) =   \; 2i   e^{-t} h(t)D(-t)
\end{aligned}
\end{equation}
where we defined
$$
h(t) =  e^{i\theta(t)} \sin \theta(t)\frac{X(-t)}{X^+(t)}.
$$
By the Sokhotski--Plemelj formula
\begin{align*}
\frac{X(-t)}{X^+(t)} & =
e^{-i\theta(\infty)}
\exp\left(\frac 1 { \pi } \int_0^\infty \frac{\theta(\tau)-\theta(\infty)}{\tau+t}d\tau-
\frac 1 { \pi   } \dashint_0^\infty \frac{\theta(\tau)-\theta(\infty)}{\tau-t}d\tau
\right)e^{i\theta(\infty)-i  \theta(t)} \\
&=
\exp\left(-\frac {2t} { \pi } \dashint_0^\infty
\frac{ \theta(\tau)-\theta(\infty)  }{\tau^2-t^2}
d\tau
\right)  e^{-i  \theta(t)}
\end{align*}
and integration by parts shows that
\begin{equation}
\label{halt}
\begin{aligned}
h_0(t) = h(\nu t) & =
\exp\left(-\frac {2t} { \pi } \dashint_0^\infty
\frac{ \theta_0(\tau)-\theta_0(\infty)  }{\tau^2-t^2}
d\tau
\right)   \sin \theta_0(t)\\
&=
\exp\left(-\frac 1 \pi \dashint_0^\infty  \theta_0'(\tau) \log \left|\frac{t+\tau}{t-\tau}\right|d\tau
\right)
 \sin\theta_0(t)\ge 0,
\end{aligned}
\end{equation}
a real valued function.

\medskip
\noindent
b) By the estimates \eqref{Phigr} and \eqref{Xzest}, the functions in the right hand sides of \eqref{SDfbn} are
H\"older continuous and vanish at infinity. Therefore, applying the Plelemlj formula, we get
\begin{equation}
\label{SDiefbn}
\begin{aligned}
&
S(z)  = -\frac 1 \pi \int_0^\infty \frac{h(\tau)e^{-\tau}}{\tau-z}S(-\tau)d\tau + P_1(z) \\
&
D(z)  = \;\; \frac 1 \pi \int_0^\infty \frac{h(\tau)e^{-\tau}}{\tau-z}D(-\tau)d\tau + P_2(z)
\end{aligned}
\end{equation}
where $P_1(z)$ and $P_2(z)$ are arbitrary polynomials. The same estimates imply that $S(-t)$ and $D(-t)$ grow not faster
than $t^{\frac{\alpha-1}2}$ as $t\to\infty$. In fact,  using the definitions of $\Phi_1$ and $\Phi_0$, 
it can be seen that $S(-t)$ and $D(-t)$ are asymptotic to constants as $t\to\infty$,
since the exact formula for the inverse kernel \eqref{kappa} implies that the eigenfunctions vanish at the endpoints of the interval.
Therefore the polynomials must be of the zero degree, that is, $P_1(z) :=c_1$ and $P_2(z) :=c_2$ with real constants $c_1$ and $c_2$.

If we set $z:=-t$ for $t\in \Real_{>0}$ in \eqref{SDiefbn}, the integral equations for $S(-t)$ and $D(-t)$ are obtained:
\begin{align*}
&
S(-t)  = -\frac 1 \pi \int_0^\infty \frac{h(\tau)e^{-\tau}}{\tau+t}S(-\tau)d\tau + c_1 \\
&
D(-t)  = \;\; \frac 1 \pi \int_0^\infty \frac{h(\tau)e^{-\tau}}{\tau+t}D(-\tau)d\tau + c_2
\end{align*}
and by linearity
$$
\begin{aligned}
& S(-\nu t) = c_1 p_-(t)  \\
& D(-\nu t) = c_2 p_+(t)
\end{aligned}
$$
Combining this with the definition \eqref{SDPhi} we get
\begin{equation}
\label{PP}
\begin{aligned}
&
\Phi_{0}(z\nu) =    X_0(z)\Big( c_1 p_-(-z) - c_2 p_+(-z)\Big)\\
&
\Phi_{1}(z\nu) =    X_0(z)\Big(c_1 p_-(-z) + c_2 p_+(-z)\Big)
\end{aligned}
\end{equation}
Plugging these expressions into \eqref{eq1} gives the equation
$$
 \big(\xi   - \bar \xi\big)c_1 +\big( \bar \eta+\eta\big)c_2   =0,
$$
where $\xi$ and $\eta$ are defined in \eqref{xibareta}.
This equation has a nontrivial solution if and only if the condition \eqref{algfbn} is satisfied. The expressions \eqref{defPhi} now follow from \eqref{PP} after discarding the multiplicative constants.
Thus starting with a solution of the eigenproblem we constructed a solution of the integro-algebraic system
defined in above. The converse construction is checked by reverting all the operations.
\end{proof}

The value of $X_0(i)$ for the function defined in \eqref{Xz0} can be found explicitly and will be useful in
the calculations to follow:

\begin{lem}
\begin{equation}
\label{argmod}
\arg \{X_0(i) \}= \frac{3-\alpha}{8}\pi \quad\text{and}\quad
\big|X_0(i)\big| = \sqrt{\frac{\alpha-1}{2}}
\end{equation}
\end{lem}

\begin{proof}
Define
$$
\Gamma_0(z)=\frac 1 { \pi } \int_0^\infty \frac{\theta(t)-\theta(\infty)}{t-\nu z}dt =
\frac 1 { \pi } \int_0^\infty \frac{\theta_0(t)-\theta_0(\infty)}{t- z}dt
$$
where $\theta_0(t):=\theta(\nu t)$ does not depend on $\nu$. Then
\begin{align*}
\Im\left\{\Gamma_0(i)\right\}  & =\frac 1 { \pi } \int_0^\infty \frac{\theta_0(t)-\theta_0(\infty)}{t^2+1}dt \\
&
=
\frac 1 { \pi } \int_0^1 \frac{\theta_0(u)-\theta_0(\infty)}{u^2+1}dt
+
\frac 1 { \pi } \int_0^1 \frac{\theta_0(u^{-1})-\theta_0(\infty)}{u^2+1} du\\
&
=
\frac 1 { \pi } \int_0^1 \frac{\theta_0(u)+\theta_0(u^{-1})}{u^2+1}du
-
\frac {2\theta_0(\infty)} { \pi } \int_0^1 \frac{1}{u^2+1} du.
\end{align*}
Further,
\begin{align*}
&
 \theta_0(u) + \theta_0(u^{-1})   = \arg\Big\{
\Lambda^+(\nu u)\Lambda^+(\nu u^{-1})
\Big\}= \\
&
\arg\left\{
\Big(1-
u^{\alpha-1 }
 e^{ - i\frac {\pi   }  2(\alpha-1)} \Big)
\Big(1-
u^{1-\alpha}
e^{- i  \frac {\pi } 2(\alpha-1)} \Big)
\right\}
= \\
&
\arg\left\{
-u^{\alpha-1 }  e^{ - i\frac {\pi   }  2(\alpha-1)}\Big|1-u^{1-\alpha} e^{  i\frac {\pi   }  2(\alpha-1)} \Big|^2
\right\}= \\
&
\arg\left\{
-   e^{ - i\frac {\pi   }  2(\alpha-1)}
\right\} =
\pi - \frac{\pi} 2(\alpha-1) = \theta_0(\infty)
\end{align*}
and $\theta_0(u)+\theta_0(u^{-1})=\theta_0(\infty)$ by continuity. Consequently
$$
\Im\left\{\Gamma_0(i)\right\}  =
-\frac {\theta_0(\infty)} { \pi } \int_0^1 \frac{1}{u^2+1} du=
-\frac {\theta_0(\infty)} { 4 }
$$
which gives the expression claimed in \eqref{argmod}:
$$
\arg\big\{X_0(i)\big\} = \arg\Big\{(-i)^{-\theta_0(\infty)/\pi}e^{i\,\Im\{\Gamma_0(i)\}}\Big\}=
\arg\Big\{ e^{ i\frac {\theta_0(\infty)} { 4 }}\Big\}=\frac{\theta_0(\infty)}{4}.
$$

To compute the absolute value of $X_0(i)$,
let us consider the product
\begin{equation}\label{XXz}
X_0(z)X_0(-z) =
(-z )^{ -\frac{\theta(\infty)}\pi}z^{ -\frac{\theta(\infty)}\pi}
\exp \left(\frac 1 { \pi } \int_0^\infty \frac{\widetilde \theta_0(t)}{t-z}dt+\frac 1 { \pi } \int_0^\infty \frac{\widetilde \theta_0(t)}{t+z }dt\right),
\end{equation}
where  $\widetilde \theta_0(t):=\theta_0(t)-\theta_0(\infty)$.
Define the complex constant
$$
C_1 = \frac 1 { \pi } \int_0^\infty \frac{\widetilde \theta_0(t)}{t-i}dt
+\frac 1 { \pi } \int_0^\infty \frac{\widetilde \theta_0(t)}{t+i }dt,
$$
then the expression in the exponent in \eqref{XXz} can be rewritten as
\begin{align*}
&
\frac 1 { \pi } \int_0^\infty \frac{\widetilde \theta_0(t)}{t-z}dt
+\frac 1 { \pi } \int_0^\infty \frac{\widetilde \theta_0(t)}{t+z }dt =\\
&
\frac 1 { \pi } \int_0^\infty \widetilde \theta_0(t)\left(\frac{1}{t-z}-\frac{1}{t-i}\right)dt
+\frac 1 { \pi } \int_0^\infty \widetilde \theta_0(t)\left(\frac{1}{t+z }-  \frac{1}{t+i}\right)dt + C_1=\\
&
\frac {z-i} { \pi } \left(\int_0^\infty  \frac{\widetilde \theta_0(t)}{(t-z)(t-i)} dt
-  \int_0^\infty  \frac{\widetilde \theta_0(t)}{(t+z)(t+i) } dt\right) + C_1=\\
&
\frac {z-i} { \pi } \left(\int_0^\infty  \frac{ \theta_0(t)}{(t-z)(t-i)} dt
-  \int_0^\infty  \frac{ \theta_0(t)}{(t+z)(t+i) } dt\right) + \theta_0(\infty)\frac {2} \pi \log \frac z i + C_1 =\\
&
\frac {z-i} { \pi }  \int_{-\infty}^\infty  \frac{ \theta_0(t)}{(t-z)(t-i)} dt
  + \theta_0(\infty)\frac {2} \pi \log \frac z i + C_1
\end{align*}
where we used the property $\theta_0(-t)=-\theta_0(t)$ and the identity
\begin{align*}
\frac {z-i} { \pi } \left(\int_0^\infty  \frac{\theta_0(\infty)}{(t-z)(t-i)} dt
-  \int_0^\infty  \frac{\theta_0(\infty)}{(t+z)(t+i) } dt\right)   = \frac 2 \pi \log \frac i z.
\end{align*}

Plugging this back into \eqref{XXz} we obtain
\begin{align*}
X_0(z)X_0(-z) =
C_2\exp \left(
\frac {z-i} { 2\pi i }  \int_{-\infty}^\infty  \frac{ \log \Lambda_0^+(t)/\Lambda_0^-(t)}{(t-z)(t-i)} dt
\right)
\end{align*}
where $C_2$ is a complex constant and $\Lambda_0(z):=\Lambda(z\nu)$. Define $\widetilde \Lambda_0(z):= \Lambda_0(z)/(z^2+1)$, then
$$
\int_{-\infty}^\infty  \frac{ \log \Lambda_0^+(t)/\Lambda_0^-(t)}{(t-z)(t-i)} dt =
\int_{-\infty}^\infty  \frac{ \log \widetilde\Lambda_0^+(t) }{(t-z)(t-i)} dt -
\int_{-\infty}^\infty  \frac{ \log  \widetilde\Lambda_0^-(t)}{(t-z)(t-i)} dt.
$$
Since $\pm i$ are the only roots of $\Lambda_0(z)$, the function $\widetilde \Lambda_0(z)$ is analytic on $\mathbb{C}\setminus \Real$ and
therefore the integrals in the right hand side can be computed by means of the residue theorem.
Integrating the function $g(\xi)=\log \widetilde \Lambda_0(\xi)/(\xi-z)(\xi-i)$ over half circle contours
in the upper and lower half planes, we get
\begin{align*}
\frac 1{2\pi i}\int_{-\infty}^\infty g^+(t)dt -\frac 1{2\pi i}\int_{-\infty}^\infty g^-(t)dt & = \Res(g,z)+\Res(g,i) =
 \frac{\log\widetilde \Lambda_0(z)}{z-i}
+
\frac{\log \Lambda'_0(i)/2i}{i-z}
\end{align*}
and consequently
$$
X_0(z)X_0(-z) =
C_3 \widetilde \Lambda_0(z) = C_3 \frac{\Lambda_0(z)}{z^2+1}.
$$
Constant $C_3$ can be identified through the limit
$$
C_3 = \lim_{z\to\infty, \mathrm{Im}(z)>0}\frac{X_0(z)X_0(-z)(z^2+1)}{ \Lambda_0(z) } =
 1
$$
and the claimed formula follows
\begin{align*}
|X_0(i)|^2=X_0(i)X_0(-i) = \lim_{z\to i}
C_3 \frac{\Lambda_0(z)}{z^2+1}=
C_3 \frac{\Lambda'_0(i)}{ 2i }=\frac{\alpha-1}{2}.
\end{align*}
\end{proof}

\subsubsection{Properties of the integro-algebraic system}

Solvability of the equations in \eqref{qpmeq} again depends on the properties of operator $A$ defined in \eqref{Aop}.
The following analog of Lemma \ref{lem-A} shows that $A$ is a contraction, but this time, uniformly over $\nu>0$:

\begin{lem}\label{lemopA}
For any $\alpha_0\in (1,2)$ there exists an $\eps>0$, such that $\|A\|\le 1-\eps$ for all $\nu>0$ and $\alpha\in (1,\alpha_0]$.
\end{lem}

\begin{proof}
The function under the exponent in \eqref{halt} is nonpositive, since
$\theta_0'(t)= d\theta_0(t)/dt\ge 0$ for all $t> 0$.
Since $\sin\theta_0(0)=0$
$$
h_0(t)\le \sin\theta_0(\infty) =\sin   \tfrac {3-\alpha} 2 \pi<1, \quad t\ge 0.
$$
Hence by the same argument as in Lemma \ref{lem-A}, we have $\|A\|\le \|h_0\|_\infty$ and the claim follows since $\|h_0\|_\infty\le \sin   \frac {3-\alpha_0} 2 \pi:=1-\eps$.
\end{proof}

The following lemma gathers several useful estimates:

\begin{lem}\label{lem-aest}
For any $\alpha_0\in (1,2)$ there exist a constant $C$, such that for all $\nu>0$ and all $\alpha\in (1,\alpha_0)$
$$
\big|p_\pm(\pm i)-1\big|\le C \nu^{-1} \quad \text{and}\quad
\big|p_\pm (\tau)-1\big|\le C\nu^{-1}\tau^{-1}, \quad\tau>0.
$$
\end{lem}

\begin{proof}
The operator $A$ maps the unit function into $L^2(0,\infty)$ and by Minkowski's inequality
\begin{equation}\label{A1mink}
\|A\, 1\|_2 \le \|h_0\|_\infty\nu^{-\frac 1 2}.
\end{equation}
Consequently equations \eqref{qpmeq} are uniquely solvable
for any $\nu>0$ through the fixed-point iterations and $(p_\pm(t)-1; t\ge 0)\in L^2(0,\infty)$.
The rest of the proof is analogous to Lemma \ref{lem-bnds}.
\end{proof}

\subsubsection{Inversion of the Laplace transform}

The solution to \eqref{eigen} is obtained  by inverting the Laplace transform $\widehat \varphi(z)$:

\begin{lem}
Let $(\Phi_0, \Phi_1, \nu)$ satisfy the integro-algebraic system from Lemma \ref{lemia2}, then the function
\begin{multline}
\label{phixx}
\varphi(x)  =
 \frac 2 {\alpha-1}\Re\Big\{i e^{i\nu x}   \Phi_{0}(i\nu ) \Big\}
  + \\
  \frac 1 { \pi  }\int_{0}^\infty \frac {\sin\theta_0(u)}{\gamma_0(u)}
  \Big(e^{- u \nu x}\Phi_{0}(- u\nu)-e^{ -u \nu(1-x)} \Phi_{1}(- u\nu)\Big)d  u
\end{multline}
where $\gamma_0(u) :=|\Lambda^+(\nu u)|$, and $\lambda$ defined by the formula \eqref{nuflan} solve the eigenproblem \eqref{eigen}.
\end{lem}

\begin{proof}
After removal of the singularities, the expression in \eqref{hatvarphi} becomes an entire function and hence 
$$
 \varphi(x)  \propto
\frac 1{2\pi i}\lim_{R\to\infty}\int_{ -Ri}^{  Ri} e^{zx} \frac{1}{\Lambda(z)}\Big(e^{-z} \Phi_{1}(-z)- \Phi_{0}(z)\Big)dz=
 \frac 1{2\pi i} \int_{ -i\infty}^{  i\infty} (f_1(z)-f_2(z))dz,
$$
where
$$
f_1(z) = \frac{e^{z(x-1)} \Phi_{1}(-z)}{\Lambda(z)}\quad\text{and} \quad f_2(z)=   \frac{e^{zx}\Phi_{0}(z)}{\Lambda(z)}.
$$


Integration over the contours depicted at Figure \ref{fig2} yields
\begin{align*}
 \varphi(x)\propto & - \Res\big(f_2,z_0\big) -  \Res\big(f_2,-z_0\big)  \\
&
+\frac 1{2\pi i}\int_{0}^\infty \Big(f_1^+(t)  -  f_1^{-}(t)\Big) dt
   +\frac 1{2\pi i} \int_{-\infty}^0\Big(f_2^+(t)  - f_2^-(t)\Big) dt.
\end{align*}

Plugging the expressions for $f_1$ and $f_2$ we obtain
\begin{align*}
\frac 1 {2\pi i}\int_{0}^\infty \big(f_1^+(t)- f_1^-(t)\big)dt & =
\frac 1 {2\pi i}\int_{0}^\infty e^{t(x-1)} \Phi_{1}(-t) \left(\frac{1}{\Lambda^+(t)}- \frac{1}{\Lambda^-(t)}\right)dt \\
&
=-\frac 1 { \pi  }\int_{0}^\infty e^{t(x-1)} \Phi_{1}(-t)\frac{\sin \theta(t)}{\gamma(t)}  dt
\end{align*}
and, similarly,
\begin{align*}
&
\frac 1 {2\pi i}\int_{-\infty}^0 \big(f_2^+(t)-f_2^-(t)\big)dt =
\frac 1 {2\pi i}\int_{0}^\infty \big(f_2^+(-t)-f_2^-(-t)\big)dt =\\
&
\frac 1 {2\pi i}\int_{0}^\infty e^{-tx}\Phi_{0}(-t)
\left(
\frac{1}{\Lambda^+(-t)}
-\frac{1}{\Lambda^-(-t)}
\right)dt=
\frac 1 { \pi  }\int_{0}^\infty e^{-tx}\Phi_{0}(-t)\frac {\sin\theta(t)}{\gamma(t)}dt.
\end{align*}
Further,
$$
 \Res(f_2,z_0) =    \lim_{z\to z_0}(z-z_0)
\frac{e^{zx}\Phi_{0}(z)}{\Lambda(z)}=
 \frac{e^{z_0x}\Phi_{0}(z_0)}{\Lambda'(z_0)} = -\frac 1 {\alpha-1} e^{z_0x}\Phi_{0}(z_0) z_0
$$
where by the formula \eqref{Lambda2}
$$
\Lambda'(z_0)= - (\alpha-1) \nu^{1-\alpha}   z_0^{\alpha-2} i
 e^{- i \frac {\pi\alpha} 2 }
= -(\alpha-1) z_0^{-1}.
$$
Similarly,
$$
\Res(f_2,-z_0)=
 \lim_{z\to -z_0}(z+z_0)
\frac{e^{zx}\Phi_{0}(z)}{\Lambda(z)}=
\frac{e^{-z_0x}\Phi_{0}(-z_0)}{\Lambda'(-z_0)} =  \frac 1{\alpha-1} e^{-z_0x}\Phi_{0}(-z_0)z_0,
$$
since
$$
\Lambda'(-z_0) =(\alpha-1)(-z_0)^{-1}(-z_0)^{\alpha-1}\nu^{1-\alpha}i e^{i\frac {\pi} 2  \alpha     }=
(\alpha-1)z_0^{-1}.
$$
The formula \eqref{phixx} is obtained by gathering all parts together and using the conjugacies
$\overline{\Phi_{0}(z_0)}=\Phi_{0}(-z_0)$ and $\overline{\Phi_{1}(z_0)}=\Phi_{1}(-z_0)$.

\end{proof}

\subsubsection{Asymptotic analysis}

The conditions $\Im\{\xi\}=0$ and $\Re\{\bar \eta\}=0$ from \eqref{algfbn} hold if and only if
\begin{align*}
& -\frac \nu 2 + \arg\big\{X_0(-i)\big\} + \arg\big\{p_-(i)\big\} + \pi k = 0\\
& -\frac \nu 2 + \arg\big\{X_0(-i)\big\} + \arg\big\{p_+(i)\big\} + \pi m + \frac \pi 2 = 0
\end{align*}
for some integers $k$ and $m$ respectively. Hence the integro-algebraic system
\eqref{qpmfbn} and \eqref{algfbn} decouples into two separate parts
\begin{equation}
\label{subsys1}
\begin{aligned}
p_-( t)  =&  -\frac 1 \pi \int_0^\infty \frac{h_0(\tau)e^{-\nu \tau}}{\tau+  t}p_-(\tau)d\tau + 1 \\
\nu   = &  -\frac {\pi} 2 + \frac {\alpha-1}4 \pi +  2\arctan \frac{\mathrm{Im}\{p_-(i)\}}{\mathrm{Re}\{p_-(i)\}} + 2\pi n,
\quad n \in \mathbb{Z}
%
%
\end{aligned}
\end{equation}
and
\begin{equation}
\label{subsys2}
\begin{aligned}
p_+( t)  = &\phantom{+}\frac 1 \pi \int_0^\infty \frac{h_0(\tau)e^{-\nu\tau}}{\tau+ t}p_+(\tau)d\tau + 1 \\
\nu   = &  -\frac {3\pi} 2 + \frac {\alpha-1}4 \pi +  2\arctan \frac{\mathrm{Im}\{p_+(i)\}}{\mathrm{Re}\{p_+(i)\}} + 2\pi n,
\quad n\in \mathbb{Z}
\end{aligned}
\end{equation}
where we used the formulas from \eqref{argmod}.
%
The following lemma establishes solvability of these systems in two relevant regimes:
\begin{lem} \label{lem66} \

\medskip

\begin{enumerate}
\addtolength{\itemsep}{0.7\baselineskip}
\renewcommand{\theenumi}{\roman{enumi}}
\item\label{cl1}  For any $\alpha\in (1,2)$ there exists an integer $n_\alpha$ such that for any $n\ge n_\alpha$
the systems \eqref{subsys1} and \eqref{subsys2} have unique solutions

\item\label{cl2} There exists a number $\alpha^* >1$ such that for any $\alpha\in (1,\alpha^*)$, the systems \eqref{subsys1} and \eqref{subsys2} have unique solutions with positive
algebraic parts if and only if $n\ge 1$.

\end{enumerate}
\end{lem}

\begin{proof}\
Let us consider the system \eqref{subsys1}, leaving out the similar proof for \eqref{subsys2}.
The unique solution can be constructed by means of fixed point iterations.
Define
$$
f_\alpha(\nu):= \arctan \frac{\mathrm{Im}\{p_-(i)\}}{\mathrm{Re}\{p_-(i)\}}, \quad \nu >0,\; \alpha\in (1,2).
$$
By Lemma \eqref{lemopA} the operator in the integral equation in \eqref{subsys1} is
contracting on $L^2(0,\infty)$ for any $\nu>0$ and any $\alpha\in (1,2)$. Therefore it suffices to show that the map in the right hand side
of the algebraic equation
$$
\nu\mapsto F_{\alpha,n}(\nu) := -\frac {\pi} 2 + \frac {\alpha-1}4 \pi + 2 f_\alpha(\nu) + 2\pi n
$$
is a contraction as well. Denote $\delta(\alpha):= \|h_0\|_\infty = \sin   \frac {3-\alpha} 2 \pi$.
Below we show that

\medskip

\begin{enumerate}\addtolength{\itemsep}{0.7\baselineskip}
\renewcommand{\theenumi}{\alph{enumi}}

\item\label{propa} for all $\alpha>1$ close enough to $1$, any solution of \eqref{subsys1} with positive algebraic part satisfies $\nu\ge 1$

\item\label{propb} $|f_\alpha(\nu)| \le \delta(\alpha)\nu^{-1}$ for all $\nu\ge 1$

\item\label{propc} there exists an absolute constant $C$ such that for any $\nu>0$
$$
|f_\alpha'(\nu)|\le C \delta(\alpha) \nu^{-1/2}
$$
\end{enumerate}

Before checking these properties, let us explain how they imply assertions of the lemma.
Since $\nu = F_{\alpha,n}(\nu)\ge  2\pi (n-1)$ for any $\alpha\in (1,2)$, property \eqref{propc} implies that
$F_{\alpha,n}$ is contracting on its range for all sufficiently large $n\ge 1$, which proves \eqref{cl1}.
To verify  claim \eqref{cl2}, note that by \eqref{propa} and \eqref{propb} for all $\alpha>1$ close enough to $1$,
$F_{\alpha,n}(\nu)\ge \pi/4$ for all $n\ge 1$ and consequently by \eqref{propc} the map $F_{\alpha,n}$ is contracting uniformly over $\nu\ge  \pi/ 4$. This guarantees unique solution of the system
\eqref{subsys1} with  $n\ge 1$.  On the other hand, \eqref{propb} implies that $F_{\alpha, n}(\nu)\le -  \pi/2$ for all $\alpha$ small enough and $n<1$ and hence
no solution with $\nu>0$ exists.

To complete the proof we need to verify the properties \eqref{propa}-\eqref{propc}. An a priori lower bound for $\nu$ can be found
by estimating the smallest eigenvalue $\lambda_1$ of the operator $\widetilde K$, or equivalently, the largest eigenvalue of its inverse. By Theorem 5.1 in \cite{CCK}, $\widetilde K^{-1}$ takes the form of integral operator with the kernel
\begin{equation}\label{kappa}
\kappa(u,v) = \beta_\alpha (uv)^{\frac {\alpha -1}2 }\int_{u\vee v}^{1}r^{1-\alpha} (r-u)^{\frac {\alpha-3} 2}(r-v)^{\frac {\alpha-3} 2}\,dr,
\end{equation}
where
$\beta_\alpha =
\frac
{
\sin \tfrac{\alpha-1}{2}\pi
}
\pi
\frac 1
{
2-\alpha
}
B\big(\tfrac {\alpha+1}2, \tfrac {\alpha-1}2\big) ^{-1}.
$
This kernel is continuous  off the diagonal, where it has a weak singularity. Therefore
\begin{align*}
\lambda_1^{-1} = \big\|\widetilde K^{-1}\big\|_2 \le \sup_{v\in [0,1]}\int_0^1 \kappa(u,v)du = \beta_\alpha B\big(\tfrac {\alpha+1}2, \tfrac {\alpha-1}2\big) \frac {2^{2-\alpha}}{\alpha-1} =
\frac
{
\sin \tfrac{\alpha-1}{2}\pi
}
\pi
\frac {2^{2-\alpha}}{(\alpha-1)(2-\alpha)}.
\end{align*}
%
Plugging this into \eqref{nuflan} we obtain the bound in \eqref{propa}:
$$
\nu \ge   \left(\frac{\Gamma(\alpha-1)\sin \frac {\pi (\alpha-1)} 2 }{ \pi(1-\frac \alpha 2)} \lambda_1\right)^{\frac 1{\alpha-1}}\ge
 2\Gamma(\alpha)^{\frac 1{\alpha-1}} \xrightarrow[]{\alpha\to 1} 2e^{\digamma(1)} = 1.1229...
$$

Further a direct calculation based on the estimate from Lemma \ref{lemopA} and the bound \eqref{A1mink} shows that
$$
\left|\Im\{p_-(i)\}\right| \le \|h_0\|_\infty  \nu^{-1}\quad \text{and}\quad
\left|\Re\{p_-(i)\}-1\right| \le \|h_0\|_\infty  \nu^{-1}
$$
and \eqref{propb} follows. The estimate in \eqref{propc} is obtained similarly.
\end{proof}

From here on it will be convenient to enumerate all the solutions by a single index. To this end, let us define $\nu_{2n}=\nu^-_n$ and
$\nu_{2n-1}=\nu^+_n$ for $n=1,2,...$, where $\nu^-_n$ and $\nu^+_n$ are the algebraic parts of the unique solutions of \eqref{subsys1} and \eqref{subsys2},
whose existence is guaranteed by Lemma \ref{lem66}, at least for all $\alpha$ close enough to $1$.
Plugging the estimates from Lemma \ref{lem-aest} into the algebraic parts of \eqref{subsys1} and \eqref{subsys2}, we obtain
\begin{equation}\label{nusubn}
\nu_n = \pi (n-1) + \frac \pi 2 + \frac {\alpha-1}4 \pi + n^{-1} \tilde r_n(\alpha), \quad n = 1,2,...
\end{equation}
where the residual $\tilde r_n(\alpha)$ is bounded uniformly in $n$ and $\alpha\in (1,\alpha_0)$ for any $\alpha_0\in (1,2)$.

Further, in view of \eqref{PP}, solutions of \eqref{subsys1} and \eqref{subsys2} correspond to
\begin{align*}
\Phi_0(\nu z) &= \Phi_1(\nu z) = c_1\nu^{\frac {\alpha-3} 2}X_0(z)p_-(-z) \\
-\Phi_0(\nu z)  &=  \Phi_{1}(\nu z)= c_2 \nu^{\frac {\alpha-3} 2}X_0(z) p_+(-z)
\end{align*}
respectively and using the estimates of Lemma \ref{lem-aest} in the formula \eqref{phixx}, we obtain
\begin{align*}
\varphi_n(x)  \propto &
 \frac 2 {\alpha-1}\Re\Big\{i e^{i\nu_n x}   X_0(i)\Big\}    \\
&
+ \frac 1 { \pi  }\int_{0}^\infty \frac {\sin\theta_0(u)}{\gamma_0(u)}X_0(-u)
\Big(e^{- u \nu_n x}-(-1)^n e^{ -u \nu_n(1-x)} \Big)d  u + n^{-1} \widetilde r_n (x),
\end{align*}
with residual $\widetilde r_n(x)$, uniformly bounded in $n\ge 1$ and $x\in [0,1]$, for any $\alpha\in (1,2)$.
Using the formulas \eqref{argmod}, we get
$$
\Re\Big\{i e^{i\nu_n x}    X_0(i)    \Big\} =  \;
\big|X_0(i) \big|\cos \Big(\nu_n x + \frac \pi 2 +\arg\{X_0(i)\} \Big)
=
-\sqrt{\frac{\alpha-1}{2}}\cos \Big(\nu_n x -\frac{1+\alpha}{8}\pi \Big),
$$
and consequently
\begin{equation}\label{phinphin}
\begin{aligned}
\varphi_n(x) \propto & \sqrt 2\cos \Big(\nu_n x -\frac{1+\alpha}{8}\pi\Big) \\
& - \frac {\sqrt{\alpha-1}} { \pi  }\int_{0}^\infty \rho_0(u)\Big(e^{- u \nu_n x} -(-1)^n e^{ -u \nu_n(1-x)} \Big)du+ n^{-1} r_n(x),
\end{aligned}
\end{equation}
 where we defined $\rho_0(u):= \dfrac {\sin\theta_0(u)}{\gamma_0(u)}X_0(-u)>0$.
Under this normalization  $\lim_n \|\varphi_{n}\|_2=1$, since the $L^2(0,1)$-norm of the integral term vanishes as $n\to\infty$.

\subsubsection{Enumeration alignment}

As mentioned above the inverse operator has weakly singular kernel \eqref{kappa} and therefore its spectrum
is continuous with respect to $\alpha\in (1,2)$.
Since the residual in \eqref{nusubn} is uniformly bounded in $\alpha$ on any open subinterval of $(1,2)$, the same argument as in
subsection \ref{sec-cal} implies that the 	``natural" enumeration and the one fixed by Lemma \ref{lem66} may differ only by an integer
multiple of $\pi$, independent of $\alpha$. In the case of fractional Brownian motion, this additive constant could be identified by means
of the explicit formulas \eqref{Bmlambda} for $\alpha=1$. For the fractional noise, the spectrum degenerates at $\alpha=1$ and the identification
is to be done otherwise. To this end we will use \eqref{cl2} of Lemma \ref{lem66}, which asserts that for all $\alpha>1$ close enough to $1$, the two enumerations coincide for all $n\ge 1$.

Replacing $\alpha := 2-2H$ in \eqref{nusubn} and \eqref{phinphin}, the expression \eqref{phin1} is obtained with
\begin{equation}\label{rho0fBn}
 \rho_0(u)  = \dfrac{ \sin \theta_0(u) }{\gamma_0(u)}  u^{ -\theta_0(\infty)/\pi } \exp \left(\frac 1 \pi \int_0^\infty \frac{\theta_0(t)-\theta_0(\infty)}{t+u}dt\right)
\end{equation}
where $\theta_0(\infty) =\frac{1+2H}{2}\pi$ and
\begin{align*}
&
\theta_0 (u) =  \arctan  \frac{ \cos(\pi H)}{u^{2H-1}-\sin (\pi H)} \\
&
\gamma_0^2(u)  =\left(1-u^{1-2H}   \sin(\pi H)\right)^2+\left(u^{1-2H}\cos ( \pi H)\right)^2.
\end{align*}

\subsection{Proof of \eqref{phinave}}

Asymptotic approximation \eqref{phinphin} is not sufficiently accurate to imply \eqref{phinave} directly.
We will take an alternative route, based on the explicit solutions of certain integral equations, as explained below.
To this end, we will need the exact asymptotics of scalar products with power functions having integrable singularities at the
boundary points:

\begin{lem}\label{lem67}
For any $\alpha \in (0,2)\setminus \{1\}$ and  $\beta\in (0,1)$,
\begin{equation}
\label{coreq}
\int_0^1 x^{-\beta} \varphi_n(x)dx =  C(\alpha, \beta)\nu_n^{\beta-1} (1+o(1)), \quad \text{as}\ n\to\infty
\end{equation}
with a constant $C(\alpha, \beta)\ne 0$.
\end{lem}

\begin{proof}

We will give the detailed proof for $\alpha\in (0,1)$, omitting similar calculations in the complementary case $\alpha\in (1,2)$.
Clearly the residual in \eqref{phinphin} is negligible on the scale under consideration and
\begin{equation}
\label{Cab1}
\begin{aligned}
&
\nu_n^{1-\beta}\int_{0}^1 x^{-\beta} \cos\Big(\nu_n x  - \frac{1+\alpha}{8}\pi\Big)dx =
\int_{0}^{\nu_n} y^{-\beta} \cos\Big(y  - \frac{1+\alpha}{8}\pi\Big)dy =\\
&
  \cos\frac{1+\alpha}{8}\pi \int_{0}^{\nu_n} y^{-\beta} \cos y\; dy
+
  \sin\frac{1+\alpha}{8}\pi \int_{0}^{\nu_n} y^{-\beta} \sin y\;  dy \xrightarrow[]{n\to\infty} \\
&
 \Gamma(1-\beta) \left(\cos\frac{1+\alpha}{8}\pi   \cos \frac { 1-\beta } 2 \pi
+
 \sin\frac{1+\alpha}{8}\pi   \sin \frac { 1-\beta } 2 \pi\right) = \\
&
  \Gamma(1-\beta)\cos\left(\frac{1+\alpha}{8}\pi-\frac { 1-\beta } 2 \pi\right).
\end{aligned}
\end{equation}
Further
\begin{align*}
&
\int_{0}^1 x^{-\beta}  e^{-\nu_n x\tau}  dx =
(\nu_n\tau)^{\beta-1} \int_{0}^{\nu_n \tau} y^{-\beta}  e^{-y  }  dy \\
&
\int_{0}^1 x^{-\beta} e^{ -\nu_n(1-x) \tau}dx \le  e^{ -\frac 1 2\nu_n  \tau}\int_{0}^{\frac 1 2} x^{-\beta}dx+
 2\int_{\frac 1 2}^1  e^{ -\nu_n(1-x) \tau}dx\le
 e^{ -\frac 1 2\nu_n  \tau}   +
\frac{2}{\nu_n\tau}
\end{align*}
and hence by the dominated convergence
\begin{equation}\label{Cab2}
\begin{aligned}
&
\nu_n^{1-\beta} \int_{0}^1 x^{-\beta}  \int_0^\infty  \rho_0(\tau)
\Big(
e^{-\nu_n x\tau}-(-1)^n e^{-\nu_n(1-x)\tau}
\Big)  d\tau dx=\\
&
    \int_0^\infty \rho_0(\tau) \tau^{\beta-1}
\int_{0}^{\nu_n\tau}  u^{-\beta}
e^{- u }   du d\tau +O(n^{-1})\xrightarrow[]{n\to\infty} \Gamma(1-\beta) \int_0^\infty \rho_0(\tau)  \tau^{\beta-1}   d\tau.
\end{aligned}
\end{equation}

Though both constants in \eqref{Cab1} and \eqref{Cab2} are positive, it is not clear that they do not cancel out in    \eqref{phinphin}, so that the scalar product \eqref{coreq} vanishes.
This is nevertheless the case as the following calculation shows.
The integral in \eqref{Cab2} can be computed by means of the residue theorem, integrating the function
$$
g(z):= \dfrac{X_0(-z)}{\Lambda_0(z)}z^{\beta-1}
$$
over the half circle contours in the upper and lower half planes.
Applying Jordan's lemma we obtain
\begin{align*}
2\pi i \Res\{g,i\} & = \int_{-\infty}^\infty g^+(t)dt=
\int_{-\infty}^0 \dfrac{X_0^-(-t)}{\Lambda_0^+(t)}e^{(\beta-1)\pi i}|t|^{\beta-1}dt
+
\int_0^\infty \dfrac{X_0(-t)}{\Lambda_0^+(t)}t^{\beta-1}dt \\
-2\pi i \Res\{g,-i\} & = \int_{-\infty}^\infty g^-(t)dt  =
\int_{-\infty}^0 \dfrac{X^+_0(-t)}{\Lambda_0^-(t)}e^{-(\beta-1)\pi i}|t|^{\beta-1}dt
+
\int_0^\infty \dfrac{X_0(-t)}{\Lambda_0^-(t)}t^{\beta-1}dt
\end{align*}
and therefore
\begin{align*}
& \Res\big\{g,\; i\big\} + \Res\big\{g,-i\big\}
= \\
&
\frac{1}{2\pi i}\int_0^\infty
\left(
\dfrac{X_0^-(t)}{\Lambda_0^+(-t)}e^{(\beta-1)\pi i}
-
\dfrac{X_0^+(t)}{\Lambda_0^-(-t)}e^{-(\beta-1)\pi i}
\right)t^{\beta-1}dt
\\
&
+
\frac{1}{2\pi i}\int_0^\infty X_0(-t)\left(\dfrac{1}{\Lambda_0^+(t)}-\dfrac{1}{\Lambda_0^-(t)}\right)t^{\beta-1}dt=\\
&
\frac{\sin (\beta-1)\pi}{\pi}\int_0^\infty \dfrac{X_0^-(t)}{\Lambda_0^+(-t)}
t^{\beta-1}dt
-
\frac{1}{ \pi}\int_0^\infty \rho_0(t)t^{\beta-1}dt,
\end{align*}
where we used the property \eqref{prop} and the definition
$$
\rho_0(t) = \frac {\sin\theta_0(t)}{\gamma_0(t)}X_0(-t) =  \frac 1{2i} \left(\frac 1{\Lambda_0^-(t)}-\frac 1{\Lambda_0^+(t)}\right)X_0(-t), \quad t>0.
$$
The residues can be readily computed
\begin{align*}
\Res\{g,i\} = & 
\dfrac{X_0(-i)}{\Lambda'_0(i)}i^{\beta-1}= \dfrac{X_0(-i)}{\alpha-1}i^{\beta-2} \\
\Res\{g, -i\} = & 
\dfrac{X_0(i)}{\Lambda'_0(-i)}(-i)^{\beta-1}=
\dfrac{X_0(i)}{\alpha-1}(-i)^{\beta-2}
\end{align*}
and the following identity is obtained
\begin{align*}
&
\Res\{g,i\}+\Res\{g, -i\} =
 2\Re\left\{\dfrac{X_0(-i)}{\alpha-1}i^{\beta-2}\right\}=\\
&
\frac{2}{\alpha-1}\big|X_0(-i)\big| \cos\left(\arg\{X_0(-i)\}+ \frac \pi 2(\beta-2)\right)= \\
&
\sqrt{\frac{2}{\alpha-1}} \cos\left(-\frac{3-\alpha}{8}\pi + \frac \pi 2(\beta-2)\right)=
-\sqrt{\frac{2}{\alpha-1}} \cos\left(\frac{1+\alpha}{8}\pi-\frac \pi 2 (1-\beta)\right).
\end{align*}

Assembling all parts together gives the asymptotics \eqref{coreq}:
\begin{align*}
\nu_n^{1-\beta}\int_0^1 x^{-\beta} \varphi_n(x)dx \, \xrightarrow[]{n\to\infty}\, &
\Gamma(1-\beta)\frac{\sin(1-\beta)\pi}{\pi} \int_0^\infty \dfrac{X_0^-(t)}{\Lambda_0^+(-t)}
t^{\beta-1}dt = \\
&
\Gamma(1-\beta)\frac{\sin(1-\beta)\pi}{\pi} \int_0^\infty \frac{\rho_0(t)}{h_0(t)}dt >0.
\end{align*}

\end{proof}

To derive asymptotics \eqref{phinave} for $\alpha \in (0,1)$, consider the equation
\begin{equation}
\label{g0}
\int_{0}^1 g_0(y) c_\alpha |x-y|^{-\alpha}dy = 1, \quad x\in (0,1),
\end{equation}
which is known to have the explicit solution (see e.g. Lemma 3 in \cite{LB98}):
$$
g_0(x) =  \frac {2}{c_\alpha} \frac{1}{\Gamma\left(\frac {1-\alpha} 2\right)\Gamma\left(\frac{1+\alpha}2\right)}
 x^{\frac{\alpha-1} 2}( 1-x )^{\frac{\alpha-1} 2},
$$
which is unique since the operator $\widetilde K$ is positive definite.
Applying Lemma \ref{lem67} with $\beta:=\frac{1-\alpha}{2}$, we find that scalar products of $g_0$ with symmetric eigenfunctions satisfy
$$
\langle  g_0, \varphi_n\rangle \sim  \nu_n^{\frac{1-\alpha} 2-1}
$$
and, since  $g_0\in L^2(0,1)$, taking the scalar product of both sides of \eqref{g0}, we obtain \eqref{phinave}:
$$
\langle 1, \varphi_n\rangle = \lambda_n \langle g_0, \varphi_n\rangle \sim \nu_n^{ \alpha/2-3/2 }, \quad n\to\infty.
$$
Similar approach works for $\alpha \in (1,2)$: a direct calculation gives
$$
(\widetilde K  1)(x)= x^{1-\alpha}+(1-x)^{1-\alpha}, \quad x\in (0,1),
$$
and \eqref{phinave} is obtained, applying Lemma \ref{lem67} with $\beta:=\alpha-1$:
$$
\langle 1, \varphi_n\rangle = \lambda_n^{-1}   \langle  \widetilde K  1, \varphi_n\rangle \sim
\nu_n^{1-\alpha}  \nu_n^{\alpha-2} = \nu_n^{-1}.
$$
\qed

\section{Applications: the proofs}\label{sec-7}

In this section we prove the results from Section \ref{sec-3}. 

\subsection{Proof of Proposition \ref{thm2}}

\

\medskip
\noindent
\eqref{ii}
Since $u_\eps, u_0\in L^2(0, 1)$ are symmetric around $\frac 1 2$,
\begin{align*}
&
\|u_\eps-u_0\|^2_2  =    \sum_{n\; \text{odd}}    \left(\frac{ \eps \langle 1, \varphi_n\rangle}{(\eps +\lambda_n)\lambda_n}  \right)^2 \sim
\eps^2 \sum_{n\; \text{odd}}   \left(\frac{ n^{ H-\frac 3 2} }{ \eps +  n^{1-2H} }  \right)^2,
\end{align*}
where we used \eqref{phinave} and \eqref{lambdaU}.
For $H<\frac 2 3$ the series in the right hand side is convergent for any $\eps\ge 0$
and hence $\|u_\eps-u_0\|^2_2\sim \eps^2$. For $H>\frac 2 3$ the series is convergent only for  $\eps>0$,
and thus defines a function, which increases to $\infty$ as $\eps\to 0$. To capture the rate note that
$$
\sum_{n\; \text{odd}}   \left(\frac{ n^{ H-\frac 3 2} }{ \eps +  n^{1-2H} }  \right)^2 \sim
\int_1^\infty \left(\frac{x^{H-\frac 3 2}}{\eps + x^{1-2H}}\right)^2dx \sim \\
\eps^{-3 + \frac 1{2H-1}}\int_0^\infty \frac{y^{6H-5}}{\left(y^{2H-1}+ 1\right)^2}dy
$$
and consequently $\|u_\eps-u_0\|^2_2  \sim \eps^{ \frac {2-2H}{2H-1}} $ in this case.
Similar calculation gives the claimed rate for $H=\frac 2 3$.

\medskip
\noindent
\eqref{thm2-iii}
Using Lemma \ref{lem67} and the symmetry of eigenfunctions,
\begin{align*}
&
 u_\eps(1) =
\;
\eps^{-1} \int_0^1 \big(u_0(x)-u_\eps(x)\big) c_H (1-x)^{2H-2}dx = \\
&
\sum_{n\; \text{odd}}\langle 1, \varphi_n\rangle
 \frac{1}{\lambda_n(\eps+\lambda_n)}
 \int_0^1 \varphi_n(x) c_H (1-x)^{2H-2}dx \sim \\
& \sum_{n\; \text{odd}}
\frac{n^{-\frac 1 2-H}}{ \eps+  n^{1-2H} } \sim
\int_1^\infty \frac{x^{-\frac 1 2-H}}{ \eps+  x^{1-2H} }dx
\sim \eps^{-\frac 1 2} \int_0^\infty \frac{y^{H-\frac 3 2}}{ y^{2H-1} +1 }dy.
\end{align*}
\qed

\medskip 

The convergence rate in $L^2(0,1)$ norm, derived in Proposition \ref{thm2}\eqref{ii}, implies the weak convergence \eqref{weak}
for $\psi\in L^2(0,1)$ with at least the same rate, which does not exceed $\eps$ and breaks down
at $H=\tfrac 2 3$.
Of course, the actual rate can be faster, depending on $\psi$. For example,
if $\psi$ is such that $\widetilde{K}^{-1} \psi$ exists, then for all $H>\tfrac 1 2$
\begin{equation}
\label{eqi}
|\langle u_\eps -u_0 , \psi \rangle| \le 2\eps  \|\widetilde{K}^{-1} \psi\|_2 \|u_0\|_2.
\end{equation}
Indeed, subtracting equations \eqref{ie2nd} and \eqref{ie1st}, we see that   $\delta_\eps = u_\eps -u_0$ satisfies
$$
\eps \delta_\eps + \widetilde K \delta_\eps = -\eps   u_0.
$$
Since $\widetilde K$ is positive definite
$$
\eps \|\delta_\eps\|^2_2\le \eps \|\delta_\eps\|^2_2 + \langle \delta_\eps, \widetilde K \delta_\eps\rangle =
-\eps \langle \delta_\eps , u_0\rangle \le \eps \|\delta_\eps\|_2\| u_0\|_2
$$
and it follows that
$
\|\delta_\eps\|_2  \le  \| u_0\|_2.
$
Consequently
$
\|u_\eps\|_2\le \|\delta_\eps\|_2+\|u_0\|_2 \le 2\|u_0\|_2
$
and
$$
|\langle \delta_\eps, \psi \rangle| = |\langle \widetilde K \delta_\eps, \widetilde{K}^{-1}\psi \rangle|
\le \|\widetilde K \delta_\eps\|_2\|\widetilde{K}^{-1}\psi\|_2\stackrel{\dagger}{=} \eps \|u_\eps\|_2 \|\widetilde{K}^{-1}\psi\|_2\le 2\eps \| u_0\|_2\|\widehat{K}^{-1}\psi\|_2,
$$
where the equality holds, since  $\eps u_\eps=-\widetilde{K}\delta_\eps$. In particular, for $\psi = 1$, which is what we need in \eqref{m1}, a sharper rate
follows from \eqref{eqi} than  from \eqref{ii} of Proposition \ref{thm2}.

\subsection{Proof of Proposition \ref{prop-filt}}

Taking the scalar product of both sides with $\varphi_n$ gives
$$
\langle u_\eps, \varphi_n\rangle = \frac{\lambda_n \varphi_n(1)}{\eps + \lambda_n}
$$
and therefore, letting $C=\sin (\pi H)\Gamma(2H+1)$, we get
\begin{align*}
P_T &= \frac 1 {a^2} g(T,T) = \frac 1 {a^2} \frac 1 T  u_\eps(1) =
\frac 1 {a^2} \frac 1 T \sum_{n\ge 1} \langle u_\eps, \varphi_n\rangle \varphi_n(1)   \\
&=
a^{-\frac {4H}{2H+1}}\eps^{\frac 1{2H+1}} \sum_{n\ge 1} \frac{ \varphi^2_n(1)}{\eps\lambda_n^{-1} + 1}
\simeq
a^{-\frac {4H}{2H+1}}\eps^{\frac 1{2H+1}} \int_1^\infty
  \frac{ (2H+1)}{\eps (\pi x)^{2H+1}/C +1}dx  \\
& \simeq
a^{-\frac {4H}{2H+1}}  C^{\frac 1{2H+1}}\frac {2H+1} \pi  \int_0^\infty
  \frac{ 1}{y^{2H+1} +1}dy  \simeq  a^{-\frac {4H}{2H+1}}  C^{\frac 1{2H+1}} \frac 1 {\sin \frac {\pi}{2H+1}}
\end{align*}
and in turn the claimed formula.\qed

\subsection*{Acknowledgements} We are grateful to  Ya.~Yu.~Nikitin and  A.~I.~Nazarov for drawing
our attention to their papers \cite{NN04ptrf,NN04tpa} and the useful suggestions, which helped to improve presentation of our results.


\def\cprime{$'$} \def\cprime{$'$} \def\cydot{\leavevmode\raise.4ex\hbox{.}}
  \def\cprime{$'$} \def\cprime{$'$} \def\cprime{$'$}

\end{document}

%% file: nu_err.tex
%
%
\definecolor{mycolor1}{rgb}{0.00000,0.44700,0.74100}%
\begin{tikzpicture}[scale = 0.44]

\begin{axis}[%
scaled ticks=false, tick label style={/pgf/number format/fixed},
width=6.028in,
height=4.754in,
at={(1.011in,0.642in)},
scale only axis,
xmin=0,
xmax=40,
xmajorgrids,
ymin=-0.2,
ymax=0.05,
ymajorgrids,
axis background/.style={fill=white}
]
\addplot [color=mycolor1,solid,mark=*,mark options={solid},forget plot]
  table[row sep=crcr]{%
1	-0.0357997713624723\\
2	-0.000469278131831175\\
3	-0.0220441058212213\\
4	0.00442061591464515\\
5	-0.0139119176463307\\
6	0.00365915665094008\\
7	-0.0100563768794117\\
8	0.00296370056874196\\
9	-0.00785182406946916\\
10	0.00246766926819575\\
11	-0.00643001491302897\\
12	0.00211129813628474\\
13	-0.00543610761489788\\
14	0.00184845956270863\\
15	-0.00469983986229039\\
16	0.00165045189296364\\
17	-0.00412959272912161\\
18	0.00149946739276174\\
19	-0.00367168779885674\\
20	0.00138415850115337\\
21	-0.0032925390825369\\
22	0.00129708602800349\\
23	-0.0029699973011077\\
24	0.00123325378307015\\
25	-0.00268880739102428\\
26	0.00118924639009776\\
27	-0.00243807064737211\\
28	0.00116270285167275\\
29	-0.00220974661785078\\
30	0.001151985122263\\
31	-0.0019977387937189\\
32	0.00115596778914551\\
33	-0.00179730580816795\\
34	0.00117389098680576\\
35	-0.00160467407218334\\
36	0.00120526205934368\\
37	-0.00141677500491255\\
38	0.00124978869334313\\
39	-0.00123107048152349\\
40	0.00130733581698905\\
};
\addplot [color=red,solid,mark=*,mark options={solid},forget plot]
  table[row sep=crcr]{%
1	-0.181107649173972\\
2	0.0379390128453636\\
3	-0.0422762807912687\\
4	0.0164779352191573\\
5	-0.0235496617863937\\
6	0.010382348284093\\
7	-0.0162684095707917\\
8	0.00755919337680311\\
9	-0.0124102283746446\\
10	0.00594354251643736\\
11	-0.0100199621446393\\
12	0.00490396469727017\\
13	-0.00838993281586653\\
14	0.00418481838256923\\
15	-0.00720231195394661\\
16	0.00366357068362788\\
17	-0.00629310152762486\\
18	0.00327446766965522\\
19	-0.00556897290083924\\
20	0.0029792738874761\\
21	-0.00497276536091817\\
22	0.0027544001334121\\
23	-0.00446739117967354\\
24	0.0025846331789694\\
25	-0.0040276298167754\\
26	0.00245983016064599\\
27	-0.00363564215832923\\
28	0.00237306319904462\\
29	-0.00327837551242283\\
30	0.00231952248304879\\
31	-0.00294599190972633\\
32	0.00229584004765115\\
33	-0.00263087642503024\\
34	0.00229965604383153\\
35	-0.00232699198146236\\
36	0.00232933379093936\\
37	-0.0020294462372874\\
38	0.0023837650750238\\
39	-0.0017341933394448\\
40	0.00246223620963804\\
};
\end{axis}
\end{tikzpicture}%

%% file: relative_error.tex
%
%
\definecolor{mycolor1}{rgb}{0.00000,0.44700,0.74100}%
\begin{tikzpicture}[scale=0.44]

\begin{axis}[%
width=6.028in,
height=4.754in,
at={(1.011in,0.642in)},
scale only axis,
xmin=0,
xmax=40,
xmajorgrids,
ymin=-10,
ymax=90,
ylabel={Relative error in \%},
ymajorgrids,
axis background/.style={fill=white}
]
\addplot [color=mycolor1,solid,mark=*,mark options={solid},forget plot]
  table[row sep=crcr]{%
1	85.3657896939958\\
2	53.3023416384272\\
3	38.61640198084\\
4	29.583147233911\\
5	24.4068109170332\\
6	20.4183091666859\\
7	17.8102840363685\\
8	15.5810224148333\\
9	14.0149755703545\\
10	12.5946277009701\\
11	11.551332738474\\
12	10.5681289085286\\
13	9.82367612667983\\
14	9.10303010935578\\
15	8.54525518520496\\
16	7.99450885794835\\
17	7.56108223681862\\
18	7.12654335949968\\
19	6.7800820314592\\
20	6.42850706461283\\
21	6.14523744042795\\
22	5.85494983534764\\
23	5.61903468988796\\
24	5.37529960921172\\
25	5.17578303640484\\
26	4.96823597883773\\
27	4.79729808806036\\
28	4.61843716521799\\
29	4.47034736164303\\
30	4.31460921492668\\
31	4.18507235325541\\
32	4.0482442077731\\
33	3.93397775530326\\
34	3.81281114720223\\
35	3.71126345798659\\
36	3.60321333909482\\
37	3.51237099245099\\
38	3.41541516036969\\
39	3.333668200443\\
40	3.24617960811252\\
};
\addplot [color=red,solid,mark=*,mark options={solid},forget plot]
  table[row sep=crcr]{%
1	5.89010955238797\\
2	0.0253155751143845\\
3	0.707265833348227\\
4	-0.101265205366903\\
5	0.24720360620241\\
6	-0.0531991018741044\\
7	0.123538429383822\\
8	-0.0315649057410404\\
9	0.0736957575781913\\
10	-0.0207439313056632\\
11	0.0488196518299204\\
12	-0.0146677924620231\\
13	0.0346428569267587\\
14	-0.0109518150535444\\
15	0.0257950862785135\\
16	-0.00853384203588619\\
17	0.0198929143825667\\
18	-0.00688761254161342\\
19	0.0157481788922733\\
20	-0.00572945849297594\\
21	0.0127149692372556\\
22	-0.00489587231626002\\
23	0.0104177476998054\\
24	-0.00428735379515533\\
25	0.00862612974619867\\
26	-0.00384058639488452\\
27	0.00719242867568438\\
28	-0.00351374298885883\\
29	0.00601844210039713\\
30	-0.00327830784825515\\
31	0.00503681398252186\\
32	-0.00311429895239958\\
33	0.00420010271392608\\
34	-0.00300738565734549\\
35	0.00347410765615294\\
36	-0.00294706389985954\\
37	0.0028336684099158\\
38	-0.00292550393113949\\
39	0.00225995570628776\\
40	-0.0029367758649799\\
};
\end{axis}
\end{tikzpicture}%

%% file: phi25gridon.tex
%
%
\begin{tikzpicture}[scale = 0.45]

\begin{axis}[%
width=6.028in,
height=4.754in,
at={(1.011in,0.642in)},
scale only axis,
xmin=-0.05,
xmax=1.05,
xmajorgrids,
ymin=-2,
ymax=2,
ymajorgrids,
axis background/.style={fill=white}
]
\addplot [color=red,solid,forget plot]
  table[row sep=crcr]{%
0.000125	0.0355015338256694\\
0.001375	0.194238341669053\\
0.002625	0.287323691477655\\
0.003875	0.365416424475122\\
0.005125	0.43545274956969\\
0.006375	0.50010162083662\\
0.007625	0.560696214528103\\
0.008875	0.617996906225349\\
0.010125	0.672472956416788\\
0.011375	0.724428411622208\\
0.012625	0.774066442261785\\
0.013875	0.821524923335769\\
0.015125	0.866897888353056\\
0.016375	0.91024878930129\\
0.017625	0.951619458477156\\
0.018875	0.991036069544801\\
0.020125	1.02851344339056\\
0.021375	1.06405803108979\\
0.022625	1.09767018566368\\
0.023875	1.12934591531304\\
0.025125	1.15907819630736\\
0.026375	1.18685793350254\\
0.027625	1.21267477570885\\
0.028875	1.23651780960319\\
0.030125	1.25837608179999\\
0.031375	1.27823897811866\\
0.032625	1.29609656636995\\
0.033875	1.31193996811106\\
0.035125	1.32576154283315\\
0.036375	1.33755513089948\\
0.037625	1.3473161433945\\
0.038875	1.35504182411024\\
0.040125	1.36073127904136\\
0.041375	1.36438567104068\\
0.042625	1.36600817512496\\
0.043875	1.36560418237272\\
0.045125	1.36318127117576\\
0.046375	1.35874933166084\\
0.047625	1.35232050466005\\
0.048875	1.34390928539\\
0.050125	1.33353251186591\\
0.051375	1.3212094132433\\
0.052625	1.30696155910827\\
0.053875	1.29081287402271\\
0.055125	1.27278965338312\\
0.056375	1.25292054923333\\
0.057625	1.23123651880264\\
0.058875	1.20777080790237\\
0.060125	1.18255893911464\\
0.061375	1.1556386754775\\
0.062625	1.12704997856856\\
0.063875	1.09683493375478\\
0.065125	1.06503774648137\\
0.066375	1.03170468144898\\
0.067625	0.996884001364314\\
0.068875	0.960625890081128\\
0.070125	0.922982417416982\\
0.071375	0.884007477346288\\
0.072625	0.843756702831632\\
0.073875	0.802287394086857\\
0.075125	0.759658448960201\\
0.076375	0.715930303337854\\
0.077625	0.671164834706971\\
0.078875	0.62542527705427\\
0.080125	0.578776141449896\\
0.081375	0.531283151604714\\
0.082625	0.483013128622522\\
0.083875	0.434033913327944\\
0.085125	0.38441426582392\\
0.086375	0.334223805008711\\
0.087625	0.283532871985897\\
0.088875	0.2324124631429\\
0.090125	0.180934135033364\\
0.091375	0.129169877296297\\
0.092625	0.0771920634398462\\
0.093875	0.0250733096336298\\
0.095125	-0.0271135891356513\\
0.096375	-0.0792958015929416\\
0.097625	-0.131400505050041\\
0.098875	-0.183355033549822\\
0.100125	-0.235086930646292\\
0.101375	-0.286524106142592\\
0.102625	-0.33759486992877\\
0.103875	-0.388228082570695\\
0.105125	-0.43835321336713\\
0.106375	-0.487900476696909\\
0.107625	-0.536800871137374\\
0.108875	-0.584986345958134\\
0.110125	-0.632389825006593\\
0.111375	-0.678945355211559\\
0.112625	-0.724588135938721\\
0.113875	-0.769254683743598\\
0.115125	-0.812882844459053\\
0.116375	-0.855411940377647\\
0.117625	-0.896782789571091\\
0.118875	-0.936937856943128\\
0.120125	-0.975821289124063\\
0.121375	-1.01337897528962\\
0.122625	-1.04955868900779\\
0.123875	-1.0843100824882\\
0.125125	-1.11758481140161\\
0.126375	-1.14933656042017\\
0.127625	-1.17952115383336\\
0.128875	-1.20809656579088\\
0.130125	-1.23502299823321\\
0.131375	-1.26026295644724\\
0.132625	-1.28378128238426\\
0.133875	-1.3055451868863\\
0.135125	-1.32552430804172\\
0.136375	-1.34369077265774\\
0.137625	-1.36001921691684\\
0.138875	-1.37448680621194\\
0.140125	-1.38707327882304\\
0.141375	-1.39776099423608\\
0.142625	-1.40653493190818\\
0.143875	-1.4133827187704\\
0.145125	-1.41829463723115\\
0.146375	-1.42126365286251\\
0.147625	-1.42228542973288\\
0.148875	-1.42135831460708\\
0.150125	-1.41848334451657\\
0.151375	-1.41366425272892\\
0.152625	-1.40690747320867\\
0.153875	-1.39822210849294\\
0.155125	-1.38761992286009\\
0.156375	-1.37511533674059\\
0.157625	-1.36072539823919\\
0.158875	-1.34446977463103\\
0.160125	-1.32637069499884\\
0.161375	-1.30645293305238\\
0.162625	-1.28474379352999\\
0.163875	-1.26127305772751\\
0.165125	-1.23607292833029\\
0.166375	-1.20917799390142\\
0.167625	-1.18062520305983\\
0.168875	-1.15045379618558\\
0.170125	-1.11870523705183\\
0.171375	-1.08542316598387\\
0.172625	-1.05065335938569\\
0.173875	-1.01444364572674\\
0.175125	-0.976843840671849\\
0.176375	-0.937905674787612\\
0.177625	-0.89768273199842\\
0.178875	-0.856230383828046\\
0.180125	-0.813605691491737\\
0.181375	-0.769867330258966\\
0.182625	-0.725075520652518\\
0.183875	-0.679291948986937\\
0.185125	-0.632579664303532\\
0.186375	-0.585002992373182\\
0.187625	-0.536627464126871\\
0.188875	-0.487519709839444\\
0.190125	-0.437747386917896\\
0.191375	-0.38737905301786\\
0.192625	-0.336484089577781\\
0.193875	-0.285132622970881\\
0.195125	-0.233395411289405\\
0.196375	-0.181343732233675\\
0.197625	-0.129049298494013\\
0.198875	-0.0765841754360777\\
0.200125	-0.024020665842943\\
0.201375	0.0285688050459015\\
0.202625	0.0811117747542267\\
0.203875	0.133535827581731\\
0.205125	0.185768715111326\\
0.206375	0.237738453847956\\
0.207625	0.289373429465646\\
0.208875	0.340602480835043\\
0.210125	0.391354997932713\\
0.211375	0.441561033888685\\
0.212625	0.49115140299547\\
0.213875	0.5400577586672\\
0.215125	0.58821269094044\\
0.216375	0.63554983182829\\
0.217625	0.682003949225186\\
0.218875	0.727511016039798\\
0.220125	0.772008317759209\\
0.221375	0.815434518831309\\
0.222625	0.857729778298201\\
0.223875	0.898835813967147\\
0.225125	0.938695969751525\\
0.226375	0.977255307931285\\
0.227625	1.01446069923053\\
0.228875	1.05026088072927\\
0.230125	1.08460651390956\\
0.231375	1.11745026638576\\
0.232625	1.14874689239581\\
0.233875	1.17845327876317\\
0.235125	1.20652849158617\\
0.236375	1.23293385216995\\
0.237625	1.25763298788129\\
0.238875	1.28059188688501\\
0.240125	1.30177892940493\\
0.241375	1.32116493514158\\
0.242625	1.33872321439881\\
0.243875	1.35442960660468\\
0.245125	1.36826249584555\\
0.246375	1.38020284795459\\
0.247625	1.39023424418726\\
0.248875	1.39834290785678\\
0.250125	1.40451770231733\\
0.251375	1.40875016596068\\
0.252625	1.41103450614775\\
0.253875	1.41136763877865\\
0.255125	1.40974914605464\\
0.256375	1.40618132675737\\
0.257625	1.40066914668792\\
0.258875	1.39322027803523\\
0.260125	1.38384503963503\\
0.261375	1.37255643181885\\
0.262625	1.35937007172554\\
0.263875	1.34430421949367\\
0.265125	1.32737970260573\\
0.266375	1.30861993412463\\
0.267625	1.28805084709382\\
0.268875	1.26570087604919\\
0.270125	1.24160091002676\\
0.271375	1.21578425249862\\
0.272625	1.1882865761112\\
0.273875	1.15914587374735\\
0.275125	1.12840240392183\\
0.276375	1.09609863562365\\
0.277625	1.06227919156631\\
0.278875	1.02699078674883\\
0.280125	0.990282160999598\\
0.281375	0.952204013581757\\
0.282625	0.912808936489295\\
0.283875	0.872151327909294\\
0.285125	0.830287347717749\\
0.286375	0.787274795469457\\
0.287625	0.743173079458751\\
0.288875	0.698043088269769\\
0.290125	0.651947152111934\\
0.291375	0.604948908646025\\
0.292625	0.557113264910693\\
0.293875	0.508506260395017\\
0.295125	0.459195021759405\\
0.296375	0.409247622344517\\
0.297625	0.358733034536608\\
0.298875	0.307721001704907\\
0.300125	0.256281956254979\\
0.301375	0.204486918468945\\
0.302625	0.15240740006929\\
0.303875	0.100115305357881\\
0.305125	0.0476828317815866\\
0.306375	-0.00481763084543833\\
0.307625	-0.0573135987140224\\
0.308875	-0.109732593932223\\
0.310125	-0.162002245219585\\
0.311375	-0.214050389062543\\
0.312625	-0.265805166581678\\
0.313875	-0.317195121260741\\
0.315125	-0.368149312525331\\
0.316375	-0.418597380182271\\
0.317625	-0.46846968804547\\
0.318875	-0.517697368397501\\
0.320125	-0.566212469098957\\
0.321375	-0.613947997017169\\
0.322625	-0.660838060275289\\
0.323875	-0.706817907390449\\
0.325125	-0.751824069592944\\
0.326375	-0.795794396815954\\
0.327625	-0.838668194138806\\
0.328875	-0.880386256920476\\
0.330125	-0.920890988245274\\
0.331375	-0.960126464423819\\
0.332625	-0.998038512526075\\
0.333875	-1.03457478521657\\
0.335125	-1.06968483447285\\
0.336375	-1.10332018166233\\
0.337625	-1.13543438313786\\
0.338875	-1.16598309319156\\
0.340125	-1.19492412735431\\
0.341375	-1.22221752106843\\
0.342625	-1.2478255829112\\
0.343875	-1.2717129467485\\
0.345125	-1.2938466179225\\
0.346375	-1.31419603700158\\
0.347625	-1.33273308148959\\
0.348875	-1.34943215874995\\
0.350125	-1.36427018347412\\
0.351375	-1.37722666844499\\
0.352625	-1.38828369381608\\
0.353875	-1.39742599015996\\
0.355125	-1.40464090094588\\
0.356375	-1.40991845947544\\
0.357625	-1.41325134301753\\
0.358875	-1.41463494233675\\
0.360125	-1.41406731204403\\
0.361375	-1.41154921450848\\
0.362625	-1.4070841006946\\
0.363875	-1.40067810752844\\
0.365125	-1.39234005132956\\
0.366375	-1.3820814143406\\
0.367625	-1.36991632966337\\
0.368875	-1.35586156102235\\
0.370125	-1.33993648002287\\
0.371375	-1.32216303954249\\
0.372625	-1.30256574365973\\
0.373875	-1.2811716125808\\
0.375125	-1.2580101468801\\
0.376375	-1.23311328171144\\
0.377625	-1.20651536424603\\
0.378875	-1.17825305973123\\
0.380125	-1.14836536319386\\
0.381375	-1.11689348010411\\
0.382625	-1.08388083560119\\
0.383875	-1.04937294778898\\
0.385125	-1.01341743062605\\
0.386375	-0.976063863567321\\
0.387625	-0.937363788814368\\
0.388875	-0.897370573435648\\
0.390125	-0.856139401763234\\
0.391375	-0.813727137628233\\
0.392625	-0.770192292255639\\
0.393875	-0.725594921604086\\
0.395125	-0.679996547223988\\
0.396375	-0.633460072897251\\
0.397625	-0.58604969678447\\
0.398875	-0.537830823803371\\
0.400125	-0.488869973478996\\
0.401375	-0.439234689686044\\
0.402625	-0.388993447928245\\
0.403875	-0.338215559258583\\
0.405125	-0.286971075165273\\
0.406375	-0.23533069137468\\
0.407625	-0.183365645839793\\
0.408875	-0.131147642329965\\
0.410125	-0.0787487024854533\\
0.411375	-0.026241132917979\\
0.412625	0.0263026448304253\\
0.413875	0.0788101233397519\\
0.415125	0.131208881486061\\
0.416375	0.183426612594955\\
0.417625	0.23539129518014\\
0.418875	0.287031220053582\\
0.420125	0.338275163634123\\
0.421375	0.389052411689643\\
0.422625	0.439292926657327\\
0.423875	0.488927391790373\\
0.425125	0.537887332581213\\
0.426375	0.586105204956986\\
0.427625	0.633514488911592\\
0.428875	0.680049779797361\\
0.430125	0.725646878730933\\
0.431375	0.77024288273969\\
0.432625	0.81377626943117\\
0.433875	0.856186982481585\\
0.435125	0.897416515341813\\
0.436375	0.937407991399804\\
0.437625	0.976106242357714\\
0.438875	1.01345788937826\\
0.440125	1.04941139243287\\
0.441375	1.08391717452327\\
0.442625	1.11692761820673\\
0.443875	1.14839720821126\\
0.445125	1.17828251643286\\
0.446375	1.2065423399213\\
0.447625	1.23313768049262\\
0.448875	1.25803187548055\\
0.450125	1.28119057035204\\
0.451375	1.302581844753\\
0.452625	1.32217617845826\\
0.453875	1.33994656538187\\
0.455125	1.35586849636744\\
0.456375	1.36992001804039\\
0.457625	1.38208175805324\\
0.458875	1.39233695262891\\
0.460125	1.40067146848709\\
0.461375	1.40707382206178\\
0.462625	1.41153519699915\\
0.463875	1.41404945564397\\
0.465125	1.41461314678339\\
0.466375	1.41322551106093\\
0.467625	1.40988848154528\\
0.468875	1.40460668227907\\
0.470125	1.397387425555\\
0.471375	1.38824067745677\\
0.472625	1.37717909802445\\
0.473875	1.36421795108067\\
0.475125	1.3493751611536\\
0.476375	1.33267120989102\\
0.477625	1.31412918597019\\
0.478875	1.29377467629392\\
0.480125	1.27163580793161\\
0.481375	1.24774313091376\\
0.482625	1.2221296542725\\
0.483875	1.19483072331508\\
0.485125	1.16588404320443\\
0.486375	1.13532957364825\\
0.487625	1.10320949757504\\
0.488875	1.06956815800681\\
0.490125	1.03445199861425\\
0.491375	0.997909497483976\\
0.492625	0.95999110089812\\
0.493875	0.920749153738939\\
0.495125	0.880237828136405\\
0.496375	0.838513047577638\\
0.497625	0.79563241002139\\
0.498875	0.751655107988584\\
0.500125	0.70664185378201\\
0.501375	0.660654772807537\\
0.502625	0.613757350755556\\
0.503875	0.566014313044786\\
0.505125	0.51749159356959\\
0.506375	0.468256141049436\\
0.507625	0.418375928579509\\
0.508875	0.367919798399\\
0.510125	0.316957428176505\\
0.511375	0.265559133617059\\
0.512625	0.213795872426193\\
0.513875	0.161739085755045\\
0.515125	0.109460660184029\\
0.516375	0.0570327275873111\\
0.517625	0.00452766690818374\\
0.518875	-0.0479820571681958\\
0.520125	-0.100423942785169\\
0.521375	-0.152725601488561\\
0.522625	-0.204814859660543\\
0.523875	-0.256619815985945\\
0.525125	-0.308068940448347\\
0.526375	-0.359091214602978\\
0.527625	-0.409616230389387\\
0.528875	-0.459574244982797\\
0.530125	-0.508896276402836\\
0.531375	-0.5575142416074\\
0.532625	-0.605361050565999\\
0.533875	-0.652370644982449\\
0.535125	-0.698478143684558\\
0.536375	-0.743619867169822\\
0.537625	-0.787733533003746\\
0.538875	-0.830758234689591\\
0.540125	-0.872634589664268\\
0.541375	-0.913304756517393\\
0.542625	-0.952712624010738\\
0.543875	-0.990803779002315\\
0.545125	-1.02752564708663\\
0.546375	-1.06282749892432\\
0.547625	-1.09666063199264\\
0.548875	-1.1289783255135\\
0.550125	-1.15973597229221\\
0.551375	-1.18889108406411\\
0.552625	-1.21640340655595\\
0.553875	-1.2422349754287\\
0.555125	-1.26635012306961\\
0.556375	-1.28871552640728\\
0.557625	-1.30930030088087\\
0.558875	-1.32807604181315\\
0.560125	-1.34501681834372\\
0.561375	-1.36009920707573\\
0.562625	-1.37330237364365\\
0.563875	-1.38460809977836\\
0.565125	-1.39400075075751\\
0.566375	-1.40146735468868\\
0.567625	-1.4069975503842\\
0.568875	-1.41058372023109\\
0.570125	-1.4122208817193\\
0.571375	-1.41190676509945\\
0.572625	-1.40964174554859\\
0.573875	-1.40542896290233\\
0.575125	-1.39927419626412\\
0.576375	-1.39118592865113\\
0.577625	-1.38117526217073\\
0.578875	-1.36925602567196\\
0.580125	-1.35544463241978\\
0.581375	-1.33976013128843\\
0.582625	-1.32222411840172\\
0.583875	-1.30286076983754\\
0.585125	-1.28169680872812\\
0.586375	-1.25876141748775\\
0.587625	-1.23408619748915\\
0.588875	-1.20770517631434\\
0.590125	-1.17965476100076\\
0.591375	-1.14997363694641\\
0.592625	-1.11870271322816\\
0.593875	-1.08588511926479\\
0.595125	-1.05156614402345\\
0.596375	-1.01579310994817\\
0.597625	-0.978615371082952\\
0.598875	-0.94008416834588\\
0.600125	-0.900252688120435\\
0.601375	-0.859175859443968\\
0.602625	-0.816910355145127\\
0.603875	-0.773514436005159\\
0.605125	-0.729048001405616\\
0.606375	-0.683572375892474\\
0.607625	-0.637150302067571\\
0.608875	-0.589845775657943\\
0.610125	-0.541724089693361\\
0.611375	-0.492851612056219\\
0.612625	-0.443295772491028\\
0.613875	-0.393124903807409\\
0.615125	-0.342408213601125\\
0.616375	-0.291215690491601\\
0.617625	-0.239617952078667\\
0.618875	-0.187686147861039\\
0.620125	-0.135491914612627\\
0.621375	-0.0831072798952784\\
0.622625	-0.0306045052898557\\
0.623875	0.0219440115370376\\
0.625125	0.0744657836318517\\
0.626375	0.126888331563879\\
0.627625	0.179139356562215\\
0.628875	0.231146767141194\\
0.630125	0.282838865766529\\
0.631375	0.334144305011556\\
0.632625	0.384992327848208\\
0.633875	0.435312780124289\\
0.635125	0.485036293237589\\
0.636375	0.53409423675032\\
0.637625	0.582418956103551\\
0.638875	0.629943780142392\\
0.640125	0.676603199090846\\
0.641375	0.7223328119053\\
0.642625	0.767069557530449\\
0.643875	0.810751717320006\\
0.645125	0.853319070569795\\
0.646375	0.894712907568679\\
0.647625	0.934876107857672\\
0.648875	0.973753280047147\\
0.650125	1.01129083496502\\
0.651375	1.04743700484426\\
0.652625	1.08214191177489\\
0.653875	1.11535769780884\\
0.655125	1.14703858734321\\
0.656375	1.17714089600501\\
0.657625	1.20562308689323\\
0.658875	1.23244590489255\\
0.660125	1.25757235432713\\
0.661375	1.28096784007247\\
0.662625	1.30260006765793\\
0.663875	1.32243923513756\\
0.665125	1.34045798604413\\
0.666375	1.35663153440566\\
0.667625	1.3709375529649\\
0.668875	1.38335635120349\\
0.670125	1.39387081344578\\
0.671375	1.40246651105077\\
0.672625	1.40913157533437\\
0.673875	1.4138568609685\\
0.675125	1.41663587083428\\
0.676375	1.41746483740331\\
0.677625	1.41634265543324\\
0.678875	1.41327088279183\\
0.680125	1.40825379845864\\
0.681375	1.40129839446887\\
0.682625	1.39241430963438\\
0.683875	1.38161381403353\\
0.685125	1.36891185369915\\
0.686375	1.35432602723612\\
0.687625	1.33787650447276\\
0.688875	1.31958599747875\\
0.690125	1.29947980403068\\
0.691375	1.27758569813775\\
0.692625	1.25393398051531\\
0.693875	1.2285572889126\\
0.695125	1.201490701905\\
0.696375	1.1727715997293\\
0.697625	1.14243970341551\\
0.698875	1.11053687082451\\
0.700125	1.07710718801283\\
0.701375	1.04219681787909\\
0.702625	1.00585402727049\\
0.703875	0.968128970607583\\
0.705125	0.929073770196787\\
0.706375	0.888742353992441\\
0.707625	0.84719045656416\\
0.708875	0.804475467895724\\
0.710125	0.760656352616963\\
0.711375	0.715793630066629\\
0.712625	0.669949289901315\\
0.713875	0.623186647794994\\
0.715125	0.57557025591739\\
0.716375	0.527165875442379\\
0.717625	0.478040386172337\\
0.718875	0.428261633470665\\
0.720125	0.377898333828446\\
0.721375	0.32702005689275\\
0.722625	0.275697052180793\\
0.723875	0.224000246134623\\
0.725125	0.172000989101216\\
0.726375	0.119771110568521\\
0.727625	0.0673827277013719\\
0.728875	0.0149082396405791\\
0.730125	-0.0375799279271294\\
0.731375	-0.0900092680147031\\
0.732625	-0.142307386253494\\
0.733875	-0.194402004583505\\
0.735125	-0.246221219708892\\
0.736375	-0.297693444681804\\
0.737625	-0.348747603125872\\
0.738875	-0.399313147286667\\
0.740125	-0.449320232854909\\
0.741375	-0.498699818952725\\
0.742625	-0.547383698299476\\
0.743875	-0.595304590556504\\
0.745125	-0.642396297751223\\
0.746375	-0.688593800206878\\
0.747625	-0.733833278152541\\
0.748875	-0.77805220113913\\
0.750125	-0.821189476871029\\
0.751375	-0.86318553927834\\
0.752625	-0.90398234712619\\
0.753875	-0.943523546130955\\
0.755125	-0.981754446745399\\
0.756375	-1.01862226645978\\
0.757625	-1.05407603728305\\
0.758875	-1.0880667750422\\
0.760125	-1.12054744507215\\
0.761375	-1.151473197035\\
0.762625	-1.18080125909337\\
0.763875	-1.20849109659451\\
0.765125	-1.23450436546522\\
0.766375	-1.25880513688128\\
0.767625	-1.28135977656217\\
0.768875	-1.30213709330708\\
0.770125	-1.32110829501622\\
0.771375	-1.33824711470963\\
0.772625	-1.35352984883265\\
0.773875	-1.36693531971779\\
0.775125	-1.3784449034672\\
0.776375	-1.38804262474759\\
0.777625	-1.3957151828818\\
0.778875	-1.40145189654821\\
0.780125	-1.40524471947841\\
0.781375	-1.40708832072065\\
0.782625	-1.40698009483205\\
0.783875	-1.40492007502965\\
0.785125	-1.40091102206747\\
0.786375	-1.39495831248787\\
0.787625	-1.38707011195669\\
0.788875	-1.377257184409\\
0.790125	-1.36553298368589\\
0.791375	-1.35191352690947\\
0.792625	-1.33641755392131\\
0.793875	-1.31906632103223\\
0.795125	-1.2998836788292\\
0.796375	-1.278895929709\\
0.797625	-1.25613197603085\\
0.798875	-1.23162309645692\\
0.800125	-1.2054030124329\\
0.801375	-1.17750774956443\\
0.802625	-1.14797567950152\\
0.803875	-1.11684746946029\\
0.805125	-1.08416594960974\\
0.806375	-1.04997605406495\\
0.807625	-1.01432483373454\\
0.808875	-0.977261391711665\\
0.810125	-0.938836741407293\\
0.811375	-0.89910373354365\\
0.812625	-0.858117059883831\\
0.813875	-0.815933178455581\\
0.815125	-0.772610140620119\\
0.816375	-0.728207604616738\\
0.817625	-0.682786639775517\\
0.818875	-0.636409833357242\\
0.820125	-0.589141013136498\\
0.821375	-0.541045272542663\\
0.822625	-0.492188767530417\\
0.823875	-0.442638815878868\\
0.825125	-0.392463613897561\\
0.826375	-0.341732254968807\\
0.827625	-0.290514520921751\\
0.828875	-0.2388809770448\\
0.830125	-0.186902682789171\\
0.831375	-0.134651208589599\\
0.832625	-0.0821984408192336\\
0.833875	-0.0296165787225931\\
0.835125	0.0230219644047214\\
0.836375	0.0756447389691819\\
0.837625	0.128179360469965\\
0.838875	0.180553528986231\\
0.840125	0.232695133768096\\
0.841375	0.284532423672471\\
0.842625	0.335994113952725\\
0.843875	0.387009401205744\\
0.845125	0.437508063584258\\
0.846375	0.487420653716549\\
0.847625	0.536678500999776\\
0.848875	0.585213918280328\\
0.850125	0.632960103841172\\
0.851375	0.679851426572594\\
0.852625	0.725823403057792\\
0.853875	0.770812899967353\\
0.855125	0.814758027281317\\
0.856375	0.857598419994856\\
0.857625	0.899275204303624\\
0.858875	0.939731195878242\\
0.860125	0.978910782559082\\
0.861375	1.01676019737419\\
0.862625	1.05322747656414\\
0.863875	1.08826262935362\\
0.865125	1.12181760713506\\
0.866375	1.15384637069426\\
0.867625	1.18430503338384\\
0.868875	1.21315192360743\\
0.870125	1.24034755912684\\
0.871375	1.26585470480817\\
0.872625	1.28963850254329\\
0.873875	1.31166652172424\\
0.875125	1.3319087211184\\
0.876375	1.35033749264928\\
0.877625	1.36692779912242\\
0.878875	1.38165711050377\\
0.880125	1.39450555321488\\
0.881375	1.40545573829937\\
0.882625	1.41449298582081\\
0.883875	1.42160522785785\\
0.885125	1.42678314225827\\
0.886375	1.4300199659807\\
0.887625	1.43131170690261\\
0.888875	1.4306570290035\\
0.890125	1.42805737525468\\
0.891375	1.42351676236407\\
0.892625	1.4170419802012\\
0.893875	1.4086424610452\\
0.895125	1.39833036846854\\
0.896375	1.38612047814898\\
0.897625	1.37203015854716\\
0.898875	1.35607943119872\\
0.900125	1.33829094143209\\
0.901375	1.31868984592049\\
0.902625	1.29730377932711\\
0.903875	1.27416289956944\\
0.905125	1.24929984596087\\
0.906375	1.22274961296631\\
0.907625	1.1945495018844\\
0.908875	1.1647391756805\\
0.910125	1.1333604994204\\
0.911375	1.10045760971523\\
0.912625	1.06607664446933\\
0.913875	1.03026589233422\\
0.915125	0.993075598849155\\
0.916375	0.954558024884996\\
0.917625	0.914767165907133\\
0.918875	0.873758889857929\\
0.920125	0.831590734543978\\
0.921375	0.788321956142612\\
0.922625	0.74401323770337\\
0.923875	0.698726819061691\\
0.925125	0.652526286916372\\
0.926375	0.605476592864555\\
0.927625	0.55764386029429\\
0.928875	0.50909529586082\\
0.930125	0.45989918636209\\
0.931375	0.410124803982048\\
0.932625	0.359842227990671\\
0.933875	0.309122253186189\\
0.935125	0.25803638061395\\
0.936375	0.20665672203753\\
0.937625	0.155055817392371\\
0.938875	0.103306538660176\\
0.940125	0.0514821035583923\\
0.941375	-0.000344129903046792\\
0.942625	-0.0520986308193263\\
0.943875	-0.103708008291986\\
0.945125	-0.155098886327939\\
0.946375	-0.206198132127095\\
0.947625	-0.256932817768344\\
0.948875	-0.307230534009498\\
0.950125	-0.357019258664183\\
0.951375	-0.406227580606808\\
0.952625	-0.454784655081728\\
0.953875	-0.502620513901729\\
0.955125	-0.549665923283021\\
0.956375	-0.595852601677132\\
0.957625	-0.641113187741628\\
0.958875	-0.685381426225549\\
0.960125	-0.728592240714259\\
0.961375	-0.770681713212798\\
0.962625	-0.811587147103121\\
0.963875	-0.851247218088096\\
0.965125	-0.88960202683158\\
0.966375	-0.926593057839512\\
0.967625	-0.962163219966298\\
0.968875	-0.996256970582213\\
0.970125	-1.02882033976997\\
0.971375	-1.05980083327294\\
0.972625	-1.08914754393835\\
0.973875	-1.11681100874979\\
0.975125	-1.14274340463543\\
0.976375	-1.16689827361449\\
0.977625	-1.18923058706374\\
0.978875	-1.20969650469164\\
0.980125	-1.22825346579617\\
0.981375	-1.24485976913638\\
0.982625	-1.25947446346706\\
0.983875	-1.27205687709302\\
0.985125	-1.28256641036635\\
0.986375	-1.2909617125456\\
0.987625	-1.29720002261284\\
0.988875	-1.30123594073973\\
0.990125	-1.30301997127903\\
0.991375	-1.30249622349579\\
0.992625	-1.29959881581972\\
0.993875	-1.29424625937848\\
0.995125	-1.28633189370955\\
0.996375	-1.27570556566421\\
0.997625	-1.26213315931546\\
0.998875	-1.2451806516068\\
};
\addplot [color=blue,solid,forget plot]
  table[row sep=crcr]{%
0.000125	0.0378552659906866\\
0.001375	0.207668717132722\\
0.002625	0.307530665329465\\
0.003875	0.391358362548505\\
0.005125	0.466505422916246\\
0.006375	0.535785852887787\\
0.007625	0.6005942513272\\
0.008875	0.661717668396703\\
0.010125	0.719634956252825\\
0.011375	0.774650803233163\\
0.012625	0.826963840444969\\
0.013875	0.876704612376601\\
0.015125	0.923958172927655\\
0.016375	0.968778390296878\\
0.017625	1.01119728227608\\
0.018875	1.05123145795227\\
0.020125	1.08888659175039\\
0.021375	1.12416071045876\\
0.022625	1.15704660604429\\
0.023875	1.18753364089646\\
0.025125	1.21560913801401\\
0.026375	1.24125950270151\\
0.027625	1.26447108629678\\
0.028875	1.28523085238196\\
0.030125	1.30352693258129\\
0.031375	1.31934911606929\\
0.032625	1.33268922754283\\
0.033875	1.34354138035425\\
0.035125	1.35190229336504\\
0.036375	1.35777145202487\\
0.037625	1.36115135398006\\
0.038875	1.36204755381993\\
0.040125	1.36046888831969\\
0.041375	1.35642742761505\\
0.042625	1.34993874435824\\
0.043875	1.34102176604719\\
0.045125	1.32969896735732\\
0.046375	1.31599626662486\\
0.047625	1.29994315772933\\
0.048875	1.28157261162854\\
0.050125	1.26092113806953\\
0.051375	1.23802866583325\\
0.052625	1.2129385944024\\
0.053875	1.18569774680476\\
0.055125	1.15635626590306\\
0.056375	1.12496754614829\\
0.057625	1.091588225635\\
0.058875	1.05627811139057\\
0.060125	1.0191000041595\\
0.061375	0.980119703086785\\
0.062625	0.939405838018115\\
0.063875	0.897029864388914\\
0.065125	0.853065826406175\\
0.066375	0.807590342950232\\
0.067625	0.760682417343351\\
0.068875	0.712423410871177\\
0.070125	0.662896786777283\\
0.071375	0.612188074919855\\
0.072625	0.560384674659425\\
0.073875	0.507575772918531\\
0.075125	0.453852135408431\\
0.076375	0.399306025285461\\
0.077625	0.344030956265437\\
0.078875	0.28812166114672\\
0.080125	0.231673836881167\\
0.081375	0.17478406087341\\
0.082625	0.117549539286612\\
0.083875	0.0600680593032332\\
0.085125	0.00243773754757386\\
0.086375	-0.0552430810657685\\
0.087625	-0.112876042725128\\
0.088875	-0.170362917857054\\
0.090125	-0.227605728518542\\
0.091375	-0.284506992259593\\
0.092625	-0.340969815677228\\
0.093875	-0.396898102072692\\
0.095125	-0.452196654019663\\
0.096375	-0.506771406632785\\
0.097625	-0.560529524076173\\
0.098875	-0.613379593264921\\
0.100125	-0.665231733732314\\
0.101375	-0.715997801915898\\
0.102625	-0.765591495686309\\
0.103875	-0.813928519939424\\
0.105125	-0.860926714836008\\
0.106375	-0.906506206655473\\
0.107625	-0.950589525033593\\
0.108875	-0.993101755546592\\
0.110125	-1.03397063915676\\
0.111375	-1.07312671310585\\
0.112625	-1.11050341398441\\
0.113875	-1.14603720014904\\
0.115125	-1.17966764964269\\
0.116375	-1.21133756527062\\
0.117625	-1.2409930703374\\
0.118875	-1.26858369645469\\
0.120125	-1.29406246850824\\
0.121375	-1.31738599156258\\
0.122625	-1.33851450805476\\
0.123875	-1.35741197710558\\
0.125125	-1.37404612487845\\
0.126375	-1.38838851301986\\
0.127625	-1.40041457146759\\
0.128875	-1.41010364426932\\
0.130125	-1.41743902594997\\
0.131375	-1.42240798959357\\
0.132625	-1.42500180582554\\
0.133875	-1.42521575638959\\
0.135125	-1.42304914290018\\
0.136375	-1.41850529117183\\
0.137625	-1.4115915408974\\
0.138875	-1.40231923224691\\
0.140125	-1.39070368905205\\
0.141375	-1.37676419462515\\
0.142625	-1.36052395487509\\
0.143875	-1.34201006080437\\
0.145125	-1.32125344042051\\
0.146375	-1.29828881220773\\
0.147625	-1.27315461998505\\
0.148875	-1.24589297260032\\
0.150125	-1.21654956577172\\
0.151375	-1.18517361618283\\
0.152625	-1.15181776850208\\
0.153875	-1.11653801149677\\
0.155125	-1.07939357742858\\
0.156375	-1.04044685278066\\
0.157625	-0.999763263789527\\
0.158875	-0.957411169666783\\
0.160125	-0.913461743613379\\
0.161375	-0.867988858845843\\
0.162625	-0.8210689613592\\
0.163875	-0.772780939206534\\
0.165125	-0.723205987278482\\
0.166375	-0.67242747817049\\
0.167625	-0.620530816342417\\
0.168875	-0.567603294143893\\
0.170125	-0.513733942029126\\
0.171375	-0.459013384978475\\
0.172625	-0.403533687141159\\
0.173875	-0.347388192994871\\
0.175125	-0.290671373875953\\
0.176375	-0.233478664269642\\
0.177625	-0.17590630994943\\
0.178875	-0.118051194455844\\
0.180125	-0.0600106833130209\\
0.181375	-0.00188245528349258\\
0.182625	0.0562356566179212\\
0.183875	0.11424583671348\\
0.185125	0.172050447271089\\
0.186375	0.2295521955002\\
0.187625	0.286654294286148\\
0.188875	0.343260628749699\\
0.190125	0.399275918733022\\
0.191375	0.454605876529843\\
0.192625	0.509157365924736\\
0.193875	0.562838562096749\\
0.195125	0.61555910490171\\
0.196375	0.66723025144179\\
0.197625	0.717765022907089\\
0.198875	0.767078356690719\\
0.200125	0.815087246021342\\
0.201375	0.861710880777081\\
0.202625	0.906870782064549\\
0.203875	0.950490938405668\\
0.205125	0.992497930643141\\
0.206375	1.03282105747251\\
0.207625	1.07139245275897\\
0.208875	1.10814720110058\\
0.210125	1.14302344880681\\
0.211375	1.17596250651178\\
0.212625	1.20690894613471\\
0.213875	1.23581069811082\\
0.215125	1.26261913844426\\
0.216375	1.28728916926836\\
0.217625	1.30977929325382\\
0.218875	1.33005168799124\\
0.220125	1.34807226738215\\
0.221375	1.36381074038172\\
0.222625	1.37724065848289\\
0.223875	1.38833946297036\\
0.225125	1.39708852500473\\
0.226375	1.40347317488829\\
0.227625	1.40748272453494\\
0.228875	1.40911048730132\\
0.230125	1.40835379320713\\
0.231375	1.40521398952584\\
0.232625	1.39969643590092\\
0.233875	1.39181049882448\\
0.235125	1.38156953953973\\
0.236375	1.36899088418463\\
0.237625	1.35409579925632\\
0.238875	1.33690945161744\\
0.240125	1.3174608726338\\
0.241375	1.29578290480582\\
0.242625	1.27191214684834\\
0.243875	1.2458888892508\\
0.245125	1.21775705456897\\
0.246375	1.18756411641399\\
0.247625	1.15536102183973\\
0.248875	1.12120210102848\\
0.250125	1.08514498478225\\
0.251375	1.04725050254655\\
0.252625	1.00758257573539\\
0.253875	0.966208116560594\\
0.255125	0.923196911295518\\
0.256375	0.878621508019471\\
0.257625	0.832557085099473\\
0.258875	0.785081332873148\\
0.260125	0.736274317418442\\
0.261375	0.686218352459554\\
0.262625	0.634997851222491\\
0.263875	0.582699191204876\\
0.265125	0.529410564691223\\
0.266375	0.475221835025845\\
0.267625	0.420224380748575\\
0.268875	0.364510939703959\\
0.270125	0.308175458836614\\
0.271375	0.251312933261124\\
0.272625	0.194019245129549\\
0.273875	0.136391001216841\\
0.275125	0.078525372465304\\
0.276375	0.0205199308446419\\
0.277625	-0.0375275175889987\\
0.278875	-0.0955190984276487\\
0.280125	-0.153357030178436\\
0.281375	-0.210943790289082\\
0.282625	-0.26818227993202\\
0.283875	-0.324975991874613\\
0.285125	-0.381229163580853\\
0.286375	-0.436846948561471\\
0.287625	-0.491735566170451\\
0.288875	-0.545802472838093\\
0.290125	-0.598956504008522\\
0.291375	-0.651108038935763\\
0.292625	-0.70216914108798\\
0.293875	-0.752053720062955\\
0.295125	-0.800677661863204\\
0.296375	-0.847958983339128\\
0.297625	-0.893817959417014\\
0.298875	-0.938177266364103\\
0.300125	-0.980962108547552\\
0.301375	-1.02210034267597\\
0.302625	-1.0615226021292\\
0.303875	-1.09916241297897\\
0.305125	-1.13495630612423\\
0.306375	-1.16884392332672\\
0.307625	-1.20076811999143\\
0.308875	-1.230675061775\\
0.310125	-1.25851431493767\\
0.311375	-1.28423892944709\\
0.312625	-1.30780552127368\\
0.313875	-1.32917434251796\\
0.315125	-1.34830935519907\\
0.316375	-1.36517827702658\\
0.317625	-1.37975265548609\\
0.318875	-1.39200789705918\\
0.320125	-1.40192332713526\\
0.321375	-1.40948220342138\\
0.322625	-1.41467176740232\\
0.323875	-1.41748324408344\\
0.325125	-1.41791187774091\\
0.326375	-1.41595691655714\\
0.327625	-1.4116216390903\\
0.328875	-1.40491332618303\\
0.330125	-1.39584326517759\\
0.331375	-1.38442672047914\\
0.332625	-1.37068291280751\\
0.333875	-1.3546349871628\\
0.335125	-1.3363099718015\\
0.336375	-1.31573873139149\\
0.337625	-1.29295591755948\\
0.338875	-1.26799991147064\\
0.340125	-1.24091275664167\\
0.341375	-1.21174008738128\\
0.342625	-1.18053105390906\\
0.343875	-1.14733824028578\\
0.345125	-1.11221757197903\\
0.346375	-1.07522823064747\\
0.347625	-1.03643253506318\\
0.348875	-0.995895863914584\\
0.350125	-0.953686516263371\\
0.351375	-0.909875624877073\\
0.352625	-0.864537006561262\\
0.353875	-0.817747069316583\\
0.355125	-0.769584652256645\\
0.356375	-0.720130923399465\\
0.357625	-0.669469211077928\\
0.358875	-0.617684896449239\\
0.360125	-0.564865238711857\\
0.361375	-0.511099251215537\\
0.362625	-0.45647753892323\\
0.363875	-0.40109215066714\\
0.365125	-0.345036422111958\\
0.366375	-0.288404819322002\\
0.367625	-0.231292778857678\\
0.368875	-0.173796547019301\\
0.370125	-0.116013018675455\\
0.371375	-0.0580395727553072\\
0.372625	2.60918872620081e-05\\
0.373875	0.0580861202264842\\
0.375125	0.116042666674748\\
0.376375	0.173798061957151\\
0.377625	0.231254967134799\\
0.378875	0.288316561516556\\
0.380125	0.344886672080271\\
0.381375	0.400869972708268\\
0.382625	0.45617210758958\\
0.383875	0.510699887718149\\
0.385125	0.564361409231163\\
0.386375	0.617066248011472\\
0.387625	0.66872557263216\\
0.388875	0.719252333729051\\
0.390125	0.76856136963516\\
0.391375	0.816569589277584\\
0.392625	0.863196083744992\\
0.393875	0.908362275414107\\
0.395125	0.951992047050131\\
0.396375	0.994011871411068\\
0.397625	1.03435093457715\\
0.398875	1.07294125533575\\
0.400125	1.10971779909437\\
0.401375	1.14461858864411\\
0.402625	1.17758480850579\\
0.403875	1.20856090293678\\
0.405125	1.23749467025637\\
0.406375	1.26433735126807\\
0.407625	1.28904371369391\\
0.408875	1.31157211535127\\
0.410125	1.33188460253031\\
0.411375	1.3499469353282\\
0.412625	1.36572868787885\\
0.413875	1.37920325663555\\
0.415125	1.39034794719519\\
0.416375	1.39914397060057\\
0.417625	1.40557651828975\\
0.418875	1.40963474411602\\
0.420125	1.41131182349827\\
0.421375	1.41060492482328\\
0.422625	1.40751525221788\\
0.423875	1.4020480165775\\
0.425125	1.39421243680914\\
0.426375	1.38402172392396\\
0.427625	1.37149305855458\\
0.428875	1.35664756214874\\
0.430125	1.33951026043353\\
0.431375	1.32011004088257\\
0.432625	1.29847960617449\\
0.433875	1.27465541759821\\
0.435125	1.24867763324949\\
0.436375	1.22059004036035\\
0.437625	1.19043998347304\\
0.438875	1.1582782846294\\
0.440125	1.1241591465957\\
0.441375	1.08814008868125\\
0.442625	1.05028181188701\\
0.443875	1.01064813781514\\
0.445125	0.969305859409574\\
0.446375	0.926324670006192\\
0.447625	0.881777003872306\\
0.448875	0.835737956952475\\
0.450125	0.788285117533997\\
0.451375	0.739498477686866\\
0.452625	0.689460256191344\\
0.453875	0.638254799317115\\
0.455125	0.585968410868691\\
0.456375	0.532689218130992\\
0.457625	0.478507022012835\\
0.458875	0.42351314493779\\
0.460125	0.367800277322981\\
0.461375	0.311462322133969\\
0.462625	0.254594234341661\\
0.463875	0.197291863265509\\
0.465125	0.139651789445586\\
0.466375	0.081771163180946\\
0.467625	0.0237475393704208\\
0.468875	-0.0343212858601352\\
0.470125	-0.0923374378833005\\
0.471375	-0.150203141781169\\
0.472625	-0.207820856550747\\
0.473875	-0.265093482295948\\
0.475125	-0.321924478408758\\
0.476375	-0.37821807103178\\
0.477625	-0.433879369162024\\
0.478875	-0.488814571652511\\
0.480125	-0.542931077396812\\
0.481375	-0.596137688738909\\
0.482625	-0.648344717764532\\
0.483875	-0.699464186220391\\
0.485125	-0.749409927352589\\
0.486375	-0.798097765420111\\
0.487625	-0.845445641068886\\
0.488875	-0.891373754117116\\
0.490125	-0.935804697604472\\
0.491375	-0.978663588436832\\
0.492625	-1.01987819201093\\
0.493875	-1.05937904613498\\
0.495125	-1.09709957734091\\
0.496375	-1.13297621275462\\
0.497625	-1.16694848742829\\
0.498875	-1.19895914601411\\
0.500125	-1.22895423652555\\
0.501375	-1.25688321468522\\
0.502625	-1.28269900784093\\
0.503875	-1.30635811605187\\
0.505125	-1.32782064598591\\
0.506375	-1.34705044952514\\
0.507625	-1.36401511349357\\
0.508875	-1.3786860565543\\
0.510125	-1.39103853462553\\
0.511375	-1.40105175542377\\
0.512625	-1.40870884040829\\
0.513875	-1.41399689648274\\
0.515125	-1.41690699375135\\
0.516375	-1.41743425645648\\
0.517625	-1.41557779398958\\
0.518875	-1.41134074924859\\
0.520125	-1.40473025418652\\
0.521375	-1.39575745679868\\
0.522625	-1.38443750300788\\
0.523875	-1.37078947897308\\
0.525125	-1.35483637812594\\
0.526375	-1.33660509599488\\
0.527625	-1.31612638414889\\
0.528875	-1.29343476639294\\
0.530125	-1.26856847843594\\
0.531375	-1.24156943969148\\
0.532625	-1.21248318024863\\
0.533875	-1.18135872348966\\
0.535125	-1.14824854430835\\
0.536375	-1.11320843206301\\
0.537625	-1.07629747870053\\
0.538875	-1.03757789681256\\
0.540125	-0.997114963160899\\
0.541375	-0.954976859709037\\
0.542625	-0.911234643258183\\
0.543875	-0.865962040582616\\
0.545125	-0.819235374211562\\
0.546375	-0.771133383256258\\
0.547625	-0.721737177560933\\
0.548875	-0.671130013991184\\
0.550125	-0.619397207124314\\
0.551375	-0.566625942819353\\
0.552625	-0.512905174912766\\
0.553875	-0.458325476511\\
0.555125	-0.402978850025044\\
0.556375	-0.346958572006639\\
0.557625	-0.29035907347521\\
0.558875	-0.233275779921217\\
0.560125	-0.175804914613633\\
0.561375	-0.118043334717685\\
0.562625	-0.0600884068433356\\
0.563875	-0.00203784171154803\\
0.565125	0.0560105155520334\\
0.566375	0.113958805843085\\
0.567625	0.171709375527071\\
0.568875	0.229164848043623\\
0.570125	0.286228380476831\\
0.571375	0.342803772917828\\
0.572625	0.398795685345925\\
0.573875	0.454109704116298\\
0.575125	0.50865259538411\\
0.576375	0.562332406359896\\
0.577625	0.615058677438739\\
0.578875	0.666742497911455\\
0.580125	0.717296751882481\\
0.581375	0.766636208268381\\
0.582625	0.814677712829791\\
0.583875	0.861340279240812\\
0.585125	0.906545224598978\\
0.586375	0.950216342805025\\
0.587625	0.992280032747845\\
0.588875	1.03266538259961\\
0.590125	1.07130428890005\\
0.591375	1.1081316112926\\
0.592625	1.14308528300512\\
0.593875	1.17610637501391\\
0.595125	1.20713919510569\\
0.596375	1.23613143199935\\
0.597625	1.26303419349448\\
0.598875	1.28780214973238\\
0.600125	1.31039350788752\\
0.601375	1.33077018338091\\
0.602625	1.34889780431708\\
0.603875	1.36474583045894\\
0.605125	1.37828750194755\\
0.606375	1.38949998646534\\
0.607625	1.3983643570597\\
0.608875	1.40486568545118\\
0.610125	1.40899296424927\\
0.611375	1.41073922799345\\
0.612625	1.41010150371776\\
0.613875	1.40708086727102\\
0.615125	1.40168239022781\\
0.616375	1.39391513047454\\
0.617625	1.38379215908969\\
0.618875	1.37133053791933\\
0.620125	1.356551250399\\
0.621375	1.33947916362223\\
0.622625	1.32014303143099\\
0.623875	1.29857544265056\\
0.625125	1.27481272795818\\
0.626375	1.2488948938133\\
0.627625	1.22086561358637\\
0.628875	1.19077209448566\\
0.630125	1.15866506884413\\
0.631375	1.12459859537177\\
0.632625	1.08863007987565\\
0.633875	1.05082011052553\\
0.635125	1.01123242434923\\
0.636375	0.969933686824254\\
0.637625	0.92699349253253\\
0.638875	0.882484179089607\\
0.640125	0.836480773997593\\
0.641375	0.789060755122592\\
0.642625	0.740304032718918\\
0.643875	0.690292747101305\\
0.645125	0.639111186364098\\
0.646375	0.586845589109408\\
0.647625	0.533583996773978\\
0.648875	0.479416153314664\\
0.650125	0.424433351132511\\
0.651375	0.368728234590562\\
0.652625	0.312394640515651\\
0.653875	0.255527489668228\\
0.655125	0.198222622883901\\
0.656375	0.140576597202109\\
0.657625	0.0826865189473161\\
0.658875	0.0246499426143104\\
0.660125	-0.0334353556054291\\
0.661375	-0.0914714702838699\\
0.662625	-0.149360650208226\\
0.663875	-0.207005344215642\\
0.665125	-0.264308437446923\\
0.666375	-0.321173344297945\\
0.667625	-0.377504289503104\\
0.668875	-0.433206349288857\\
0.670125	-0.488185685012287\\
0.671375	-0.542349628344635\\
0.672625	-0.595606958121763\\
0.673875	-0.647867932808043\\
0.675125	-0.699044514828278\\
0.676375	-0.749050458767984\\
0.677625	-0.797801517022656\\
0.678875	-0.84521558128272\\
0.680125	-0.891212772719691\\
0.681375	-0.93571557699173\\
0.682625	-0.978649022919679\\
0.683875	-1.01994081018828\\
0.685125	-1.05952138015401\\
0.686375	-1.09732403473182\\
0.687625	-1.13328509758744\\
0.688875	-1.16734402171691\\
0.690125	-1.19944342901003\\
0.691375	-1.22952926919082\\
0.692625	-1.25755083668179\\
0.693875	-1.28346098092857\\
0.695125	-1.30721605933851\\
0.696375	-1.32877608864411\\
0.697625	-1.34810473398326\\
0.698875	-1.36516949896669\\
0.700125	-1.37994165151266\\
0.701375	-1.39239635083232\\
0.702625	-1.40251260997691\\
0.703875	-1.41027346228976\\
0.705125	-1.41566585913933\\
0.706375	-1.41868077094117\\
0.707625	-1.41931313621437\\
0.708875	-1.4175619355935\\
0.710125	-1.41343019486548\\
0.711375	-1.40692492616422\\
0.712625	-1.39805711557283\\
0.713875	-1.38684175800872\\
0.715125	-1.37329783343876\\
0.716375	-1.35744822048847\\
0.717625	-1.33931965648845\\
0.718875	-1.31894274762331\\
0.720125	-1.29635191809537\\
0.721375	-1.27158528288117\\
0.722625	-1.24468465319362\\
0.723875	-1.2156953806709\\
0.725125	-1.18466642193242\\
0.726375	-1.15165011687982\\
0.727625	-1.11670218381147\\
0.728875	-1.07988154010675\\
0.730125	-1.0412503453016\\
0.731375	-1.00087375483652\\
0.732625	-0.958819896058489\\
0.733875	-0.915159664904856\\
0.735125	-0.869966752650668\\
0.736375	-0.823317376708083\\
0.737625	-0.775290239987853\\
0.738875	-0.725966324281949\\
0.740125	-0.675428825506185\\
0.741375	-0.623763016511641\\
0.742625	-0.571056044429683\\
0.743875	-0.517396781003437\\
0.745125	-0.462875732370044\\
0.746375	-0.407584890108544\\
0.747625	-0.351617513722334\\
0.748875	-0.295067974010409\\
0.750125	-0.238031652047877\\
0.751375	-0.180604782559809\\
0.752625	-0.122884213810799\\
0.753875	-0.0649673206883618\\
0.755125	-0.00695174706198896\\
0.756375	0.0510646027606271\\
0.757625	0.108983883021549\\
0.758875	0.166708379494403\\
0.760125	0.224140769336216\\
0.761375	0.281184126335546\\
0.762625	0.337742240373915\\
0.763875	0.393719687269567\\
0.765125	0.449022084723123\\
0.766375	0.503556091496197\\
0.767625	0.557229723101609\\
0.768875	0.609952411816403\\
0.770125	0.661635240589798\\
0.771375	0.712191012053818\\
0.772625	0.761534394010518\\
0.773875	0.809582127883513\\
0.775125	0.85625317137686\\
0.776375	0.901468769367687\\
0.777625	0.945152584419494\\
0.778875	0.987230892267186\\
0.780125	1.02763270613491\\
0.781375	1.06628983177379\\
0.782625	1.10313698009428\\
0.783875	1.13811196038198\\
0.785125	1.17115570312247\\
0.786375	1.20221245943401\\
0.787625	1.23122972771696\\
0.788875	1.25815850768891\\
0.790125	1.28295328535668\\
0.791375	1.30557220835913\\
0.792625	1.32597699015831\\
0.793875	1.34413313975049\\
0.795125	1.36000992220003\\
0.796375	1.37358050938263\\
0.797625	1.38482185781653\\
0.798875	1.39371491408369\\
0.800125	1.40024454827909\\
0.801375	1.40439966186255\\
0.802625	1.40617312382195\\
0.803875	1.40556177981773\\
0.805125	1.4025665276678\\
0.806375	1.39719231539848\\
0.807625	1.38944806351267\\
0.808875	1.37934665226069\\
0.810125	1.36690496465438\\
0.811375	1.35214386159435\\
0.812625	1.33508807602195\\
0.813875	1.31576617237229\\
0.815125	1.29421058212179\\
0.816375	1.27045746617436\\
0.817625	1.24454675479234\\
0.818875	1.21652190910764\\
0.820125	1.18643001874832\\
0.821375	1.15432162184364\\
0.822625	1.1202507196144\\
0.823875	1.08427451576384\\
0.825125	1.04645349005977\\
0.826375	1.00685119617874\\
0.827625	0.965534253526843\\
0.828875	0.922572066877838\\
0.830125	0.87803687802275\\
0.831375	0.832003543467249\\
0.832625	0.78454949193467\\
0.833875	0.735754509947787\\
0.835125	0.685700604812476\\
0.836375	0.634471936843299\\
0.837625	0.582154678205628\\
0.838875	0.528836794844833\\
0.840125	0.474607904303475\\
0.841375	0.419559186049361\\
0.842625	0.363783235449513\\
0.843875	0.307373832160046\\
0.845125	0.250425785541202\\
0.846375	0.193034858784773\\
0.847625	0.135297524979038\\
0.848875	0.0773109033951079\\
0.850125	0.0191724251654064\\
0.851375	-0.039020159210627\\
0.852625	-0.0971690458449603\\
0.853875	-0.155176441533624\\
0.855125	-0.212944902506091\\
0.856375	-0.270377322400555\\
0.857625	-0.327377201890678\\
0.858875	-0.38384870656881\\
0.860125	-0.439697004705081\\
0.861375	-0.494828250970242\\
0.862625	-0.549149850015804\\
0.863875	-0.602570524218084\\
0.865125	-0.655000557127479\\
0.866375	-0.706351945798742\\
0.867625	-0.756538476307165\\
0.868875	-0.805475868211879\\
0.870125	-0.853081991967875\\
0.871375	-0.89927700593637\\
0.872625	-0.943983418674407\\
0.873875	-0.987126218939966\\
0.875125	-1.02863307715931\\
0.876375	-1.0684344672215\\
0.877625	-1.10646369091218\\
0.878875	-1.14265708414934\\
0.880125	-1.17695401189865\\
0.881375	-1.2092971593951\\
0.882625	-1.23963243982371\\
0.883875	-1.26790919896727\\
0.885125	-1.29408018849954\\
0.886375	-1.31810183546631\\
0.887625	-1.33993412495935\\
0.888875	-1.35954078341894\\
0.890125	-1.3768892247082\\
0.891375	-1.39195079881674\\
0.892625	-1.40470064640807\\
0.893875	-1.4151178580432\\
0.895125	-1.42318541324733\\
0.896375	-1.42889030611022\\
0.897625	-1.43222356981665\\
0.898875	-1.43318021024321\\
0.900125	-1.43175921980155\\
0.901375	-1.42796365518416\\
0.902625	-1.42180063490525\\
0.903875	-1.41328124665514\\
0.905125	-1.40242053270893\\
0.906375	-1.38923754139505\\
0.907625	-1.37375529992559\\
0.908875	-1.35600067160481\\
0.910125	-1.3360044172613\\
0.911375	-1.31380101853088\\
0.912625	-1.28942882848131\\
0.913875	-1.262929798823\\
0.915125	-1.23434953870992\\
0.916375	-1.20373711293543\\
0.917625	-1.17114516786957\\
0.918875	-1.13662963637684\\
0.920125	-1.10024977023372\\
0.921375	-1.06206791788197\\
0.922625	-1.02214962920177\\
0.923875	-0.980563337544152\\
0.925125	-0.93738037145019\\
0.926375	-0.892674734095152\\
0.927625	-0.846523084866784\\
0.928875	-0.799004612333517\\
0.930125	-0.750200817730606\\
0.931375	-0.700195382758813\\
0.932625	-0.64907411685889\\
0.933875	-0.596924813987396\\
0.935125	-0.543837023643159\\
0.936375	-0.489901903240509\\
0.937625	-0.435212156101898\\
0.938875	-0.3798618793278\\
0.940125	-0.323946300102464\\
0.941375	-0.26756173022022\\
0.942625	-0.210805279794717\\
0.943875	-0.153774917623639\\
0.945125	-0.0965690911947624\\
0.946375	-0.0392867040267481\\
0.947625	0.0179731785834852\\
0.948875	0.0751112633405706\\
0.950125	0.132028412420785\\
0.951375	0.188625663876985\\
0.952625	0.24480452256278\\
0.953875	0.300466886345558\\
0.955125	0.355515426319565\\
0.956375	0.409853597370374\\
0.957625	0.463385895903723\\
0.958875	0.516017890278033\\
0.960125	0.567656358778295\\
0.961375	0.618209510663338\\
0.962625	0.667587116678419\\
0.963875	0.71570053756277\\
0.965125	0.76246284110828\\
0.966375	0.807788997954782\\
0.967625	0.851595981959996\\
0.968875	0.89380276669074\\
0.970125	0.934330404178962\\
0.971375	0.973102201383652\\
0.972625	1.01004365930307\\
0.973875	1.04508263807896\\
0.975125	1.07814914492981\\
0.976375	1.10917554692625\\
0.977625	1.13809639354386\\
0.978875	1.16484847668183\\
0.980125	1.18937047827223\\
0.981375	1.21160302246627\\
0.982625	1.23148828519484\\
0.983875	1.2489697886335\\
0.985125	1.26399169872191\\
0.986375	1.27649840780783\\
0.987625	1.28643350521756\\
0.988875	1.29373864716418\\
0.990125	1.2983516612183\\
0.991375	1.30020392151042\\
0.992625	1.29921643668673\\
0.993875	1.29529340400547\\
0.995125	1.28831100552258\\
0.996375	1.2780962725422\\
0.997625	1.26438052343329\\
0.998875	1.24666515333934\\
};
\addplot [color=black,only marks,mark=o,mark options={solid},forget plot]
  table[row sep=crcr]{%
0.000125	0\\
};
\addplot [color=black,only marks,mark=o,mark options={solid},forget plot]
  table[row sep=crcr]{%
1	-1.22474487139159\\
};
\addplot [color=black,only marks,mark=o,mark options={solid},forget plot]
  table[row sep=crcr]{%
1	1.22474487139159\\
};
\node[right, align=left, text=black]
at (axis cs:0.85,1.6) {$\sqrt{2H+1}$};
\node[right, align=left, text=black]
at (axis cs:0.85,-1.6) {$-\sqrt{2H+1}$};
\end{axis}
\end{tikzpicture}%

%% file: phi75gridon.tex
%
%
\begin{tikzpicture}[scale = 0.45]

\begin{axis}[%
width=6.028in,
height=4.754in,
at={(1.011in,0.642in)},
scale only axis,
xmin=-0.05,
xmax=1.05,
xmajorgrids,
ymin=-2,
ymax=2,
ymajorgrids,
axis background/.style={fill=white}
]
\addplot [color=blue,solid,forget plot]
  table[row sep=crcr]{%
0.000125	0.00180468324874524\\
0.001375	0.0426376742469485\\
0.002625	0.0874903674954884\\
0.003875	0.134563623517958\\
0.005125	0.183136247370503\\
0.006375	0.23275839822151\\
0.007625	0.283102254489202\\
0.008875	0.33390701253437\\
0.010125	0.384953765826162\\
0.011375	0.436051535894358\\
0.012625	0.48702956720614\\
0.013875	0.537732091130229\\
0.015125	0.588015189311737\\
0.016375	0.637744416551653\\
0.017625	0.686793254790382\\
0.018875	0.735042067683285\\
0.020125	0.78237719081548\\
0.021375	0.828690384547816\\
0.022625	0.87387840459992\\
0.023875	0.917842672624285\\
0.025125	0.960488990738019\\
0.026375	1.00172745431494\\
0.027625	1.04147233570186\\
0.028875	1.0796419519904\\
0.030125	1.11615855680381\\
0.031375	1.15094845759578\\
0.032625	1.18394197855854\\
0.033875	1.21507330064695\\
0.035125	1.24428067363652\\
0.036375	1.27150625737737\\
0.037625	1.29669638629225\\
0.038875	1.31980125574201\\
0.040125	1.34077535409422\\
0.041375	1.35957687990994\\
0.042625	1.37616865875574\\
0.043875	1.39051725036635\\
0.045125	1.40259366322431\\
0.046375	1.41237280217536\\
0.047625	1.41983403162726\\
0.048875	1.4249607228453\\
0.050125	1.42774069439748\\
0.051375	1.42816566844571\\
0.052625	1.42623168374602\\
0.053875	1.42193917250539\\
0.055125	1.4152925880606\\
0.056375	1.40630038193631\\
0.057625	1.39497526630586\\
0.058875	1.38133428060329\\
0.060125	1.36539814340045\\
0.061375	1.34719181066449\\
0.062625	1.32674399482828\\
0.063875	1.30408767667382\\
0.065125	1.27925928365951\\
0.066375	1.25229925636521\\
0.067625	1.22325151176508\\
0.068875	1.19216394356108\\
0.070125	1.15908759680579\\
0.071375	1.12407719249093\\
0.072625	1.08719059592932\\
0.073875	1.04848913949758\\
0.075125	1.0080370863134\\
0.076375	0.965901977174473\\
0.077625	0.922153836202543\\
0.078875	0.876865865602366\\
0.080125	0.830113674564736\\
0.081375	0.78197564772254\\
0.082625	0.732532170572798\\
0.083875	0.6818662355509\\
0.085125	0.630062715958748\\
0.086375	0.577208651550296\\
0.087625	0.523392692609258\\
0.088875	0.468705142227844\\
0.090125	0.413238071198561\\
0.091375	0.357084579073247\\
0.092625	0.300339166238776\\
0.093875	0.243097171148852\\
0.095125	0.185455078207991\\
0.096375	0.127509772445375\\
0.097625	0.0693588650682351\\
0.098875	0.0111001853709961\\
0.100125	-0.0471680050707435\\
0.101375	-0.105347666423837\\
0.102625	-0.163340726201333\\
0.103875	-0.221049502148394\\
0.105125	-0.278376699570603\\
0.106375	-0.335225731752223\\
0.107625	-0.391500690917509\\
0.108875	-0.447106788193681\\
0.110125	-0.501950218961058\\
0.111375	-0.555938533448743\\
0.112625	-0.608980633440677\\
0.113875	-0.660987124116166\\
0.115125	-0.711870258607046\\
0.116375	-0.761544230888128\\
0.117625	-0.809925220287287\\
0.118875	-0.856931639168982\\
0.120125	-0.902484187942806\\
0.121375	-0.946506009119158\\
0.122625	-0.988922849762642\\
0.123875	-1.02966315392337\\
0.125125	-1.0686581670227\\
0.126375	-1.10584212870152\\
0.127625	-1.14115230513782\\
0.128875	-1.17452912447263\\
0.130125	-1.20591627213951\\
0.131375	-1.23526082033607\\
0.132625	-1.26251327044805\\
0.133875	-1.28762766189138\\
0.135125	-1.31056162126959\\
0.136375	-1.3312765082654\\
0.137625	-1.34973739185049\\
0.138875	-1.36591315438228\\
0.140125	-1.37977653386612\\
0.141375	-1.39130420918575\\
0.142625	-1.40047675663814\\
0.143875	-1.40727877328192\\
0.145125	-1.4116988165944\\
0.146375	-1.41372952719416\\
0.147625	-1.41336753107379\\
0.148875	-1.4106135328386\\
0.150125	-1.40547224538627\\
0.151375	-1.39795245476024\\
0.152625	-1.38806693461814\\
0.153875	-1.37583248418611\\
0.155125	-1.36126983742116\\
0.156375	-1.34440370035531\\
0.157625	-1.32526265233411\\
0.158875	-1.30387912553753\\
0.160125	-1.28028932730379\\
0.161375	-1.25453321050156\\
0.162625	-1.22665438597609\\
0.163875	-1.1967000509694\\
0.165125	-1.16472090830948\\
0.166375	-1.13077108508359\\
0.167625	-1.09490804350935\\
0.168875	-1.0571924749076\\
0.170125	-1.01768820370094\\
0.171375	-0.976462076117593\\
0.172625	-0.933583852848895\\
0.173875	-0.889126082685263\\
0.175125	-0.843163985973419\\
0.176375	-0.795775328470383\\
0.177625	-0.747040286101669\\
0.178875	-0.697041305967202\\
0.180125	-0.645862976313659\\
0.181375	-0.593591879148967\\
0.182625	-0.540316439592926\\
0.183875	-0.486126785685969\\
0.185125	-0.431114584261101\\
0.186375	-0.375372895544443\\
0.187625	-0.318996012771879\\
0.188875	-0.262079309726663\\
0.190125	-0.204719053766524\\
0.191375	-0.147012273793658\\
0.192625	-0.08905659071675\\
0.193875	-0.0309500320607111\\
0.195125	0.027209117520271\\
0.196375	0.0853224799090877\\
0.197625	0.143291754358328\\
0.198875	0.201018891888933\\
0.200125	0.258406249468095\\
0.201375	0.315356761045925\\
0.202625	0.371774090565171\\
0.203875	0.427562820606511\\
0.205125	0.482628584710832\\
0.206375	0.53687825156444\\
0.207625	0.590220052919876\\
0.208875	0.642563770136662\\
0.210125	0.693820869169351\\
0.211375	0.743904659068058\\
0.212625	0.79273042026575\\
0.213875	0.8402155775641\\
0.215125	0.886279812453773\\
0.216375	0.930845215810733\\
0.217625	0.973836405270921\\
0.218875	1.01518067416461\\
0.220125	1.05480809003048\\
0.221375	1.09265163202099\\
0.222625	1.1286472896871\\
0.223875	1.16273417871197\\
0.225125	1.19485465827106\\
0.226375	1.22495439321801\\
0.227625	1.25298247226233\\
0.228875	1.27889149035619\\
0.230125	1.30263763384343\\
0.231375	1.32418073680841\\
0.232625	1.34348435619591\\
0.233875	1.36051584120244\\
0.235125	1.37524640011653\\
0.236375	1.3876511060854\\
0.237625	1.39770898288751\\
0.238875	1.40540301127246\\
0.240125	1.41072018756677\\
0.241375	1.41365151044013\\
0.242625	1.41419202506013\\
0.243875	1.41234081192733\\
0.245125	1.40810100685599\\
0.246375	1.40147977587881\\
0.247625	1.39248831929658\\
0.248875	1.38114183583425\\
0.250125	1.36745951977072\\
0.251375	1.35146451519273\\
0.252625	1.33318386590939\\
0.253875	1.31264848906556\\
0.255125	1.28989311466958\\
0.256375	1.2649562384175\\
0.257625	1.23788002049745\\
0.258875	1.20871025876663\\
0.260125	1.17749628083086\\
0.261375	1.14429088808151\\
0.262625	1.10915023024115\\
0.263875	1.07213374097493\\
0.265125	1.03330402387011\\
0.266375	0.992726749577423\\
0.267625	0.950470553125332\\
0.268875	0.906606886873778\\
0.270125	0.861209939182023\\
0.271375	0.81435649016944\\
0.272625	0.766125783558872\\
0.273875	0.716599383740133\\
0.275125	0.665861053073352\\
0.276375	0.613996604965431\\
0.277625	0.561093759848966\\
0.278875	0.507241983028178\\
0.280125	0.45253235108656\\
0.281375	0.397057395921071\\
0.282625	0.340910941947056\\
0.283875	0.284187939141469\\
0.285125	0.226984330201374\\
0.286375	0.169396850282319\\
0.287625	0.111522910428864\\
0.288875	0.0534603735102692\\
0.290125	-0.00469254495080058\\
0.291375	-0.0628375119680143\\
0.292625	-0.120876170961998\\
0.293875	-0.178710379157391\\
0.295125	-0.236242304037623\\
0.296375	-0.293374667254077\\
0.297625	-0.350010823743314\\
0.298875	-0.406054992064078\\
0.300125	-0.461412387405591\\
0.301375	-0.51598938602283\\
0.302625	-0.569693683418953\\
0.303875	-0.622434448472431\\
0.305125	-0.67412248214758\\
0.306375	-0.724670366780407\\
0.307625	-0.773992610741419\\
0.308875	-0.8220057948525\\
0.310125	-0.868628716245206\\
0.311375	-0.913782521402626\\
0.312625	-0.957390841993306\\
0.313875	-0.999379917353701\\
0.315125	-1.039678751503\\
0.316375	-1.07821915511522\\
0.317625	-1.11493597420128\\
0.318875	-1.149767083263\\
0.320125	-1.18265359907894\\
0.321375	-1.21353986492569\\
0.322625	-1.24237367558176\\
0.323875	-1.26910623541719\\
0.325125	-1.29369235704941\\
0.326375	-1.31609042107409\\
0.327625	-1.33626257817288\\
0.328875	-1.35417468506645\\
0.330125	-1.3697964542981\\
0.331375	-1.38310145728613\\
0.332625	-1.39406718445608\\
0.333875	-1.40267509496402\\
0.335125	-1.40891062424207\\
0.336375	-1.41276321749051\\
0.337625	-1.41422635353567\\
0.338875	-1.41329755860693\\
0.340125	-1.40997839823985\\
0.341375	-1.40427447797121\\
0.342625	-1.39619543579878\\
0.343875	-1.38575493758503\\
0.345125	-1.37297062349035\\
0.346375	-1.35786414100172\\
0.347625	-1.34046098529503\\
0.348875	-1.32079062975725\\
0.350125	-1.29888629760383\\
0.351375	-1.27478506804008\\
0.352625	-1.24852764939472\\
0.353875	-1.22015848859884\\
0.355125	-1.18972552042099\\
0.356375	-1.15728024809445\\
0.357625	-1.12287749164961\\
0.358875	-1.08657547281575\\
0.360125	-1.04843555030544\\
0.361375	-1.00852223470594\\
0.362625	-0.966903016153696\\
0.363875	-0.923648274177456\\
0.365125	-0.878831164465543\\
0.366375	-0.832527474897807\\
0.367625	-0.784815504441781\\
0.368875	-0.735775939801758\\
0.370125	-0.685491713387034\\
0.371375	-0.634047863635491\\
0.372625	-0.581531381537631\\
0.373875	-0.528031077280785\\
0.375125	-0.473637430562711\\
0.376375	-0.418442413419977\\
0.377625	-0.362539402692253\\
0.378875	-0.306022883738103\\
0.380125	-0.248988485381566\\
0.381375	-0.191532608161885\\
0.382625	-0.133752467443836\\
0.383875	-0.0757457216062558\\
0.385125	-0.0176105212417505\\
0.386375	0.0405548741077554\\
0.387625	0.0986520478147286\\
0.388875	0.156582804288744\\
0.390125	0.214249122844522\\
0.391375	0.271553524170675\\
0.392625	0.328399087605662\\
0.393875	0.384689680795251\\
0.395125	0.440330109728441\\
0.396375	0.495226281930266\\
0.397625	0.54928536219151\\
0.398875	0.602415932274248\\
0.400125	0.654528140745929\\
0.401375	0.705533863816211\\
0.402625	0.755346846601764\\
0.403875	0.803882849427134\\
0.405125	0.851059792751248\\
0.406375	0.896797898580723\\
0.407625	0.941019837308774\\
0.408875	0.983650777231524\\
0.410125	1.02461869076239\\
0.411375	1.06385423239611\\
0.412625	1.10129112460443\\
0.413875	1.13686598837356\\
0.415125	1.17051873440154\\
0.416375	1.20219238357256\\
0.417625	1.23183345049763\\
0.418875	1.25939173915039\\
0.420125	1.28482072011867\\
0.421375	1.30807732212838\\
0.422625	1.32912227489494\\
0.423875	1.34791997722405\\
0.425125	1.36443864129456\\
0.426375	1.37865033680792\\
0.427625	1.39053103855761\\
0.428875	1.40006065804998\\
0.430125	1.40722308576335\\
0.431375	1.41200621484382\\
0.432625	1.41440196394078\\
0.433875	1.41440628957047\\
0.435125	1.41201918641488\\
0.436375	1.40724470020995\\
0.437625	1.40009091308567\\
0.438875	1.39056995213259\\
0.440125	1.37869786049316\\
0.441375	1.36449480921509\\
0.442625	1.34798473910339\\
0.443875	1.32919566970976\\
0.445125	1.30815929125978\\
0.446375	1.28491127800004\\
0.447625	1.25949086268932\\
0.448875	1.23194113834972\\
0.450125	1.20230860956868\\
0.451375	1.17064349044777\\
0.452625	1.13699924850571\\
0.453875	1.10143286070516\\
0.455125	1.06400446323129\\
0.456375	1.02477735667138\\
0.457625	0.983817890506032\\
0.458875	0.941195346928461\\
0.460125	0.896981814935742\\
0.461375	0.851252074951647\\
0.462625	0.804083472032071\\
0.463875	0.755555789373695\\
0.465125	0.705751104726988\\
0.466375	0.654753651028054\\
0.467625	0.602649684383811\\
0.468875	0.549527333149269\\
0.470125	0.495476473047267\\
0.471375	0.440588439357676\\
0.472625	0.384956176514646\\
0.473875	0.328673668620688\\
0.475125	0.271836221740834\\
0.476375	0.214539852663271\\
0.477625	0.156881582680022\\
0.478875	0.0989588172439106\\
0.480125	0.0408696411123444\\
0.481375	-0.0172878153609646\\
0.482625	-0.0754150691482481\\
0.483875	-0.133413925760144\\
0.485125	-0.191186199430207\\
0.486375	-0.24863420538293\\
0.487625	-0.305660782167691\\
0.488875	-0.362169476239247\\
0.490125	-0.418064710725589\\
0.491375	-0.473251947661134\\
0.492625	-0.527637844527839\\
0.493875	-0.581130416742866\\
0.495125	-0.633639186544253\\
0.496375	-0.685075342698416\\
0.497625	-0.735351888104785\\
0.498875	-0.784383786152762\\
0.500125	-0.832088070812249\\
0.501375	-0.878384121847559\\
0.502625	-0.923193602010734\\
0.503875	-0.966440787468323\\
0.505125	-1.00805232435526\\
0.506375	-1.04795803505083\\
0.507625	-1.08609035406742\\
0.508875	-1.12238485156533\\
0.510125	-1.15677992658811\\
0.511375	-1.18921761104904\\
0.512625	-1.21964296691177\\
0.513875	-1.24800460120073\\
0.515125	-1.27425432250465\\
0.516375	-1.29834794296625\\
0.517625	-1.32024463525651\\
0.518875	-1.33990743123536\\
0.520125	-1.35730291833453\\
0.521375	-1.37240166748511\\
0.522625	-1.38517828262394\\
0.523875	-1.39561114786844\\
0.525125	-1.40368245653504\\
0.526375	-1.40937854677897\\
0.527625	-1.41268992874819\\
0.528875	-1.41361099394644\\
0.530125	-1.41214001881914\\
0.531375	-1.40827947799547\\
0.532625	-1.40203604823186\\
0.533875	-1.39342019984041\\
0.535125	-1.38244657259515\\
0.536375	-1.3691334731796\\
0.537625	-1.35350364767813\\
0.538875	-1.33558343858916\\
0.540125	-1.31540322370161\\
0.541375	-1.29299687202109\\
0.542625	-1.26840251520323\\
0.543875	-1.24166165388754\\
0.545125	-1.21281958476422\\
0.546375	-1.18192482010701\\
0.547625	-1.14902985191712\\
0.548875	-1.11419021704701\\
0.550125	-1.07746491095346\\
0.551375	-1.03891585878547\\
0.552625	-0.998608240890839\\
0.553875	-0.956610389937087\\
0.555125	-0.912993322405299\\
0.556375	-0.867830613937404\\
0.557625	-0.821198629850915\\
0.558875	-0.773176401405032\\
0.560125	-0.723845134753428\\
0.561375	-0.673288063064512\\
0.562625	-0.621590675064933\\
0.563875	-0.568840572946395\\
0.565125	-0.515126864189477\\
0.566375	-0.460540469562333\\
0.567625	-0.405173411010601\\
0.568875	-0.349119595008769\\
0.570125	-0.292473712587078\\
0.571375	-0.235331646286439\\
0.572625	-0.177789732561258\\
0.573875	-0.119945561904536\\
0.575125	-0.0618968511552261\\
0.576375	-0.0037418572913074\\
0.577625	0.0544213786223557\\
0.578875	0.112494207636972\\
0.580125	0.170378529464427\\
0.581375	0.227976363422394\\
0.582625	0.285190516096177\\
0.583875	0.341924243736232\\
0.585125	0.398081406551487\\
0.586375	0.453567045370414\\
0.587625	0.508287544701536\\
0.588875	0.562150377833299\\
0.590125	0.615064256651275\\
0.591375	0.666939706956935\\
0.592625	0.717689225124693\\
0.593875	0.767227000338648\\
0.595125	0.81546905465272\\
0.596375	0.862333924555688\\
0.597625	0.907742257803291\\
0.598875	0.95161760603025\\
0.600125	0.993885453379632\\
0.601375	1.0344744427765\\
0.602625	1.07331583378563\\
0.603875	1.11034429408255\\
0.605125	1.1454968783656\\
0.606375	1.17871426635514\\
0.607625	1.20994018269346\\
0.608875	1.2391221839297\\
0.610125	1.26621058621107\\
0.611375	1.29115971128503\\
0.612625	1.31392726521042\\
0.613875	1.33447499859301\\
0.615125	1.3527681810949\\
0.616375	1.36877565377952\\
0.617625	1.38247036208094\\
0.618875	1.39382940708832\\
0.620125	1.40283360010202\\
0.621375	1.40946748535739\\
0.622625	1.41371986311559\\
0.623875	1.41558380823534\\
0.625125	1.41505619302812\\
0.626375	1.41213767646018\\
0.627625	1.40683334186527\\
0.628875	1.39915205016846\\
0.630125	1.38910719890446\\
0.631375	1.376715413094\\
0.632625	1.36199779987663\\
0.633875	1.344979137276\\
0.635125	1.32568862510499\\
0.636375	1.30415851545675\\
0.637625	1.28042537400068\\
0.638875	1.25452922463666\\
0.640125	1.22651429072498\\
0.641375	1.19642757122753\\
0.642625	1.16432010649909\\
0.643875	1.13024608228907\\
0.645125	1.09426342213398\\
0.646375	1.05643300715036\\
0.647625	1.01681856373211\\
0.648875	0.975487110578928\\
0.650125	0.932508850976169\\
0.651375	0.887956497575838\\
0.652625	0.841905141891098\\
0.653875	0.794432690031871\\
0.655125	0.745619735422819\\
0.656375	0.695548861881342\\
0.657625	0.644304487504533\\
0.658875	0.591973449527965\\
0.660125	0.538644127646553\\
0.661375	0.484407180207304\\
0.662625	0.429353924009175\\
0.663875	0.373577639443444\\
0.665125	0.317172533073053\\
0.666375	0.260234473278312\\
0.667625	0.202859338432258\\
0.668875	0.145144339272071\\
0.670125	0.0871869611841631\\
0.671375	0.0290857054812889\\
0.672625	-0.0290615879954663\\
0.673875	-0.0871563978049652\\
0.675125	-0.145100598278932\\
0.676375	-0.202795855499022\\
0.677625	-0.260144558799924\\
0.678875	-0.317049997485073\\
0.680125	-0.373415903261581\\
0.681375	-0.429146604573816\\
0.682625	-0.484147811954009\\
0.683875	-0.538326787966934\\
0.685125	-0.591591874926698\\
0.686375	-0.643852635983103\\
0.687625	-0.695020650075006\\
0.688875	-0.74500966679921\\
0.690125	-0.793734948498942\\
0.691375	-0.84111421769098\\
0.692625	-0.887066819075369\\
0.693875	-0.931515488883804\\
0.695125	-0.974384854395485\\
0.696375	-1.01560254239886\\
0.697625	-1.05509830745314\\
0.698875	-1.09280581392332\\
0.700125	-1.128661084461\\
0.701375	-1.1626036073208\\
0.702625	-1.19457543143539\\
0.703875	-1.22452295166944\\
0.705125	-1.25239530941652\\
0.706375	-1.27814549485123\\
0.707625	-1.30172957372238\\
0.708875	-1.32310761366831\\
0.710125	-1.34224376202598\\
0.711375	-1.3591056133341\\
0.712625	-1.37366425512927\\
0.713875	-1.38589501568558\\
0.715125	-1.39577751870502\\
0.716375	-1.40329500306477\\
0.717625	-1.40843435323171\\
0.718875	-1.41118682228648\\
0.720125	-1.41154806966812\\
0.721375	-1.40951725286796\\
0.722625	-1.40509793628287\\
0.723875	-1.39829699523361\\
0.725125	-1.38912642786532\\
0.726375	-1.37760150835949\\
0.727625	-1.36374185529226\\
0.728875	-1.34757029937687\\
0.730125	-1.32911469045214\\
0.731375	-1.30840600263074\\
0.732625	-1.28547938277608\\
0.733875	-1.26037298669567\\
0.735125	-1.23312977699597\\
0.736375	-1.20379558525831\\
0.737625	-1.17242014378754\\
0.738875	-1.13905607783161\\
0.740125	-1.10375974505004\\
0.741375	-1.06659116335263\\
0.742625	-1.02761312834258\\
0.743875	-0.986891105452686\\
0.745125	-0.94449389587894\\
0.746375	-0.900493527871498\\
0.747625	-0.854964348168941\\
0.748875	-0.807982889233447\\
0.750125	-0.759628536210611\\
0.751375	-0.709983390709385\\
0.752625	-0.659131146115948\\
0.753875	-0.607157923423165\\
0.755125	-0.554150945594578\\
0.756375	-0.500200386774997\\
0.757625	-0.445397210210859\\
0.758875	-0.389834220846681\\
0.760125	-0.333604694580062\\
0.761375	-0.276804262007119\\
0.762625	-0.219528696783093\\
0.763875	-0.161874986349564\\
0.765125	-0.103939925602805\\
0.766375	-0.0458220312108982\\
0.767625	0.0123807005213963\\
0.768875	0.0705697201444985\\
0.770125	0.1286471273786\\
0.771375	0.186514798462052\\
0.772625	0.244074534050469\\
0.773875	0.301229084594158\\
0.775125	0.357882317268716\\
0.776375	0.413938517758519\\
0.777625	0.469302554718231\\
0.778875	0.523880894534763\\
0.780125	0.57758178091153\\
0.781375	0.630314495932717\\
0.782625	0.681989527312342\\
0.783875	0.732519815181945\\
0.785125	0.781819788620248\\
0.786375	0.829806856375213\\
0.787625	0.876399304721855\\
0.788875	0.921518686098491\\
0.790125	0.965088581715251\\
0.791375	1.00703611122221\\
0.792625	1.04728977112846\\
0.793875	1.0857818471363\\
0.795125	1.12244713353963\\
0.796375	1.15722445203498\\
0.797625	1.19005442039603\\
0.798875	1.22088189109291\\
0.800125	1.2496546313384\\
0.801375	1.27632459086558\\
0.802625	1.3008468121706\\
0.803875	1.32317949889082\\
0.805125	1.34328502866255\\
0.806375	1.36113003105514\\
0.807625	1.37668448336937\\
0.808875	1.38992175791363\\
0.810125	1.40081962727954\\
0.811375	1.40936031720623\\
0.812625	1.41552955119116\\
0.813875	1.41931657416273\\
0.815125	1.42071541228666\\
0.816375	1.4197236183332\\
0.817625	1.41634380930685\\
0.818875	1.41058113237993\\
0.820125	1.40244578435313\\
0.821375	1.39195145377649\\
0.822625	1.37911685620814\\
0.823875	1.36396313033481\\
0.825125	1.34651637646117\\
0.826375	1.32680604281286\\
0.827625	1.30486646347199\\
0.828875	1.28073417830003\\
0.830125	1.25445049436904\\
0.831375	1.22605982228746\\
0.832625	1.19561093690893\\
0.833875	1.1631555835733\\
0.835125	1.1287483620551\\
0.836375	1.09244770692829\\
0.837625	1.0543158020953\\
0.838875	1.01441740787089\\
0.840125	0.972819724993988\\
0.841375	0.929593372004252\\
0.842625	0.884812280191493\\
0.843875	0.838552478497114\\
0.845125	0.790891941295267\\
0.846375	0.741911848306378\\
0.847625	0.691695054857363\\
0.848875	0.640327652489397\\
0.850125	0.587895993504411\\
0.851375	0.534489360034386\\
0.852625	0.480198119906286\\
0.853875	0.425115298187966\\
0.855125	0.369333547542203\\
0.856375	0.312947857979559\\
0.857625	0.256053660655684\\
0.858875	0.198748432408988\\
0.860125	0.141128609823635\\
0.861375	0.0832923313912701\\
0.862625	0.0253375127750649\\
0.863875	-0.0326368167796029\\
0.865125	-0.0905321716126729\\
0.866375	-0.148250793122609\\
0.867625	-0.205694625699367\\
0.868875	-0.262765451171826\\
0.870125	-0.319366267262064\\
0.871375	-0.375401461304053\\
0.872625	-0.43077578385747\\
0.873875	-0.485394448643414\\
0.875125	-0.539164551938813\\
0.876375	-0.591995222467736\\
0.877625	-0.643796259762746\\
0.878875	-0.69447979023718\\
0.880125	-0.743958568952611\\
0.881375	-0.792149241953916\\
0.882625	-0.838969384648401\\
0.883875	-0.884339487699509\\
0.885125	-0.928181210034444\\
0.886375	-0.970420671157078\\
0.887625	-1.01098543323541\\
0.888875	-1.04980649334662\\
0.890125	-1.08681649528512\\
0.891375	-1.1219530399809\\
0.892625	-1.1551556104527\\
0.893875	-1.18636757521238\\
0.895125	-1.21553466258518\\
0.896375	-1.24260669220878\\
0.897625	-1.26753764107581\\
0.898875	-1.29028444950629\\
0.900125	-1.31080705020832\\
0.901375	-1.32906976696618\\
0.902625	-1.34504136632284\\
0.903875	-1.35869382383244\\
0.905125	-1.37000232140103\\
0.906375	-1.37894663442107\\
0.907625	-1.38551116041646\\
0.908875	-1.38968331538685\\
0.910125	-1.39145515201614\\
0.911375	-1.3908214303563\\
0.912625	-1.3877829234403\\
0.913875	-1.38234311588211\\
0.915125	-1.37451015703044\\
0.916375	-1.36429485098314\\
0.917625	-1.35171400462313\\
0.918875	-1.33678703419634\\
0.920125	-1.31953793962035\\
0.921375	-1.29999322119389\\
0.922625	-1.27818525600119\\
0.923875	-1.25414882159295\\
0.925125	-1.22792308389819\\
0.926375	-1.19954980765336\\
0.927625	-1.16907500129329\\
0.928875	-1.13654885559731\\
0.930125	-1.10202426759904\\
0.931375	-1.06555670481448\\
0.932625	-1.02720555575316\\
0.933875	-0.987033988247533\\
0.935125	-0.945107479526448\\
0.936375	-0.901493626031084\\
0.937625	-0.856263514904783\\
0.938875	-0.809491544240027\\
0.940125	-0.761253552887548\\
0.941375	-0.711628418230098\\
0.942625	-0.66069583642411\\
0.943875	-0.608539703648255\\
0.945125	-0.555244423854157\\
0.946375	-0.500896864256723\\
0.947625	-0.445584079570791\\
0.948875	-0.389396767713601\\
0.950125	-0.332425455133793\\
0.951375	-0.274762528943879\\
0.952625	-0.216499868107496\\
0.953875	-0.157732373149997\\
0.955125	-0.0985540750628069\\
0.956375	-0.0390601840857367\\
0.957625	0.0206549300792949\\
0.958875	0.0804960554407384\\
0.960125	0.140367329008167\\
0.961375	0.20017390184258\\
0.962625	0.259822168252715\\
0.963875	0.319218366868113\\
0.965125	0.378268829745705\\
0.966375	0.436881652322215\\
0.967625	0.494966933124478\\
0.968875	0.552435362733033\\
0.970125	0.609198469587007\\
0.971375	0.665170728625315\\
0.972625	0.720267822356213\\
0.973875	0.774409167981524\\
0.975125	0.827514226521631\\
0.976375	0.879506631513262\\
0.977625	0.930312036757676\\
0.978875	0.979860710285585\\
0.980125	1.02808380100657\\
0.981375	1.07491755875083\\
0.982625	1.1203012196815\\
0.983875	1.16417965340788\\
0.985125	1.2064996766541\\
0.986375	1.24721441891697\\
0.987625	1.28628130301644\\
0.988875	1.32366450046526\\
0.990125	1.35933347694393\\
0.991375	1.39326364168929\\
0.992625	1.42543890425253\\
0.993875	1.45585281717195\\
0.995125	1.48450864220277\\
0.996375	1.51142190891596\\
0.997625	1.53662651119458\\
0.998875	1.56018360938888\\
};
\addplot [color=red,solid,forget plot]
  table[row sep=crcr]{%
0.000125	0.00160652614897256\\
0.001375	0.037912222344971\\
0.002625	0.0778003665035326\\
0.003875	0.119706551050667\\
0.005125	0.163009486256726\\
0.006375	0.20732514635223\\
0.007625	0.252376744158742\\
0.008875	0.297946037230981\\
0.010125	0.343851868750971\\
0.011375	0.389937001102584\\
0.012625	0.436062218044196\\
0.013875	0.482100969642862\\
0.015125	0.527937295219927\\
0.016375	0.573462770713219\\
0.017625	0.61857592056857\\
0.018875	0.663180687943721\\
0.020125	0.707186155384191\\
0.021375	0.75050550761336\\
0.022625	0.79305564352524\\
0.023875	0.834757134136653\\
0.025125	0.875534008099275\\
0.026375	0.915313242193612\\
0.027625	0.954024641182314\\
0.028875	0.991601015394566\\
0.030125	1.0279781205239\\
0.031375	1.0630943319614\\
0.032625	1.09689050454856\\
0.033875	1.1293104719894\\
0.035125	1.16030060546619\\
0.036375	1.18981012121169\\
0.037625	1.2177905141222\\
0.038875	1.24419633340947\\
0.040125	1.26898449745593\\
0.041375	1.29211509124203\\
0.042625	1.31355030103465\\
0.043875	1.33325550064152\\
0.045125	1.35119850458334\\
0.046375	1.36735023974501\\
0.047625	1.38168403917848\\
0.048875	1.3941762128022\\
0.050125	1.40480572419732\\
0.051375	1.41355460584561\\
0.052625	1.42040752471604\\
0.053875	1.42535186507764\\
0.055125	1.42837795067012\\
0.056375	1.429479038124\\
0.057625	1.42865109583264\\
0.058875	1.4258927853299\\
0.060125	1.4212057376769\\
0.061375	1.41459433678475\\
0.062625	1.40606583528784\\
0.063875	1.39563001871845\\
0.065125	1.38329959976954\\
0.066375	1.36908996019856\\
0.067625	1.35301924927808\\
0.068875	1.33510809813737\\
0.070125	1.31537989026389\\
0.071375	1.29386057928136\\
0.072625	1.27057871758859\\
0.073875	1.24556525929953\\
0.075125	1.21885364973767\\
0.076375	1.19047975142503\\
0.077625	1.16048185765336\\
0.078875	1.12890044870609\\
0.080125	1.09577826941223\\
0.081375	1.06116027433557\\
0.082625	1.0250935521858\\
0.083875	0.987627187016252\\
0.085125	0.948812229490353\\
0.086375	0.908701702109568\\
0.087625	0.867350389335648\\
0.088875	0.824814861968333\\
0.090125	0.78115337185343\\
0.091375	0.73642569085986\\
0.092625	0.690693173878974\\
0.093875	0.644018545584517\\
0.095125	0.596465915464355\\
0.096375	0.548100530577799\\
0.097625	0.498988893463088\\
0.098875	0.449198477788858\\
0.100125	0.398797817762823\\
0.101375	0.347856173734487\\
0.102625	0.296443704388467\\
0.103875	0.244631134927961\\
0.105125	0.192489861367662\\
0.106375	0.140091627933485\\
0.107625	0.0875087219794924\\
0.108875	0.0348134900408264\\
0.110125	-0.0179213666258567\\
0.111375	-0.0706232303107235\\
0.112625	-0.123219330713663\\
0.113875	-0.175637258719237\\
0.115125	-0.227804644772557\\
0.116375	-0.279649612951203\\
0.117625	-0.331100541997407\\
0.118875	-0.382086512407947\\
0.120125	-0.432537226024345\\
0.121375	-0.482382934576485\\
0.122625	-0.531554974935893\\
0.123875	-0.579985405242883\\
0.125125	-0.627607451707811\\
0.126375	-0.674355298209185\\
0.127625	-0.720164500983346\\
0.128875	-0.76497178189639\\
0.130125	-0.80871524694416\\
0.131375	-0.851334537215444\\
0.132625	-0.892770860586493\\
0.133875	-0.932966962376676\\
0.135125	-0.971867318387345\\
0.136375	-1.009418243045\\
0.137625	-1.0455679407746\\
0.138875	-1.08026646015184\\
0.140125	-1.11346586159879\\
0.141375	-1.14512034067989\\
0.142625	-1.17518621894639\\
0.143875	-1.20362198491438\\
0.145125	-1.23038836379261\\
0.146375	-1.25544838933465\\
0.147625	-1.27876752945411\\
0.148875	-1.30031359321838\\
0.150125	-1.32005682470433\\
0.151375	-1.33796997787456\\
0.152625	-1.35402839269622\\
0.153875	-1.3682099184307\\
0.155125	-1.38049498463194\\
0.156375	-1.39086667218797\\
0.157625	-1.39931070215507\\
0.158875	-1.40581550928217\\
0.160125	-1.4103721349427\\
0.161375	-1.41297430716992\\
0.162625	-1.4136185345697\\
0.163875	-1.41230402925666\\
0.165125	-1.40903264338645\\
0.166375	-1.40380894021769\\
0.167625	-1.39664025254961\\
0.168875	-1.38753660084842\\
0.170125	-1.37651061247278\\
0.171375	-1.36357758085872\\
0.172625	-1.34875550965094\\
0.173875	-1.33206497979595\\
0.175125	-1.31352915948743\\
0.176375	-1.29317372430652\\
0.177625	-1.27102691474934\\
0.178875	-1.24711947153898\\
0.180125	-1.2214845432634\\
0.181375	-1.19415763093233\\
0.182625	-1.16517662265659\\
0.183875	-1.13458170940422\\
0.185125	-1.10241529757175\\
0.186375	-1.06872192656538\\
0.187625	-1.03354830237791\\
0.188875	-0.996943148514546\\
0.190125	-0.958957229135093\\
0.191375	-0.919643139331547\\
0.192625	-0.879055313714033\\
0.193875	-0.837250016872706\\
0.195125	-0.794285188099539\\
0.196375	-0.750220312544364\\
0.197625	-0.705116403752036\\
0.198875	-0.659035971360234\\
0.200125	-0.612042868508849\\
0.201375	-0.56420215905479\\
0.202625	-0.515580076118715\\
0.203875	-0.466244003569601\\
0.205125	-0.416262276109817\\
0.206375	-0.365704128256114\\
0.207625	-0.314639547607406\\
0.208875	-0.263139267549394\\
0.210125	-0.211274635210043\\
0.211375	-0.159117485418789\\
0.212625	-0.106740020391813\\
0.213875	-0.0542147894943368\\
0.215125	-0.00161455397753847\\
0.216375	0.0509878360606418\\
0.217625	0.103519562927326\\
0.218875	0.155907856341083\\
0.220125	0.20808017679462\\
0.221375	0.259964239036339\\
0.222625	0.311488217207977\\
0.223875	0.362580779395446\\
0.225125	0.413171137583297\\
0.226375	0.46318921914623\\
0.227625	0.512565780972927\\
0.228875	0.56123246221767\\
0.230125	0.609121841607867\\
0.231375	0.656167582850301\\
0.232625	0.702304560563623\\
0.233875	0.747468903713942\\
0.235125	0.791598035303295\\
0.236375	0.834630844509346\\
0.237625	0.876507735336831\\
0.238875	0.917170739425141\\
0.240125	0.956563537333193\\
0.241375	0.99463156368666\\
0.242625	1.03132209851081\\
0.243875	1.06658435143519\\
0.245125	1.10036948584033\\
0.246375	1.13263069899677\\
0.247625	1.16332331212927\\
0.248875	1.19240484212613\\
0.250125	1.21983500150344\\
0.251375	1.24557581234115\\
0.252625	1.26959159700766\\
0.253875	1.29184913902698\\
0.255125	1.31231756836255\\
0.256375	1.33096857435034\\
0.257625	1.34777628930855\\
0.258875	1.36271747508258\\
0.260125	1.37577139489143\\
0.261375	1.38692000634295\\
0.262625	1.39614782394698\\
0.263875	1.40344210831974\\
0.265125	1.40879270589619\\
0.266375	1.41219223728651\\
0.267625	1.41363597017448\\
0.268875	1.41312190638365\\
0.270125	1.41065074617554\\
0.271375	1.4062259055749\\
0.272625	1.39985350026205\\
0.273875	1.39154234936644\\
0.275125	1.38130395185536\\
0.276375	1.36915247742759\\
0.277625	1.35510474006442\\
0.278875	1.33918018550263\\
0.280125	1.32140085531409\\
0.281375	1.3017913590703\\
0.282625	1.2803788526662\\
0.283875	1.25719293852831\\
0.285125	1.23226576123048\\
0.286375	1.20563177867531\\
0.287625	1.1773279093112\\
0.288875	1.14739328617228\\
0.290125	1.11586939551984\\
0.291375	1.08279983151315\\
0.292625	1.0482304243013\\
0.293875	1.01220897904832\\
0.295125	0.97478540975154\\
0.296375	0.936011480491815\\
0.297625	0.895940903767286\\
0.298875	0.854629143009818\\
0.300125	0.812133387597634\\
0.301375	0.768512468958264\\
0.302625	0.723826780635458\\
0.303875	0.678138183330588\\
0.305125	0.631509927621703\\
0.306375	0.584006567705532\\
0.307625	0.535693870017004\\
0.308875	0.486638719711092\\
0.310125	0.436909028641521\\
0.311375	0.386573643347562\\
0.312625	0.3357022530489\\
0.313875	0.284365299473248\\
0.315125	0.232633814220308\\
0.316375	0.180579466748313\\
0.317625	0.128274271069039\\
0.318875	0.0757906997463957\\
0.320125	0.0232013548843903\\
0.321375	-0.029420904512226\\
0.322625	-0.0820032791268049\\
0.323875	-0.134472910241055\\
0.325125	-0.186757218126308\\
0.326375	-0.238783762590863\\
0.327625	-0.290480571956254\\
0.328875	-0.341776029323828\\
0.330125	-0.392599131102743\\
0.331375	-0.442879521199588\\
0.332625	-0.492547591057662\\
0.333875	-0.541534570023575\\
0.335125	-0.589772639714333\\
0.336375	-0.637195022049754\\
0.337625	-0.683736063654388\\
0.338875	-0.729331326580964\\
0.340125	-0.773917686714709\\
0.341375	-0.817433420510173\\
0.342625	-0.859818285810837\\
0.343875	-0.901013597053799\\
0.345125	-0.940962307414585\\
0.346375	-0.979609159180997\\
0.347625	-1.01690058829977\\
0.348875	-1.05278503024847\\
0.350125	-1.08721273620409\\
0.351375	-1.12013611354514\\
0.352625	-1.15150951796197\\
0.353875	-1.18128958275443\\
0.355125	-1.20943500724782\\
0.356375	-1.2359068997682\\
0.357625	-1.26066854408216\\
0.358875	-1.28368573203478\\
0.360125	-1.3049265434464\\
0.361375	-1.32436158945552\\
0.362625	-1.34196397065222\\
0.363875	-1.35770932003278\\
0.365125	-1.37157583664352\\
0.366375	-1.38354432626409\\
0.367625	-1.39359822542839\\
0.368875	-1.40172361659985\\
0.370125	-1.40790925325198\\
0.371375	-1.41214657243022\\
0.372625	-1.41442971488675\\
0.373875	-1.41475552362714\\
0.375125	-1.4131235478892\\
0.376375	-1.40953603140422\\
0.377625	-1.40399800608641\\
0.378875	-1.39651706337858\\
0.380125	-1.38710364487334\\
0.381375	-1.37577070518978\\
0.382625	-1.36253402133435\\
0.383875	-1.34741184118776\\
0.385125	-1.33042518813378\\
0.386375	-1.31159750051591\\
0.387625	-1.29095493509168\\
0.388875	-1.26852598824914\\
0.390125	-1.24434180527761\\
0.391375	-1.2184358080154\\
0.392625	-1.19084389045643\\
0.393875	-1.16160425672785\\
0.395125	-1.1307573921349\\
0.396375	-1.09834601177832\\
0.397625	-1.0644149923965\\
0.398875	-1.02901131694634\\
0.400125	-0.992184003815887\\
0.401375	-0.953984049040359\\
0.402625	-0.914464344623134\\
0.403875	-0.873679609194656\\
0.405125	-0.831686312116535\\
0.406375	-0.788542601701517\\
0.407625	-0.744308190833771\\
0.408875	-0.699044398294515\\
0.410125	-0.652813799943268\\
0.411375	-0.60568050366292\\
0.412625	-0.557709677970522\\
0.413875	-0.508967837174817\\
0.415125	-0.459522372335769\\
0.416375	-0.409441841790081\\
0.417625	-0.358795494395231\\
0.418875	-0.30765355086151\\
0.420125	-0.256086723943904\\
0.421375	-0.204166512432995\\
0.422625	-0.151964731974775\\
0.423875	-0.0995536877365578\\
0.425125	-0.0470059452532026\\
0.426375	0.0056057382425122\\
0.427625	0.0582085122972552\\
0.428875	0.110729545504463\\
0.430125	0.163096115598447\\
0.431375	0.215235716959149\\
0.432625	0.26707615350586\\
0.433875	0.31854564687752\\
0.435125	0.369572933927001\\
0.436375	0.420087361863295\\
0.437625	0.470018984870865\\
0.438875	0.519298694933796\\
0.440125	0.567858177823057\\
0.441375	0.615630305646175\\
0.442625	0.66254882242183\\
0.443875	0.70854887051571\\
0.445125	0.753566650459272\\
0.446375	0.79753993819672\\
0.447625	0.840407737833085\\
0.448875	0.882110800265116\\
0.450125	0.922591272675374\\
0.451375	0.961793215377169\\
0.452625	0.999662234554259\\
0.453875	1.03614597807732\\
0.455125	1.07119389853674\\
0.456375	1.10475747018974\\
0.457625	1.13679021700308\\
0.458875	1.1672477814223\\
0.460125	1.19608798962477\\
0.461375	1.22327090699308\\
0.462625	1.24875889326329\\
0.463875	1.27251665166904\\
0.465125	1.29451128310191\\
0.466375	1.31471233312202\\
0.467625	1.33309182662937\\
0.468875	1.3496243086776\\
0.470125	1.3642869160761\\
0.471375	1.37705925005162\\
0.472625	1.38792374447296\\
0.473875	1.39686522708186\\
0.475125	1.40387143482668\\
0.476375	1.40893253929236\\
0.477625	1.41204165095926\\
0.478875	1.41319433490671\\
0.480125	1.41238911311065\\
0.481375	1.40962696892105\\
0.482625	1.40491184632301\\
0.483875	1.3982501397098\\
0.485125	1.38965116190522\\
0.486375	1.37912677955653\\
0.487625	1.36669155860813\\
0.488875	1.35236270783293\\
0.490125	1.3361600605584\\
0.491375	1.31810604208922\\
0.492625	1.29822564186301\\
0.493875	1.27654637784559\\
0.495125	1.25309825880471\\
0.496375	1.2279137412553\\
0.497625	1.20102768847168\\
0.498875	1.17247731562738\\
0.500125	1.14230218124849\\
0.501375	1.11054398567734\\
0.502625	1.07724672575697\\
0.503875	1.04245642456083\\
0.505125	1.00622145509898\\
0.506375	0.968591765975696\\
0.507625	0.929619518704324\\
0.508875	0.889358591641642\\
0.510125	0.847864936290379\\
0.511375	0.805195777217252\\
0.512625	0.76141025497287\\
0.513875	0.716568909368199\\
0.515125	0.670734038160281\\
0.516375	0.623968871805332\\
0.517625	0.576338222684667\\
0.518875	0.5279079531579\\
0.520125	0.478745257584436\\
0.521375	0.428918197143603\\
0.522625	0.378495596605454\\
0.523875	0.327547258096266\\
0.525125	0.276143864677605\\
0.526375	0.224356580373156\\
0.527625	0.172256940119604\\
0.528875	0.119917068034765\\
0.530125	0.0674095758744715\\
0.531375	0.0148071560977426\\
0.532625	-0.0378175305556779\\
0.533875	-0.0903915545777833\\
0.535125	-0.142842215940218\\
0.536375	-0.195096662106428\\
0.537625	-0.247082797163573\\
0.538875	-0.298728574452143\\
0.540125	-0.349962581165528\\
0.541375	-0.40071364492969\\
0.542625	-0.450911756374179\\
0.543875	-0.500487343557372\\
0.545125	-0.549371860670443\\
0.546375	-0.597497385669934\\
0.547625	-0.644797547695115\\
0.548875	-0.691206786292975\\
0.550125	-0.736660942602\\
0.551375	-0.781096924425544\\
0.552625	-0.824453216654157\\
0.553875	-0.866669976197401\\
0.555125	-0.907688764839622\\
0.556375	-0.947452628032963\\
0.557625	-0.985906520179564\\
0.558875	-1.02299738898804\\
0.560125	-1.05867389348894\\
0.561375	-1.09288647258622\\
0.562625	-1.12558776805687\\
0.563875	-1.15673269503273\\
0.565125	-1.18627805477215\\
0.566375	-1.21418304470926\\
0.567625	-1.2404087663582\\
0.568875	-1.26491919671965\\
0.570125	-1.28768032031099\\
0.571375	-1.30866073029824\\
0.572625	-1.32783111065779\\
0.573875	-1.34516521147445\\
0.575125	-1.3606389534504\\
0.576375	-1.37423102014917\\
0.577625	-1.38592232036896\\
0.578875	-1.39569696560636\\
0.580125	-1.40354134071006\\
0.581375	-1.40944469500727\\
0.582625	-1.41339867529699\\
0.583875	-1.41539781837261\\
0.585125	-1.41543956623473\\
0.586375	-1.4135238761176\\
0.587625	-1.40965321658619\\
0.588875	-1.40383295888705\\
0.590125	-1.39607137538626\\
0.591375	-1.3863792281399\\
0.592625	-1.3747697482112\\
0.593875	-1.36125902255255\\
0.595125	-1.34586597592916\\
0.596375	-1.328611833425\\
0.597625	-1.30952060234737\\
0.598875	-1.28861841861881\\
0.600125	-1.26593454898801\\
0.601375	-1.24150031161996\\
0.602625	-1.21534965761487\\
0.603875	-1.18751849016823\\
0.605125	-1.15804567744919\\
0.606375	-1.12697193549147\\
0.607625	-1.09434041058262\\
0.608875	-1.0601959724147\\
0.610125	-1.02458623827999\\
0.611375	-0.987560418727911\\
0.612625	-0.949169904231527\\
0.613875	-0.909467642881279\\
0.615125	-0.868508620763326\\
0.616375	-0.826349786514527\\
0.617625	-0.783049529780347\\
0.618875	-0.738667592426652\\
0.620125	-0.693265439610433\\
0.621375	-0.646906178204439\\
0.622625	-0.599654012903794\\
0.623875	-0.551574154566471\\
0.625125	-0.502733188041811\\
0.626375	-0.453198986510985\\
0.627625	-0.403040034400633\\
0.628875	-0.352325916435414\\
0.630125	-0.301126512445767\\
0.631375	-0.249513083285801\\
0.632625	-0.197556990887638\\
0.633875	-0.145330312780823\\
0.635125	-0.092905014453232\\
0.636375	-0.0403540641212345\\
0.637625	0.0122498802176514\\
0.638875	0.0648338409744873\\
0.640125	0.117325364484632\\
0.641375	0.169651382115485\\
0.642625	0.221739549045776\\
0.643875	0.273517598832721\\
0.645125	0.324914072078459\\
0.646375	0.37585778800216\\
0.647625	0.426277937632146\\
0.648875	0.476104687577181\\
0.650125	0.525269282573487\\
0.651375	0.573703631138004\\
0.652625	0.621340393412292\\
0.653875	0.668113588607164\\
0.655125	0.713958691256207\\
0.656375	0.758812206099299\\
0.657625	0.802611744087127\\
0.658875	0.845296770183063\\
0.660125	0.886808023268726\\
0.661375	0.927088404819744\\
0.662625	0.966081720322077\\
0.663875	1.00373408864989\\
0.665125	1.03999321367027\\
0.666375	1.07480927387789\\
0.667625	1.10813363178714\\
0.668875	1.13992025601753\\
0.670125	1.17012496877567\\
0.671375	1.19870633583837\\
0.672625	1.22562434291365\\
0.673875	1.25084183081453\\
0.675125	1.27432371176325\\
0.676375	1.29603772439565\\
0.677625	1.3159537765451\\
0.678875	1.33404398129333\\
0.680125	1.35028325681931\\
0.681375	1.36464937109394\\
0.682625	1.37712240409772\\
0.683875	1.38768476801507\\
0.685125	1.39632180088259\\
0.686375	1.40302180070267\\
0.687625	1.40777545823678\\
0.688875	1.41057586664692\\
0.690125	1.41141925830107\\
0.691375	1.41030428054394\\
0.692625	1.40723288462371\\
0.693875	1.40220884398128\\
0.695125	1.39523922474167\\
0.696375	1.38633348733379\\
0.697625	1.37550437447127\\
0.698875	1.36276639054416\\
0.700125	1.34813728237138\\
0.701375	1.33163711175315\\
0.702625	1.31328914044252\\
0.703875	1.29311827318644\\
0.705125	1.27115254877093\\
0.706375	1.24742218115341\\
0.707625	1.22196029212052\\
0.708875	1.19480209350552\\
0.710125	1.16598483148687\\
0.711375	1.13554835656778\\
0.712625	1.103535077821\\
0.713875	1.06998927625878\\
0.715125	1.03495703214866\\
0.716375	0.998486800046827\\
0.717625	0.960629344722252\\
0.718875	0.921437036683366\\
0.720125	0.880963764836262\\
0.721375	0.839265677266469\\
0.722625	0.796400293039088\\
0.723875	0.752427398603568\\
0.725125	0.707407333408044\\
0.726375	0.661402538396211\\
0.727625	0.614476491596331\\
0.728875	0.566694604582026\\
0.730125	0.518122481775436\\
0.731375	0.468827478546949\\
0.732625	0.418877625431832\\
0.733875	0.368342521844512\\
0.735125	0.317291575619026\\
0.736375	0.265795572443805\\
0.737625	0.213925584725496\\
0.738875	0.161753702551095\\
0.740125	0.109352101781778\\
0.741375	0.0567929260031058\\
0.742625	0.00414888015619266\\
0.743875	-0.0485068683576944\\
0.745125	-0.101101479545178\\
0.746375	-0.153562551401598\\
0.747625	-0.205817519759634\\
0.748875	-0.257793752678806\\
0.750125	-0.309419357803878\\
0.751375	-0.360623280283023\\
0.752625	-0.411334527975775\\
0.753875	-0.461483133873139\\
0.755125	-0.510999210210869\\
0.756375	-0.55981481122217\\
0.757625	-0.607862250725897\\
0.758875	-0.655075260216822\\
0.760125	-0.701388010296622\\
0.761375	-0.746737004259376\\
0.762625	-0.791059357946699\\
0.763875	-0.834293969978288\\
0.765125	-0.876380515511229\\
0.766375	-0.917261360756337\\
0.767625	-0.956879812354616\\
0.768875	-0.995181288609635\\
0.770125	-1.03211248073767\\
0.771375	-1.06762234553266\\
0.772625	-1.10166218885229\\
0.773875	-1.13418497654203\\
0.775125	-1.16514539842364\\
0.776375	-1.19450068758765\\
0.777625	-1.22221068011677\\
0.778875	-1.24823711599748\\
0.780125	-1.27254367566217\\
0.781375	-1.29509681491202\\
0.782625	-1.31586579994182\\
0.783875	-1.33482179190854\\
0.785125	-1.35193884993092\\
0.786375	-1.36719280209893\\
0.787625	-1.38056323347372\\
0.788875	-1.39203154908236\\
0.790125	-1.40158218941439\\
0.791375	-1.4092014563536\\
0.792625	-1.41487952374894\\
0.793875	-1.41860845276459\\
0.795125	-1.42038341108548\\
0.796375	-1.42020146322999\\
0.797625	-1.41806360451919\\
0.798875	-1.41397272535064\\
0.800125	-1.40793483046434\\
0.801375	-1.39995801073859\\
0.802625	-1.3900534518654\\
0.803875	-1.37823542496361\\
0.805125	-1.36452044124629\\
0.806375	-1.34892722339694\\
0.807625	-1.33147751439552\\
0.808875	-1.31219605012404\\
0.810125	-1.29110968824799\\
0.811375	-1.26824735813602\\
0.812625	-1.24364088161142\\
0.813875	-1.21732492528247\\
0.815125	-1.18933587714379\\
0.816375	-1.15971288706799\\
0.817625	-1.12849649445695\\
0.818875	-1.09573075762618\\
0.820125	-1.06146100459689\\
0.821375	-1.02573509923936\\
0.822625	-0.988602033942245\\
0.823875	-0.950114086583994\\
0.825125	-0.910324519542593\\
0.826375	-0.869288858709403\\
0.827625	-0.827063451835012\\
0.828875	-0.783707658050536\\
0.830125	-0.739281491972801\\
0.831375	-0.693846919567542\\
0.832625	-0.64746662019692\\
0.833875	-0.60020503747394\\
0.835125	-0.552128306673989\\
0.836375	-0.503303240326216\\
0.837625	-0.453797223706223\\
0.838875	-0.403679049564722\\
0.840125	-0.353018837775057\\
0.841375	-0.301887000218144\\
0.842625	-0.25035412881944\\
0.843875	-0.198491847180512\\
0.845125	-0.146372723751607\\
0.846375	-0.0940689692069052\\
0.847625	-0.0416535466327356\\
0.848875	0.0108014015608409\\
0.850125	0.0632222095992991\\
0.851375	0.115536224432652\\
0.852625	0.167670437597055\\
0.853875	0.219553076983403\\
0.855125	0.271111226286619\\
0.856375	0.322273402897622\\
0.857625	0.372968153991714\\
0.858875	0.423125680926432\\
0.860125	0.472675409838629\\
0.861375	0.521548608606682\\
0.862625	0.569676950859943\\
0.863875	0.616993909830082\\
0.865125	0.663433567564402\\
0.866375	0.708930688306205\\
0.867625	0.753421834793151\\
0.868875	0.796845481528554\\
0.870125	0.83914105561718\\
0.871375	0.880249010215722\\
0.872625	0.920111932109664\\
0.873875	0.958674671233139\\
0.875125	0.995883331824449\\
0.876375	1.03168535089109\\
0.877625	1.06603088263045\\
0.878875	1.09887155687139\\
0.880125	1.13016215585261\\
0.881375	1.15985797229197\\
0.882625	1.18791756163683\\
0.883875	1.21430118973563\\
0.885125	1.23897252504956\\
0.886375	1.26189594839379\\
0.887625	1.28303932868546\\
0.888875	1.30237243135553\\
0.890125	1.3198686241914\\
0.891375	1.33550213706583\\
0.892625	1.34925085854889\\
0.893875	1.36109470301721\\
0.895125	1.3710170556697\\
0.896375	1.37900337267772\\
0.897625	1.38504120720954\\
0.898875	1.38912133184717\\
0.900125	1.39123778703924\\
0.901375	1.39138673849183\\
0.902625	1.38956648004957\\
0.903875	1.38577855120943\\
0.905125	1.38002778298138\\
0.906375	1.37232111464043\\
0.907625	1.36266758831325\\
0.908875	1.3510797517859\\
0.910125	1.33757224077899\\
0.911375	1.32216346637625\\
0.912625	1.30487269544867\\
0.913875	1.2857229082181\\
0.915125	1.2647390456149\\
0.916375	1.24194972413522\\
0.917625	1.21738424846341\\
0.918875	1.19107550414766\\
0.920125	1.16305815461959\\
0.921375	1.13337037362167\\
0.922625	1.10205080918669\\
0.923875	1.06914149662749\\
0.925125	1.0346860188577\\
0.926375	0.998730943301429\\
0.927625	0.961324255816264\\
0.928875	0.922515272707718\\
0.930125	0.882355776631916\\
0.931375	0.840899979769616\\
0.932625	0.79820319103653\\
0.933875	0.754321756560059\\
0.935125	0.709314163667862\\
0.936375	0.663241025833303\\
0.937625	0.616163693071239\\
0.938875	0.56814421133076\\
0.940125	0.519246734669142\\
0.941375	0.469535935086747\\
0.942625	0.419078717437684\\
0.943875	0.367941046810735\\
0.945125	0.316190933439917\\
0.946375	0.263896503279847\\
0.947625	0.211127742387695\\
0.948875	0.157953248462303\\
0.950125	0.104443293185523\\
0.951375	0.0506678146505429\\
0.952625	-0.00330177717718947\\
0.953875	-0.0573959180013559\\
0.955125	-0.111543784667125\\
0.956375	-0.165675324526847\\
0.957625	-0.219719765027325\\
0.958875	-0.273607348159248\\
0.960125	-0.327269426298787\\
0.961375	-0.380637298480771\\
0.962625	-0.433642277337544\\
0.963875	-0.486217180073098\\
0.965125	-0.538296385888241\\
0.966375	-0.589814679401185\\
0.967625	-0.640707312358644\\
0.968875	-0.690911522032882\\
0.970125	-0.74036658513734\\
0.971375	-0.78901231694725\\
0.972625	-0.836790833841901\\
0.973875	-0.883644702697167\\
0.975125	-0.929520408104572\\
0.976375	-0.974365126840777\\
0.977625	-1.01812888579405\\
0.978875	-1.06076268591879\\
0.980125	-1.10222205649883\\
0.981375	-1.14246380507145\\
0.982625	-1.18144823334453\\
0.983875	-1.21913730142358\\
0.985125	-1.25549827545165\\
0.986375	-1.2905005151448\\
0.987625	-1.32411783064289\\
0.988875	-1.3563271054278\\
0.990125	-1.38711043518924\\
0.991375	-1.41645564743521\\
0.992625	-1.44435555287457\\
0.993875	-1.47080887479383\\
0.995125	-1.49582295394386\\
0.996375	-1.51941574979514\\
0.997625	-1.54161782769075\\
0.998875	-1.56247931893952\\
};
\addplot [color=black,only marks,mark=o,mark options={solid},forget plot]
  table[row sep=crcr]{%
1	1.58113883008419\\
};
\addplot [color=black,only marks,mark=o,mark options={solid},forget plot]
  table[row sep=crcr]{%
1	-1.58113883008419\\
};
\addplot [color=black,only marks,mark=o,mark options={solid},forget plot]
  table[row sep=crcr]{%
0.000125	0\\
};
\node[right, align=left, text=black]
at (axis cs:0.85,1.6) {$\sqrt{2H+1}$};
\node[right, align=left, text=black]
at (axis cs:0.85,-1.6) {$-\sqrt{2H+1}$};
\end{axis}
\end{tikzpicture}%

%% file: contour.tex
\begin{tikzpicture}
\draw[help lines,->] (-3,0) -- (3,0) coordinate (xaxis);
\draw[help lines,->] (0,-3) -- (0,3) coordinate (yaxis);


\path[draw,line width=0.8pt,
decoration={markings,
mark=at position 0.75cm with {\arrow[line width=1pt]{<}},
mark=at position 3.5cm with {\arrow[line width=1pt]{<}},
mark=at position 5.7cm with {\arrow[line width=1pt]{<}}
},
postaction=decorate] (0,0.1) -- (2,0.1) arc (0:90:2) -- (0,1.2) arc (90:-90:0.1) -- (0,0.1);

\path[draw,line width=0.8pt,
decoration={markings,
mark=at position 0.75cm with {\arrow[line width=1pt]{>}},
mark=at position 3.5cm with {\arrow[line width=1pt]{>}},
mark=at position 5.7cm with {\arrow[line width=1pt]{>}}
},
postaction=decorate] (0,-0.1) -- (2,-0.1) arc (0:-90:2) -- (0,-1.0) arc (90:-90:0.1) -- (0,-0.1);

\path[draw,line width=0.8pt,
decoration={markings,
mark=at position 0.75cm with {\arrow[line width=1pt]{<}},
mark=at position 3.5cm with {\arrow[line width=1pt]{<}},
mark=at position 5.7cm with {\arrow[line width=1pt]{<}}
},
postaction=decorate] (0,0.1) -- (-2,0.1) arc (180:90:2) ; 

\path[draw,line width=0.8pt,
decoration={markings,
mark=at position 0.75cm with {\arrow[line width=1pt]{>}},
mark=at position 3.5cm with {\arrow[line width=1pt]{>}},
mark=at position 5.7cm with {\arrow[line width=1pt]{>}}
},
postaction=decorate] (0,-0.1) -- (-2,-0.1) arc (180:275:2) ; 

\draw (0,1.1) circle (1pt);
\draw (0,-1.1) circle (1pt);

\node[below] at (xaxis) {$\mathrm{Re}(z)$};
\node[left] at (yaxis) {$\mathrm{Im}(z)$};
\node at (-0.3, 1.1) {$z_0$};
\node at (-0.45, -1.1) {$-z_0$};
\node at (1.8,1.8) {$C_R^+$};
\node at (1.8,-1.8) {$C_R^-$};
\node at (-1.8,1.8) {${^+}C_R$};
\node at (-1.8,-1.8) {${^-}C_R$};
\node at (1.0,0.3) {$\Sigma_+$};
\node at (1.0,-0.35) {$\Sigma_-$};
\node at (-1.0,0.3) {${_+}\Sigma$};
\node at (-1.0,-0.35) {${_-}\Sigma$};
\node at (0.45, 1.5) {$L^+_{\delta,R}$};
\node at (0.45, -1.5) {$L^-_{\delta,R}$};
\end{tikzpicture}